\documentclass[12pt]{article}
\usepackage[utf8]{inputenc}

\usepackage{amssymb,   enumitem,   xcolor}
\usepackage{tikz}
\usepackage[normalem]{ulem}

\usepackage{graphics,   graphicx,   amsmath}
\usepackage{lmodern,   microtype}

\usepackage{mathtools,   amsthm,   icomma,   upgreek,   xfrac,   amsmath,   dsfont,   float}
\mathtoolsset{showonlyrefs,   mathic}

\usepackage[colorlinks=true,    urlcolor=blue,    linkcolor=blue,    citecolor=blue]{hyperref}

\allowdisplaybreaks

\theoremstyle{plain}
\newtheorem{theorem}{Theorem}[section]
\newtheorem{lemma}[theorem]{Lemma}

\theoremstyle{definition}
\newtheorem{definition}[theorem]{Definition}

\newtheorem{corollary}{Corollary}[section]

\newtheorem{remark}[theorem]{Remark}
\newtheorem{fact}{Fact}[section]

\numberwithin{equation}{section}

\newcommand{\E}{\mathbf{E}}
\newcommand{\X}{X^{(r)}_n}
\renewcommand{\P}{\mathbf{P}}
\newcommand{\R}{\mathbf{R}}
\newcommand{\N}{\mathbf{N}}
\newcommand{\FF}{{\cal E}}

\DeclareMathOperator{\Var}{Var}

\def\ind{{\mathbb I}}

\newcommand{\cov}{\mathrm{Cov}}
\renewcommand{\P}{\mathbb{P}}
\newcommand{\pr}{\mathbb{P}}
\newcommand{\cP}{{\mathcal P}}

\newcommand{\V}{\mathcal{V}}

\newcommand{\A}{\mathcal{A}}

\newcommand{\Subgraphs}{{\cal K}}

\newcommand{\covers}{{\mathcal C}}
\newcommand{\cover}{{\mathbb C}}

\newcommand{\klik}{K}
\newcommand{\kli}{\mathbf{K}}
\newcommand{\emp}{N}

\newcommand{\indeksy}{{\bf\Gamma}}

\newcommand{\Pra}[1]{\pr\left(#1\right)}
\newcommand{\Gnmp}{{\mathcal G}(n,    m,    p)}

\newcommand{\Gnp}{{\mathcal G}(n,  {\bf p})}

\newcommand{\hatp}{\mathbf{p}}

\newcommand{\ble}{\Psi}

\newcommand{\Pone}{{\bf P1}}
\newcommand{\Ptwo}{{\bf P2}}

\newcommand{\II}{{\bf I}}
\newcommand{\JJ}{{\bf J}}
\newcommand{\LL}{{\bf L}}

\newcommand{\qq}{{\bf P}}
\newcommand{\pp}{{\bf p}}
\newcommand{\calP}{{\cal P}}

\allowdisplaybreaks

\date{\vspace{-5ex}}

\title{
\huge
Normal approximation for subgraph counts\\ in random intersection graphs: the case of cliques
} 

\author{ 
Katarzyna Rybarczyk\thanks{Faculty of Mathematics and Computer Science,      Adam Mickiewicz University,      Pozna\'n,      Poland.
e-mail: {\tt kryba@amu.edu.pl}. ORCID: 0000-0002-3998-6482}
\and
Grzegorz Serafin\thanks{Faculty of Pure and Applied Mathematics,       Wroc{\l}aw University of Science and Technology,      Ul. Wybrze\.ze Wyspia\'nskiego 27,       Wroc{\l}aw,      Poland.
e-mail: {\tt grzegorz.serafin@pwr.edu.pl}. ORCID: 0000-0003-3170-6850}}

\begin{document}
\maketitle

\begin{abstract} 
Let $X^{(r)}_n$ denote the number of copies of a complete graph $K_r$ in the random intersection graph model on $n$ vertices. We derive upper bounds on the distance between the standardized version of $X^{(r)}_n$ and the standard normal distribution,  measured in both the Kolmogorov and Wasserstein metrics. We complement these results with necessary and sufficient conditions for asymptotic normality. As a byproduct,  we provide precise estimates for the variance of $X^{(r)}_n$. In the process,  we develop new general techniques for analysing so--called clique covers,  which play a crucial role in the study of subgraph probabilities in random intersection graphs.
\end{abstract}
\noindent\emph{MSC2020 subject classification}: Primary 05C80,    60F05; Secondary 05C82.

\section{Introduction}

The random intersection graph model,  denoted by $\Gnmp$,  was introduced in \cite{1999KarSchSin}. In some respects it resembles the classical $\Gnp$  model,  in which edges appear independently with probability $\pp$,  chosen to equal the edge probability in $\Gnmp$. The most interesting features,  however,  are precisely those that distinguish the two models. In random intersection graphs,  in contrast to $\Gnp$,  edges are not generated independently,  which in particular affects the probabilities of small subgraphs. Consequently,  the seminal article on random intersection graphs \cite{1999KarSchSin} focused on threshold functions for the property of containing a copy of a given graph $H$ and demonstrated how markedly the two models differ. Subsequent studies have confirmed that the analysis of small subgraphs is central to understanding both the similarities and the differences between $\Gnmp$ and $\Gnp$. In particular,  discrepancies in the distribution of the number of triangles account for significant differences between the two models \cite{FillScheinSinger2000Equivalence, KimLeeNa2018Equivalence, Rybarczyk2011Equivalence, BrennanBresterNagaraj2020Equivalence}.

In this article, we focus on the subgraph count in $\Gnmp$. It is one of the most fundamental statistics in random graph theory and has been studied extensively for several decades. In the case of $\Gnp$, asymptotic normality conditions were established in \cite{Ruc} and complemented by convergence rates in the Wasserstein distance in \cite{BKR}. An analogous result for the more natural Kolmogorov distance was first derived in \cite{R} for triangles, and later extended to arbitrary graphs in \cite{PSbej}. Recently, many results dealing with other random graph models have appeared; see, for example, \cite{BCJ, FLSZ, KR, LP1, LP2, Maugis, MNS, PSstoch}.

In the context of the asymptotic distribution of the number of small subgraphs in $\Gnmp$, research has focused mainly on  convergence to the Poisson distribution \cite{2010RybSta, 2017RybSta}. For certain related random intersection graph models and in various parameter regimes, normal approximations for the number of triangles \cite{Bloznelis2018} and for the number of 2--connected subgraphs \cite{Bloznelis2023} have been established.
However, none of the techniques developed in these studies can be adapted to derive conditions for the asymptotic normality of the number of small cliques in $\Gnmp$. The only known result concerning asymptotic normality for larger subgraphs in $\Gnmp$ applies solely to triangles \cite{Dong2023}. However, the distance bounds obtained there are far from optimal. In particular, they do not imply convergence to the normal distribution for $m \ge n^{2}$.
On the other hand, results in \cite{FillScheinSinger2000Equivalence, KimLeeNa2018Equivalence, Rybarczyk2011Equivalence, BrennanBresterNagaraj2020Equivalence} show that the larger the value of $m$, the more similar the models $\Gnmp$ and $\Gnp$ become. Therefore, convergence to the normal distribution should be expected in this parameter regime. Finally, we note that the authors of \cite{Dong2023} assume that the edge probability tends to zero.

Recently, the authors of this paper have developed a new method that considerably simplifies the analysis of the normal approximation for the number of small subgraphs in random intersection graphs. In \cite{Edges}, an application of this method was demonstrated for the number of edges and a wide range of parameters was considered. 
In the present paper, we develop new combinatorial tools aimed at extending this approach from \cite{Edges} to the more general setting of arbitrary cliques. Our results cover the full range of parameters typically considered in the context of $\Gnmp$,  including the cases where the edge probability is constant or tends to 1. This is achieved,  in part,  through the probabilistic techniques developed in \cite{Edges}. Nevertheless, considering larger graphs significantly increases the complexity of the problem, which required a detailed analysis of subgraphs probabilities. We develop techniques for analysing so-called clique covers. These methods allow for the comparison of occurrence probabilities of different subgraphs, which proves crucial for studying the central moments of certain random variables -- a necessary step in establishing convergence to the normal distribution. We believe that these results may be of independent interest, as the techniques introduced here can be applied to study the probabilities of other subgraphs in random intersection graphs. Furthermore, the general framework developed in this work may also prove valuable for estimating central moments in other models of random graphs.

In the article we   use the following asymptotic notation:
\begin{itemize}
	\item $a_n=o(b_n) \Leftrightarrow $ $a_n/b_n\to 0$  as $n\to \infty$,   
	\item $a_n=O(b_n) \Leftrightarrow $ $|a_n|\le Cb_n$  for some $C>0$ and all  $n\in \N$,  
	\item $a_n\lesssim b_n \text{ or }b_n\gtrsim a_n\Leftrightarrow   a_n=O(b_n)$ and $a_n,    b_n>0$,   
	\item $a_n\asymp b_n \Leftrightarrow   a_n\lesssim b_n$ and $ b_n\lesssim a_n$.
\end{itemize}  
All constants implicit in the asymptotic notation are functions of the size of the considered graph.

We also use standard notation from graph theory.  For a graph $G=(V(G),    E(G))$ we denote by $V(G)$ its vertex set and by $E(G)$ its edge set. 
$G'\subseteq G$ is a subgraph of $G$ if $V(G')\subseteq V(G)$ and $E(G')\subseteq E(G)$. For $V'\subseteq V(G)$ and $E'=\{e\in E(G): e\subseteq V'\}$,    $G[V']=(V', E')$ is the subgraph of $G$ induced on the vertex set $V'$. 
For $E'\in E(G)$ and $V'=\{v:\exists_{e\in E'}\,    v\in e\}$,    $G[E']=(V',    E')$ is the subgraph of $G$ induced on the edge set $E'$,    and    $G-E'=(V(G),       E(G)\setminus E')$.
Furthermore, 
$G\cup G'=(V(G)\cup V(G'),    E(G)\cup E(G'))$ and 
$G\cap G'=(V(G)\cap V(G'),    E(G)\cap E(G'))$,    $\deg_{G}(v)=|\{e\in E(G): v\in e\}|$ stands for the degree of a vertex $v$ in $G$.  We write $\cong$ for the graph isomorphism.

We denote by $\klik_i$ a clique of size $i$. In our further considerations,   any  $\kli$ and any $\kli_{x}$,   with an index $x$,   will be a graph isomorphic to $\klik_r$,  for fixed $r\ge2$. In addition,    we  denote by $\emp_i$  an empty graph on $i$ vertices. In particular,    we  allow a graph to be empty with no vertices and no edges, and denote it by $\emp_0$ or $\klik_0$.

\section{Main results and discussion}
\label{sec:mainresults}

Let us first introduce  the random intersection graph model  $\Gnmp$. We are  given  a set of $n\in\N$ vertices $\V$,  $|\V|=n$,  and a set of attributes $\A=\{a_1,     \ldots,     a_m\}$,  $m=m(n)\in\N$. Any  attribute $a_i$,      $1\le i\le m$,     is assigned to any vertex $v\in\V$,     (we say that vertex $v$ {\it chose} attribute $a_i$) independently with probability $p=p(n)\in(0,  1)$. Then,      two vertices $v$ and $v'$ are connected by an edge in $\Gnmp$ whenever they chose at least one common attribute. The existence probability of any edge will be denoted by $\hatp$ and is given by
\begin{align}\label{eq:pp}
	\hatp=1-(1-p^2)^{m}\asymp mp^2\wedge 1, 
\end{align}
where the estimate is valid for  $p<c$ for any fixed $c\in (0,  1)$.

We  study the random variable 
$\X$,  $r\geq2$,  that counts cliques on $r$ vertices in $\Gnmp$.
Let  $\Subgraphs$ be the family of all cliques of size $r$ in the complete graph with the vertex set $\V$. For $\kli\in\Subgraphs$,    let $\ind_\kli$ be the indicator random variable of the event $\{\kli\subseteq \Gnmp \}$ and define
\begin{equation}
	\qq=\E [\ind_{\kli}]=\Pra{\kli\subseteq \Gnmp},  \quad \kli\in \Subgraphs.
\end{equation}
Then
$$
\X=\sum_{\kli\in \Subgraphs}\ind_{\kli}, \ \ \ \text{ and }\ \ \ 
\E [\X]=\sum_{\kli\in \Subgraphs}\E [\ind_{\kli}] =\binom{n}{r}\qq.
$$
Note that estimates of $\qq$ in terms of  $m$,  $p$,  and $\pp$ were derived in \cite{1999KarSchSin}, and  are recalled below  in Lemma \ref{Lem:PiOptimal}. Our main goal is to examine the asymptotic normality of the standardized random variable 
$$\widetilde \X:=\frac{\X-\E[\X]}{\sqrt{\Var[\X]}}, $$
and to establish bounds on its convergence rates to a normally distributed random variable $\mathcal N\sim \mathcal N(0,     1)$. 
An important role in this context plays the expression
\begin{equation}\label{Eq:Def:Ble}
	\ble := (n^r\qq) \wedge (n^2\pp), 
\end{equation}
which has the following probabilistic interpretation 
$$\ble \asymp \binom{n}{r}\qq\wedge\binom{n}{2}\pp=\E[\X]\wedge \E[X_n^{(2)}] .$$
Furthermore,  $\ble$ is analogous to the expression appearing in the context of small subgraphs of $\Gnp$   (see for example $\Phi$ in \cite{KsiazkaJLR}). However,  contrary to the  case of $\Gnp$,   $\ble$ is not comparable to $\min_{2\le i\le r}\{\E[X_n^{(i)}]\}$.

Below,  we present our first result,  i.e. the estimates of  the variance of $\X$.

\begin{theorem}\label{Tw:Wariancja}
	Let $\X$ be the number of cliques of size $r$ in $\Gnmp$ and $mp^3\le c_r$ for some constant $c_r$ depending on $r$ only. Then
	\begin{align}\label{eq:Var}
		\Var [\X] &\asymp
		\frac{(n^r\qq)^2}{\ble}(1-\pp)+n^{2r-1}\left(\frac{\qq}{\pp}\right)^2mp^3(1-\pp)^2.
	\end{align}
\end{theorem}
The first term on the right-hand side of \eqref{eq:Var} is analogous to the estimate of the variance in $\Gnp$.  However,  the second term does not appear in the respective formula for $\Gnp$. Its presence stems from the fact that in $\Gnmp$ edges with a common vertex are not independent,  and therefore two graphs that share a vertex are not generated independently.

Quantitative central limit theorems are usually described by means of the Kolmogorov distance or the Wasserstein one.  For two random variables $X$ and $Y$ we define the Kolmogorov distance $d_K(X,     Y)$ between them as
\begin{align*}
	d_{K}(X,     Y):=\sup_{t\in \R}\big|\P\left(X\le t\right)-\P(Y\le t)\big|,     
\end{align*}
while the Wasserstein distance $d_W(X,     Y)$ is given by
\begin{align*}
	d_{W}(X,     Y):=\sup_{\text{Lip}(h)\le1}\big|\E[h(X)-h(Y)]\big|,     
\end{align*}
where $\text{Lip}(h)$ stands for the   Lipschitz constant of a function $h:\R\rightarrow\R$. Since we  work with both of the distances  simultaneously,      we also denote
$$d_{K/W}(X,     Y):=\max\{d_{K}(X,     Y),      d_{W}(X,     Y)\}.$$
The main result of this article is as follows.
\begin{theorem}\label{Thm:Main} 
	Assume $mp^3\le {c_r}$ with $c_r$ from Theorem \ref{Tw:Wariancja}. Then
	$$\widetilde \X\stackrel{\mathcal D}{\longrightarrow}\mathcal N \ \Leftrightarrow\ (1-\pp)\ble\rightarrow\infty, $$
	where $\mathcal N\sim \mathcal N(0,   1)$. Furthermore,  	there exist  constants $m_0,   n_0\in\N$ and $ C,  C'>0$  depending on $r$ only such that for $n\ge n_0,  m\ge m_0$ we have 
	\begin{align}\label{eq:main}
		d_{K/W}\left(\widetilde{\X},   \mathcal N\right)&\le C \left(\frac{1}{(1-\pp)\ble}+\frac{1}{n}+\frac1{nmp}\right)^{1/2}, 
	\end{align}
	and 
	\begin{align}\label{eq:main2}
		d_{K/W}\left(\widetilde{\X},   \mathcal N\right)
		\le C' \left(\frac1{(1-\pp)\ble}\right)^{1/4}.
	\end{align}
\end{theorem}

In the $\Gnp$ model,  the bound \eqref{eq:main} takes the form $\big((1-\pp)\ble\big)^{-1/2}$ \cite{PSbej,  R}. The complexity of the $\Gnmp$ model makes it necessary to include two additional terms under the square root. However,  the term $n^{-1/2}$ is not surprising in this context,  as it matches the bound from the classical Berry-Esseen theorem. The term $nmp$,  in turn,  is closely related to the structure of the $\Gnmp$ model and represents the average number of links between vertices and attributes,  which is assumed to tend to infinity. Let us also note that in \cite{Edges},  where we considered only the case $r=2$,  the term $1/m$ appears instead of $1/(nmp)$ in \eqref{eq:main},  which is worse whenever $mp>1$. This inconsistency follows from the fact that in \cite{Edges} the calculations were simpler and yielded the minimum of these two terms,  whereas we bounded it by $1/m$ for the sake of a simpler expression.

The organization of the remaining part of article is as follows. In Section~\ref{Sec:Preliminaries}, we begin by presenting Lemma~\ref{Lemma:Main}, a result taken from \cite{Edges}. It relates $d_{K/W}\left(\widetilde{\X},   \mathcal N\right)$ to a formula involving the probabilities that specific graphs appear as subgraphs of $\Gnmp$. It suffices to compute these  probabilities in order to obtain an upper bound for the distance to the normal distribution, and thus to prove Theorem~\ref{Thm:Main}.
Section~\ref{Sec:CCs} is devoted to the development of techniques that yield sharp estimates for the probabilities of the graphs considered in Lemma~\ref{Lemma:Main}. In the subsequent two sections, we prove Theorem~\ref{Tw:Wariancja} and Theorem~\ref{Thm:Main}, respectively. Both proofs rely on the results established in Section~\ref{Sec:CCs}.

\section{Preliminaries}\label{Sec:Preliminaries}

\subsection{Graphs $G_{t,i,j}$}\label{Sec:SubgraphsIntegrals}

\begin{figure}
	\begin{center}
		\includegraphics[width=15cm]{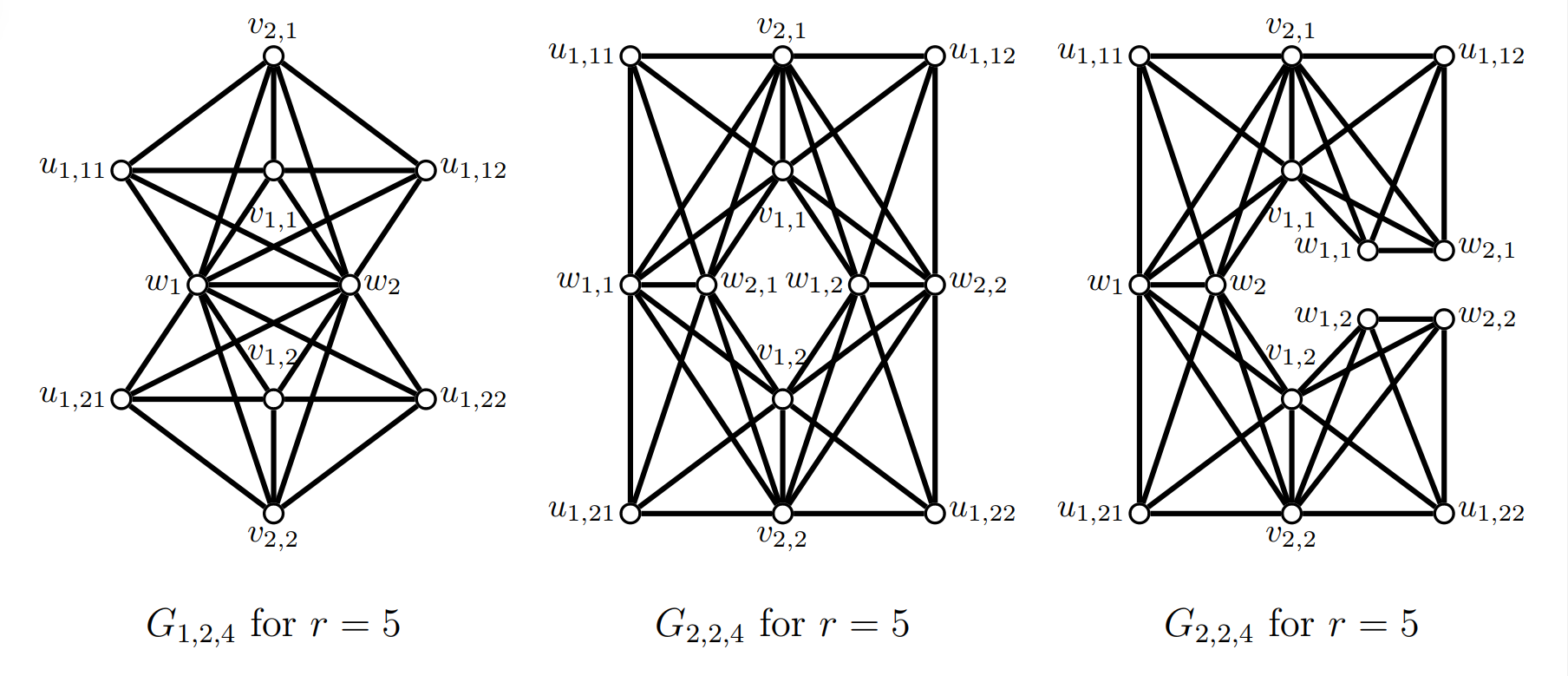}
	\end{center}
	\caption{\label{Fig:GtExamples}  Examples of $G_{t,  i,  j}$ for $r=5$,   $i=2$,   and $j=4$.}
\end{figure}

In this section we define graphs $G_{t,i,j}$ introduced in \cite{Edges}. In general, for each $t,  i,  j$, $G_{t,i,j}$ is a union of four cliques 
$\kli_{ab}$,  $ab\in \indeksy=\{11,  12,  21,  22\}$ of size $r$, that is
\begin{equation}\label{Eq:DefGt}
	G_{t,  i,  j}=\bigcup_{ab\in \indeksy}\kli_{ab}, 
\end{equation}
where the  cliques $\kli_{ab}$,  $ab\in \indeksy$,   intersect in a special manner depending on $t,  i,  j$,  which is  described below.
In  Figure~\ref{Fig:GraphRepresentations} one may find the geometrical idea behind the definition,   and in Figure~\ref{Fig:GtExamples} we present some particular cases.

\begin{definition}

	Let $t=1, 2, 3$, $0\le i < j \le r$, and $a,b\in \{1,2\}$ . We define graphs $G_{t,  i,  j}$ as follows.
	\begin{itemize} 
		\item[(i)] For $0\le i < j \le r$, let
		\begin{align*}
			V(G_{1,  i,  j})=&
			\{v_{l,  a}: l\in \{1,  \ldots,  i\},   a\in \{1,  2\}\}\\
			&\cup
			\{w_{l}: l\in \{1,  \ldots,  j-i\}\}\\
			&\cup 
			\{u_{l,  ab}: l\in \{1,  \ldots,  r-j\},   a,  b\in \{1,  2\}\},
		\end{align*}
		and for each $a,b\in \{1,2\}$   let
		\begin{equation*}
			K_{ab}=G_{1,  i,  j}[\{v_{l,  a}: l\in \{1,  \ldots,  i\}\}\cup  \{w_{l}: l\in \{1,  \ldots,  j-i\}\}\cup \{u_{l,  ab}: l\in \{1,  \ldots,  r-j\}\}]
		\end{equation*}
		be a clique.
		\item[(ii)] For $1\le i < j \le r$, let
		\begin{align*}
			V(G_{2,  i,  j})=&
			\{v_{l,  a}: l\in \{1,  \ldots,  i\},   a\in \{1,  2\}\}\\
			&\cup
			\{w_{l,  b}: l\in \{1,  \ldots,  j-i\},   b\in \{1,  2\}\}\\
			&\cup 
			\{u_{l,  ab}: l\in \{1,  \ldots,  r-j\},   a,  b\in \{1,  2\}\}.
		\end{align*}
		and for each $a,b\in \{1,2\}$   let
		\begin{equation*}
			K_{ab}=G_{2,  i,  j}[\{v_{l,  a}: l\in \{1,  \ldots,  i\}\}\cup  \{w_{l,  b}: l\in \{1,  \ldots,  j-i\}\}\cup \{u_{l,  ab}: l\in \{1,  \ldots,  r-j\}\}]
		\end{equation*}
		be a clique.
		\item[(iii)] For $1\le i < j \le r$, let
		\begin{align*}
			V(G_{3,  i,  j})=&
			\{v_{l,  a}: l\in \{1,  \ldots,  i\},   a\in \{1,  2\}\}\\
			&\cup
			\{w_{l}: l\in \{1,  \ldots,  j-i\}\}\\
			&\cup
			\{w_{l,  a}: l\in \{1,  \ldots,  j-i\},   a\in \{1,  2\}\}\\
			&\cup 
			\{u_{l,  ab}: l\in \{1,  \ldots,  r-j\},   a,  b\in \{1,  2\}\}.
		\end{align*}
		and for each $a\in \{1,2\}$   let
		\begin{align*}
			K_{a1}&=G_{3,  i,  j}[\{v_{l,  a}: l\in \{1,  \ldots,  i\}\}\cup  \{w_{l}: l\in \{1,  \ldots,  j-i\}\}\cup \{u_{l,  a1}: l\in \{1,  \ldots,  r-j\}\}]\\
			K_{a2}&=G_{3,  i,  j}[\{v_{l,  a}: l\in \{1,  \ldots,  i\}\}\cup  \{w_{l,  a}: l\in \{1,  \ldots,  j-i\}\}\cup \{u_{l,  a2}: l\in \{1,  \ldots,  r-j\}\}]
		\end{align*}
		be a clique.
	\end{itemize}
	Moreover  no other edges except those of $K_{a,b}$, $a,b\in \{1,2\}$, are present in $G_{1,  i,  j}$.
	
\end{definition}
In the sequel,   we will omit indices $i,  j$ and write $G_t=G_{t,  i,  j}$,   when it will be clear from the context which values of $i,   j$ are considered.
Note that for $t=1$ and $j=r$ we have $\kli_{11}=\kli_{12}$ and $\kli_{21}=\kli_{22}$, which implies the next fact. While this case is special, the same proof methods will nevertheless apply.
\begin{fact}\label{Fact:G1r}
	$G_{1, i, r}$ is isomorphic to a union of two cliques $\klik_r$ intersecting on $r-i$ vertices. In particular $G_{1, 0, r}\cong \klik_r$. 
\end{fact}

\begin{figure}
	\begin{center}
		\includegraphics[width=13cm]{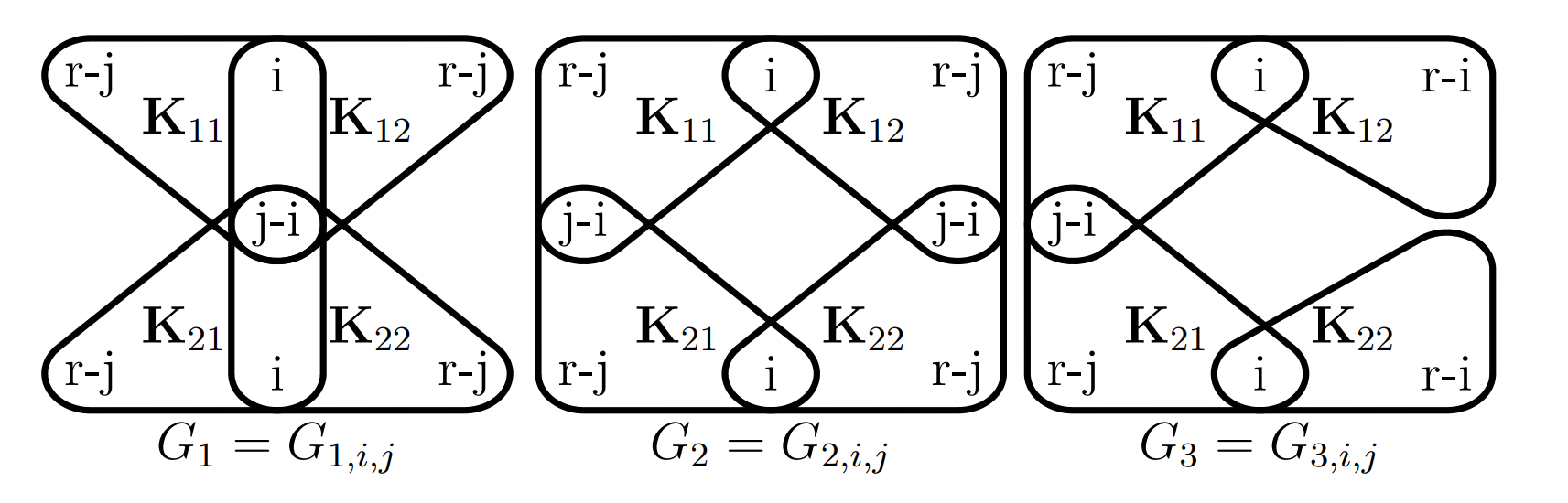}
		\caption{\label{Fig:GraphRepresentations} Symbolic depictions of graphs $G_{t,i,j}$.
		}
	\end{center}
\end{figure}

Given  $i,   j$,   for $t\in \{1,    2,    3\}$,     $\II\subseteq \indeksy$ and  $G_t=G_{t,    i,    j}$,    we define 
\begin{align}\label{eq:GtI}
	G_{t,   \II}:=\bigcup_{ab\in\II}\kli_{ab},  
\end{align}
i.e., $G_{t,   \II}$ is the subgraph of $G_t$ induced by the vertex sets of the cliques with indices in~$\II$. In addition,  we use the convention that $G_{t,   \emptyset}$ has neither vertices nor edges and $\pi(G_{t,   \emptyset})=1$. 

The following lemma was proven in \cite{Edges}. 

\begin{lemma}[Corollary 4.3 \cite{Edges}]\label{Lemma:Main}
	Let $t=1,2,3$ and $G_t=G_{t,i,j}$ be defined as above. Then there exists a constant $C'=C'(r)>0$ such that
	\begin{equation*}
		d_{K/W}\left( \widetilde \X ,      \mathcal{N}\right)
		\le\frac{C'}{\Var [\X]} 
		\sum_{t,i,j}\left(n^{|V (G_t)|} \sum_{\II\subseteq\indeksy}(-1)^{|\II|}\qq^{|\indeksy|-|\II|}\pi(G_{t,   \II})\right).
	\end{equation*}
\end{lemma}
In the remainder of the paper, we focus on obtaining estimates for the quantity appearing on the right-hand side of the equation in the lemma.

\section{Subgraph probabilities}\label{Sec:CCs}

In this section, we present general results concerning probabilities of subgraph occurrence in $\Gnmp$. These results will be crucial for determining the asymptotics of the variance $\Var[\X]$ and for deriving bounds on the probabilities appearing in the formula of Lemma~\ref{Lemma:Main}. We believe that the results of this section may be of independent interest, as they provide convenient tools for analysing the probability that a given graph appears as a subgraph of $\Gnmp$ for a wide range of graphs, not only those involving cliques.

In the following discussion,  we will determine the probabilities of graphs that can be represented as unions of graphs. For such graphs,  we introduce new notation.
\begin{definition}\label{def:oplus}
	For two graphs $H_1$ and $H_2$,  let  $H_1'$ and $H_2'$ be their respective isomorphic copies,  such that $H_1'\cap H_2'$ is a clique on $i$ vertices ($i\ge 0$). Then 
	$H_1\oplus_{i}H_2 = H_1'\cup H_2'$,  i.e. the graph consisting of copies of $H_1$ and $H_2$ intersecting on a clique on $i$ vertices. Similarly,  by 
	$H_1\oplus_{i,   j}H_2$, $i, j\ge1$, we  denote a graph consisting of copies of $H_1$ and $H_2$ intersecting on two  disjoint cliques,   one of size $i$,    and the other of size $j$ (see Figure~\ref{Fig:CliquesSums}).  
\end{definition}
In Definition~\ref{def:oplus},  
formally we consider a new vertex set and isomorphic copies of $H_1$ and $H_2$. However, for simplicity, we refer to these isomorphic copies simply as $H_1$ and $H_2$. Clearly, not all graphs of the form $H_1 \oplus_i H_2$ (or $H_1 \oplus_{i,j} H_2$) are isomorphic. Nevertheless, it will be clear from the context which graph is meant, or this distinction will be irrelevant.

\begin{figure}
	\begin{center}
		\includegraphics[width=6cm]{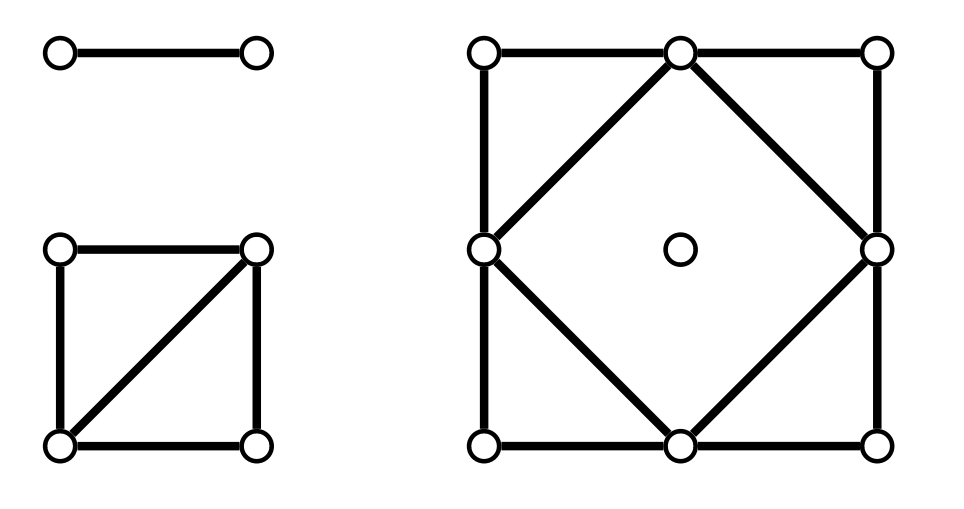}
	\end{center}
	\caption{\label{Fig:CliquesSums} Examples of $\klik_3\oplus_2 \klik_3\oplus_0 \klik_2$ and $(\klik_3\oplus_1 \klik_3)\oplus_{1,   1}(\klik_3\oplus_1 \klik_3)\oplus_0\emp_1$}
\end{figure}

\subsection{Clique covers}
In this section,  we define the clique covers introduced in \cite{1999KarSchSin} with some modifications proposed in \cite{SubgraphProb}. They play a crucial role in determining the probabilities of subgraphs in $\Gnmp$.

\begin{definition}
	For a given graph $H$ we define  $\calP(H)$ as the family of subsets of $V(H)$ that contain both ends of at least one edge of $H$. The sets $C\in\calP(H)$ of size $t$ will be called $t$-sets.
	A subfamily $\cover\subseteq \calP(H)$ is called a {\it clique cover} of $H$,   if    every edge (a $2$-set)  $e\in E(H)$ is contained in some $C\in\cover$.  Furthermore,  the set of all clique covers of $H$ will be denoted by~$\covers(H)$. 
\end{definition}

For a graph $H$  with $V(H)\subseteq \V$ we say that $a\in \A$ {\it builds} $C\subseteq V(H)$ if $C$ is the set of all the vertices of $H$ that chose $a$,  i.e.  $C=\{v\in V(H): v \text{ chose }a\}.$
The probability of such an event is given by
\begin{align}\label{eq:p_C}
	p_C:=p^{|C|}(1-p)^{|V(H)|-|C|}.
\end{align}
Then,   $H$ is {\it given by a clique cover} $\cover$ in $\Gnmp$ if  every $C\in \cover$  is built by an attribute from $\A$ and no attribute from $\A$ builds any $C\in \calP(H)\setminus \cover$. 

By $\pi(H,     \cover)$ we denote the probability that $H$ is given by a clique cover $\cover$ in $\Gnmp$. For a subfamily of clique covers $\covers\subseteq \covers(H)$ let $\pi(H,     \covers)$ be  the probability that $H$ is given by a clique cover from   $\covers$. We clearly have
\begin{equation}\label{Eq:SubgraphProbabilityCCfamily}
	\pi(H,     \covers)=\sum_{\cover\in \covers} \pi(H,     \cover). 
\end{equation}
In particular,     the probability that $H$ occurs as a subgraph in $\Gnmp$ is given by
\begin{align}\label{Eq:SubgraphProb}
	\pi(H):=\Pra{H\subseteq \Gnmp}=\pi(H,     \covers(H))
	=\sum_{\cover\in\covers(H)}\pi(H,     \cover). 
\end{align}
Next,     for a clique cover
$\cover_+\subseteq \calP (H)$ and a set
$\cover_-\subseteq \calP(H)\setminus \cover_+$
we write  $\covers(H,     \cover_+,     \cover_-)$ for the family of  all the clique covers $\cover$ containing all sets from  $\cover_+$ and no set from $\cover_-$,      i.e.  such that $\cover_+\subseteq \cover\subseteq\calP(H)\setminus \cover_-$. The sets from $\cover_-$ will be called {\it forbidden}. Additionally,  we will be using the shorthand notation
\begin{equation}\label{Eq:SubgraphProbabilityCCfamily2}
	\pi(H,     \cover_+,     \cover_-):=\pi(H,     \covers(H,     \cover_+,     \cover_-))=\sum_{\cover\in \covers(H,     \cover_+,     \cover_-)} \pi(H,     \cover). 
\end{equation}
In particular,    we have the following identity $\pi(H,     \cover)=\pi(H,     \cover,     \calP(H)\setminus\cover)$.

The following asymptotics were shown in  \cite{SubgraphProb}. See also
\cite{1999KarSchSin},    \cite{2017RybSta},    \cite{Dong2023} for some previous,  less general results.

\begin{lemma}[\cite{SubgraphProb}]\label{Lem:ProbCCs} 
	Let $H$ be a graph and   $\calP(H)$  be the family of subsets of $V(H)$ containing both ends of at least one edge of $H$. Fix a clique cover  $\cover_+$   and a set $\cover_-\subseteq \calP(H)\setminus \cover_+$.
	If  $m\to \infty$ and $p$ is bounded away from $1$ by a constant,    then 
	\begin{align}\label{eq:CC1}
		\pi(H,   \cover_+,    \cover_-)=\left(1+O(m^{-\delta_{|V(H)|}})\right)  \left(1-\sum_{C\in \cover_-}p_C\right)^{m}\prod_{C\in\cover_+}(1-e^{-mp_C}),  
	\end{align}
	where $p_C$ is defined in \eqref{eq:p_C},  and $\delta_{k}=1/({k2^{k+1}})$.

	Moreover, for $mp^2\le 1$ we have
	\begin{equation}\label{eq:CC2}
		\pi(H,   \cover_+,    \cover_-)\asymp \pi(H,   \cover_+)\asymp \prod_{C\in\cover_+}mp^{|C|}=m^{|\cover_+|}p^{\sum_{C\in \cover_+} |C|},  
	\end{equation}
	and
	\begin{equation}\label{eq:CC3}
		\pi(H,  E(H))=(1+O(m^{-\delta_{|V(H)|}}))\pp^{|E(H)|}, 
	\end{equation}
	\ \\[-10pt]where $\delta_{k}=1/({k2^{k+1}})$.
	
	Furthermore,   for $mp^2\ge 1$ and $mp^3\le 1$,  we have
	\begin{equation}\label{eq:CC4}
		\pi(H,   \cover_+,    \cover_-)
		\asymp e^{-|\cover_-^{(2)}|mp^2}\prod_{C\in\cover_+\setminus \cover_+^{(2)}}mp^{|C|},      
	\end{equation}
	where $\cover_+^{(2)}$ and $\cover_-^{(2)}$ are the families of all 2--sets in $\cover_+$ and $\cover_-$,    respectively.
\end{lemma}

We close this section with  another result from \cite{SubgraphProb}. 
For two families of sets $\cover$ and $\cover'$ we write 
\begin{align}\label{Bumpeq}
	\cover\Bumpeq\cover',
\end{align}
if  there is a bijection $f:\cover\rightarrow \cover'$ such that $|f(C)|=|C|$ for any $C\in \cover$.

\begin{lemma}[Proposition 2.3 \cite{SubgraphProb}]\label{Lem:CCequality}
	Let $H_1$ and $H_2$ be two graphs with the same number of vertices. Let,      for $i=1,     2$,      $\cover_+^{(i)}$ be a clique cover of $H_i$ and 
	$\cover_-^{(i)}\subseteq \calP(H_i)\setminus \cover_+^{(i)}$. If $\cover_+^{(1)}\Bumpeq\cover_+^{(2)}$  and $\cover_-^{(1)}\Bumpeq\cover_-^{(2)}$,    then
	$$
	\pi(H_1,     \cover_+^{(1)},     \cover_-^{(1)})=\pi(H_2,    \cover_+^{(2)},     \cover_-^{(2)}).
	$$
\end{lemma}

\subsection{Basic subgraph probabilities}

This section collects various general results concerning subgraph probabilities.

\begin{fact}\label{Lem:Division2}
	For any graphs $H_1,  H_2$ it holds that
	$$\pi(H_1\cup H_2)\geq \pi(H_1)\pi(H_2).$$
\end{fact}
\begin{proof}
	The indicators $\mathbf 1_{\{H_1\subseteq \Gnmp\}}$ and $\mathbf 1_{\{H_2\subseteq \Gnmp\}}$ are increasing in relation to Bernoulli random variables representing links between vertices and their attributes. Therefore,  by FKG inequality (see for example Theorem 2.12 in \cite{KsiazkaJLR}), we have
	$$\E[\mathbf 1_{\{H_1\subseteq \Gnmp\}}\mathbf 1_{\{H_2\subseteq \Gnmp\}}]\ge \E[\mathbf 1_{\{H_1\subseteq \Gnmp\}}]\E[\mathbf 1_{\{H_2\subseteq \Gnmp\}}], $$
	which is equivalent to
	the assertion.
\end{proof}

The following lemma was proved in \cite{1999KarSchSin} under an additional assumption on 
$m$. However,  the argument in \cite{1999KarSchSin} can be adapted to yield the version stated below, see also \cite{SubgraphProb}. The lemma can also be obtained using the methods developed in this paper.

\begin{lemma}\label{Lem:PiOptimal}For  $\kli\cong \klik_r$ we have
	\begin{align*}
		\qq=\pi(\kli)&\asymp
		\begin{cases}
			\pi (\kli,   \{V(\kli)\})\asymp mp^{r}&\text{ for }p\le m^{-\frac{r+1}{2r}};\\
			\pi (\kli,   E(\kli))\asymp \pp^{\binom{r}{2}}&\text{ for }p\ge m^{-\frac{r+1}{2r}}.
		\end{cases}
	\end{align*}
\end{lemma}
Let us note that due to the estimate $\pp\asymp mp^2$ valid for $mp^2<1$, the statement of the lemma is equivalent to  $\qq=\pi(\kli)\asymp \big(mp^r+\pp^{\binom{r}{2}}\big).$

It will be convenient to define
\begin{align}\label{Eq:Def:Blei}
	\ble_i=\ble_i(n,   m,   p,   r)&=
	\begin{cases}
		n^imp^i\asymp n^i\pi(K_i,  \{V(K_i)\})   &\text{ for }p\le m^{-\frac{r+1}{2r}};\\	
		n^i\pp^{\binom{i}{2}}\asymp n^i\pi(K_i,  E(K_i))   &\text{ for } p\ge m^{-\frac{r+1}{2r}}.	
	\end{cases}
\end{align}
The following conclusion concerning $\ble$,  defined in \eqref{Eq:Def:Ble},  holds true
\begin{align}\label{Eq:BleaBlei}
	\min_{2\le i\le r}\ble_i\asymp \ble, 
\end{align}
which is a consequence of  the fact that for $s\in[0, 1]$ we have $\min_{2\le i\le r}\{n^is^{\binom{i}{2}}\}=(n^2s)\wedge (n^rs^{\binom{r}{2}})$,   see Remark 3.17 in \cite{KsiazkaJLR} or Proposition 4.7 in \cite{PSbej}. Let us note that Lemma \ref{Lem:PiOptimal} implies that $\ble_i$ is not comparable to  $\E[X_n^{(i)}]$,  and $\min_{2\le i\le r}\E[X_n^{(i)}]$ is not comparable to $\ble$,  whereas it is the case in $\Gnp$ model.

The next result will be mostly useful in analysing subgraph probabilities in the case when $mp^2$ is bounded away from zero,  which is equivalent to  $\pp\asymp1$.

\begin{fact}\label{Lem:DuzeMp2dwa} For any non-empty  graph $H$ we have
	$$\pp^{|E(H)|}\le\pi(H)\le \pp, $$
	and
	$$1-\pi(H)\asymp 1-\pp.$$
\end{fact}
\begin{proof}
	The inequality $\pi(H)\le \pp$ follows from the inclusion $\{H \subseteq \Gnmp\}\subseteq \{e\in \Gnmp\} $ for any $e\in E(H)$.  On the other hand,   Fact~\ref{Lem:Division2} applied inductively gives  $\pp^{|E(H)|}\le\pi(H)$,  which  gives us  the first part of the assertion. Furthermore,  this implies
	\begin{align*}
		1-\pp\le 1-\pi(H)\le1-\pp^{|E(H)|}\le |E(H)|(1-\pp),  
	\end{align*}
	which completes the proof.
\end{proof}

\subsection{Comparing clique covers}\label{Sec:CCs:Basic}

In this section, we develop techniques for identifying which clique covers of a given graph are most probable, i.e., those that contribute most to the sum \eqref{Eq:SubgraphProb}. In particular, these methods enable the determination of the variance and covariance of the indicators of a given subgraph's existence in $\Gnmp$. We begin by introducing some additional notation.

For $\cover\in \covers(H)$,    $V'\subseteq V(H)$ and    $H'=H[V']$  we define
\begin{equation}\label{Eq:InducedCliqueCover}
	\cover[V']=\{C\cap V':\exists_{e\in E(H')} e\subseteq C\}.
\end{equation}

Furthermore,  for  $\cover'\subseteq \cover$,    and $\cover''\subseteq \mathcal{P}(V(H))$ we define a modification $\cover[\cover'\mapsto \cover'']$  of $\cover$ by
\begin{equation}\label{Eq:ChangeCliqueCover}
	\cover[\cover'\mapsto \cover'']=(\cover\setminus\cover')\cup \cover''.
\end{equation}
{In what follows,   $\cover'$ and $\cover''$ will be chosen in such a way that $\cover[\cover'\mapsto \cover'']$ will always be a clique cover of the considered graph $H$.}

\begin{definition}\label{Def:Simple}
	Let $H$ be a graph.
	A clique cover $\cover\in \covers(H)$ is a {\it simple  clique cover} of $H$  if it satisfies the following properties:
	\begin{itemize}
		\item[\Pone] $\forall_{C_1,   C_2\in \cover}$ $|C_1\cap C_2|\le 1$;
		\item[\Ptwo] $\forall_{C\in \cover}\  H[C]$ is connected. 
	\end{itemize}
\end{definition}
Let us note that {\Pone} implies that each edge is covered by at most one set from $\cover$.

\begin{lemma}\label{Lem:Simple}
	Let $mp^2\le 1$,    $H$ be a graph,    and $\cover\in\covers(H)$.
	\begin{itemize}
		\item[(i)] There exists a simple $\cover'\in\covers(H)$ such that $$\pi(H,   \cover)\lesssim \pi(H,   \cover').$$
		
		\item[(ii)] For every clique cover with no {\Pone}  there exists $\cover'\in\covers(H)$ with {\Pone} such that 
		\begin{equation}\label{Eq:noP1toP1}
			\pi(H,   \cover)\lesssim \pi(H,   \cover'),   
			\quad\forall_{C\in \cover}\exists_{C'\in \cover'}C\subseteq C',    \quad
			\text{and}\quad |\cover'|\le |\cover|-1.
		\end{equation}
		Additionally,  if $\cover$ satisfies \Ptwo,  then the assertion remains true with $\cover'$  satisfying $\Ptwo$ as well.
	\end{itemize}
\end{lemma}
\begin{remark}
	Lemma~\ref{Lem:Simple}(i)  implies that    to estimate a subgraph probability for $mp^2\le 1$,    we may consider only simple clique covers,    which we will do from now on.
\end{remark}
\begin{proof}[Proof of Lemma~\ref{Lem:Simple}] 
	
	We start with the proof of $(ii)$. Assume that $\cover$ does not have {\Pone}. Then,  there are  two sets $C',   C''\in \cover$ that intersect on at least two vertices. Therefore,     by \eqref{eq:CC2},  we get  
	\begin{align*}
		\frac{\pi(H,   \cover)}{\pi(H,   \cover[\{C',   C''\}\mapsto \{C'\cup C''\}])}
		\asymp \frac{mp^{|C'|}mp^{|C''|}}{mp^{|C'\cup C''|}}
		=mp^{|C'\cap C''|}\le mp^2\le 1.
	\end{align*}
	Applying this procedure inductively,    we arrive at a clique cover $\cover$ with {\Pone} and satisfying \eqref{Eq:noP1toP1}. 
	Moreover,   if $\cover$ satisfies \Ptwo,  then  the graphs $H[C']$ and $H[C'']$ are connected,  and hence  $H[C'\cup C'']$ is  connected as well. Therefore,  the procedure described above leads to $\cover'$ with {\Ptwo}. 
	
	In view of $(ii)$,  in order to prove $(i)$ it is enough to show existence of $\cover'$ satisfying $\Ptwo$.
	If $\cover$ does not satisfy \Ptwo,  there exists $C\in\cover$ such that $H[C]$ is not connected.  Let $C_1,  \ldots,  C_k$ be pairwise disjoint sets such that $C_1\cup C_2\cup \ldots \cup C_k\subseteq C$,  $|C_i|\ge 2$,  $i=1, \ldots,  k$,  the set $C_1\cup C_2\cup \ldots \cup C_k$ covers all edges of $H[C]$,  and  $H[C_i]$ is connected,  $i=1, \ldots,  k$. Since $H[C]$ contains an edge by the definition of clique cover,  it holds that  $k\ge1$.  Then,    by \eqref{eq:CC2},  for $\cover'=\cover[\{C\}\mapsto\{C_1,  \ldots,  C_k\}]$ we have
	\begin{align*}
		\frac{\pi(H,   \cover)}{\pi(H,   \cover')}\asymp \frac{mp^{|C|}}{m^kp^{|C_1|+ \ldots+|C_k|}}= \frac{p^{|C|-|C_1|-\ldots-|C_k|}}{m^{k-1}}\le 1.
	\end{align*}
	Since $H[C_1],  \ldots ,  H[C_k]$ are connected,  after repeating the reasoning above finite number of times,    we obtain a new clique cover with {\Ptwo}.  
\end{proof}

The remainder of the section is primarily devoted to the proof of the lemma below. For $mp^2\le 1$, it will play a key role in estimating the probabilities of the graphs $G_t$ in which the intersecting cliques $\kli_{ab}$ share at least two vertices. Several intermediate results will also prove useful in other arguments. Moreover, the reasoning employed in the proof of this lemma may be applied to the analysis of probabilities of other graphs.

\begin{lemma}\label{Lem:Odrywanie}
	Assume $mp^2\le 1$. Let  $j-i\ge 2$ and let $i\ge 0$ for $t=1$ and $i\ge 2$ for  $t=2, 3$. Moreover, let $t$ and $\II$,  $|\II|\ge 2$,  be such that $G_{t, \II}$ is connected. Than it holds that,  for any simple $\cover\in \covers (G_{t, \II})$,  there exists $ab\in \II$ such that 
	$$
	n^{|V(G_{t, \II})|}\pi(G_{t, \II},   \cover)\lesssim \frac{n^r\qq\cdot n^{|V(G_{t, \II\setminus\{ab\}})|}\pi(G_{t, \II\setminus\{ab\}})}{\ble}, 
	$$
	and $G_{t, \II\setminus\{ab\}}$ is connected.
\end{lemma}

\begin{remark}
	In the case $t=1$ and $r=j$,  for $|\II|\ge 3$ {or} $\II\in \{\{11, 12\}, \{21, 22\}\}$ it is easy to find $ab$ such that  $G_{t, \II\setminus\{ab\}}=G_{t, \II}$. Then, the above lemma follows trivially  by  the inequality $\ble\le n^r\qq$. Furthermore, if we additionally assume   $i=0$,  the same argument holds true  for any $\II$ such that $|\II|=2$. 
\end{remark}

We begin with a lemma that provides  simple estimates of clique cover probabilities when coverings by 2-sets are more likely,  and when coverings by larger sets are more probable.

\begin{lemma}\label{Lem:Splitting}
	Let $mp^2\le 1$,    $H$ be a  graph,    $\cover\in \covers(H)$ {be a simple clique cover},    $C\in \cover$,    $|C|\ge 3$. 
	\begin{itemize}
		\item[(i)] If $p\le m^{-\frac{|C|+1}{2|C|}}$ and $H[C]$ is a clique then
		$$
		\pi(H,   \cover[\{C\}\mapsto E(H[C])])\lesssim \pi(H,   \cover).
		$$
		\item[(ii)] If $p\ge m^{-\frac{|C|+1}{2|C|}}$, then
		$$
		\pi(H,   \cover)\lesssim \pi(H,   \cover[\{C\}\mapsto E(H[C])]).
		$$
		
		\item[(iii)] If $Am^{-\frac{h+1}{2h}}\le p\le m^{-1/2}$,   where $h=|V(H)|$ and $A\ge 1$,    then 
		$$
		\pi(H)=(1+O(A^{-1}))\pi(H,  E(H))=(1+O(A^{-1}+m^{-\delta_h}))\pp^{|E(H)|},  
		$$
		where $\delta_h=1/({h2^{h+1}})$.
	\end{itemize}

\end{lemma}

\begin{proof}
	First set $C\in \cover\in \covers(H)$ such that $|C|=k\ge 3$. Denote $E_C=|E(H[C])|$. By \eqref{eq:CC2},  for $mp^2\le 1$, 
	\begin{align}\label{Eq:Splitting:main}
		\frac{\pi(H,   \cover)}
		{\pi(H,   \cover[\{C\}\mapsto E(H[C])])}
		\asymp
		\frac{mp^{k}}{(mp^2)^{E_C}}=\left(\frac{m^{-\frac{E_C-1}{2E_C-k}}}{p}\right)^{2E_C-k}=\left(\frac{m^{f(E_C)}}{p}\right)^{2E_C-k}, 
	\end{align}
	where
	$
	f(x):=-\frac{x-1}{2x-k}=-\frac{1}{2}-\frac{k-2}{4x-2k}.
	$
	We will use the following properties of the function $f$
	\begin{equation}
		\text{$f(x)$ is increasing for $x > k/2$}\quad\text{ and }\quad f\left(\binom{k}{2}\right)=-\frac{k+1}{2k}.
	\end{equation} 
	Furthermore,  note that by {\Ptwo} the graph $H[C]$ is connected, and recall the assumption $k\ge 3$.   This gives us
	\begin{equation}\label{Eq:Splitting:positive}
		2E_C-k\ge 2(k-1)-k = k-2\ge 1.
	\end{equation}
	
	The assumptions of $(i)$ imply $E_C=\binom{k}{2}$ and $p\le  m^{-\frac{k+1}{2k}}= m^{f(\binom{k}{2})}$,  which yields
	\begin{equation}
		\frac{m^{f(E_C)}}{p}\ge 1. 
	\end{equation}
	Consequently,  $(i)$ follows by \eqref{Eq:Splitting:main} and \eqref{Eq:Splitting:positive}.
	
	Now let us consider the case $(ii)$. $E_C\le \binom{k}{2}$ and,  since by {\Ptwo} $H[C]$ is connected,  we have $E_C\ge k-1>k/2$ for $k\ge 3$. Therefore, by monotonicity of $f$ and the assumption $p\ge m^{-\frac{k+1}{2k}}=m^{f(\binom{k}{2})}$,  we obtain
	$$
	\frac{m^{f(E_C)}}{p}\le \frac{m^{f(\binom{k}{2})}}{p}\le 1.
	$$
	Applying this and \eqref{Eq:Splitting:positive} to \eqref{Eq:Splitting:main}, we get $(ii)$.
	
	Under the assumptions from $(iii)$ we have $p\ge Am^{f(\binom{h}{2})}$,  where $h=|V(H)|$. For any $\cover\in \covers(H)\setminus \{E(H)\}$ and any $C\in \cover$ of size $k=|C|$,  $3\le k\le h$,  by monotonicity of $f$ we have
	\begin{equation}
		\frac{m^{f(E_C)}}{p}\le \frac{m^{f(\binom{k}{2})}}{Am^{f(\binom{h}{2})}}\le A^{-1}.
	\end{equation}
	Therefore,  applying \eqref{Eq:Splitting:main} and \eqref{Eq:Splitting:positive} inductively for each $C$,  $|C|\ge 3$,  we get that  for any $\cover\in \covers(H)\setminus\{E(H)\}$ and  $A\ge 1$
	\begin{equation}
		\pi(H, \cover)\lesssim A^{-1} \pi(H, E(H)).
	\end{equation}
	This combined with \eqref{Eq:SubgraphProbabilityCCfamily} proves the first equality in $(iii)$. The second equality follows then by \eqref{eq:CC3}.
\end{proof}

\begin{figure}
	\begin{center}
		\includegraphics[width=10cm]{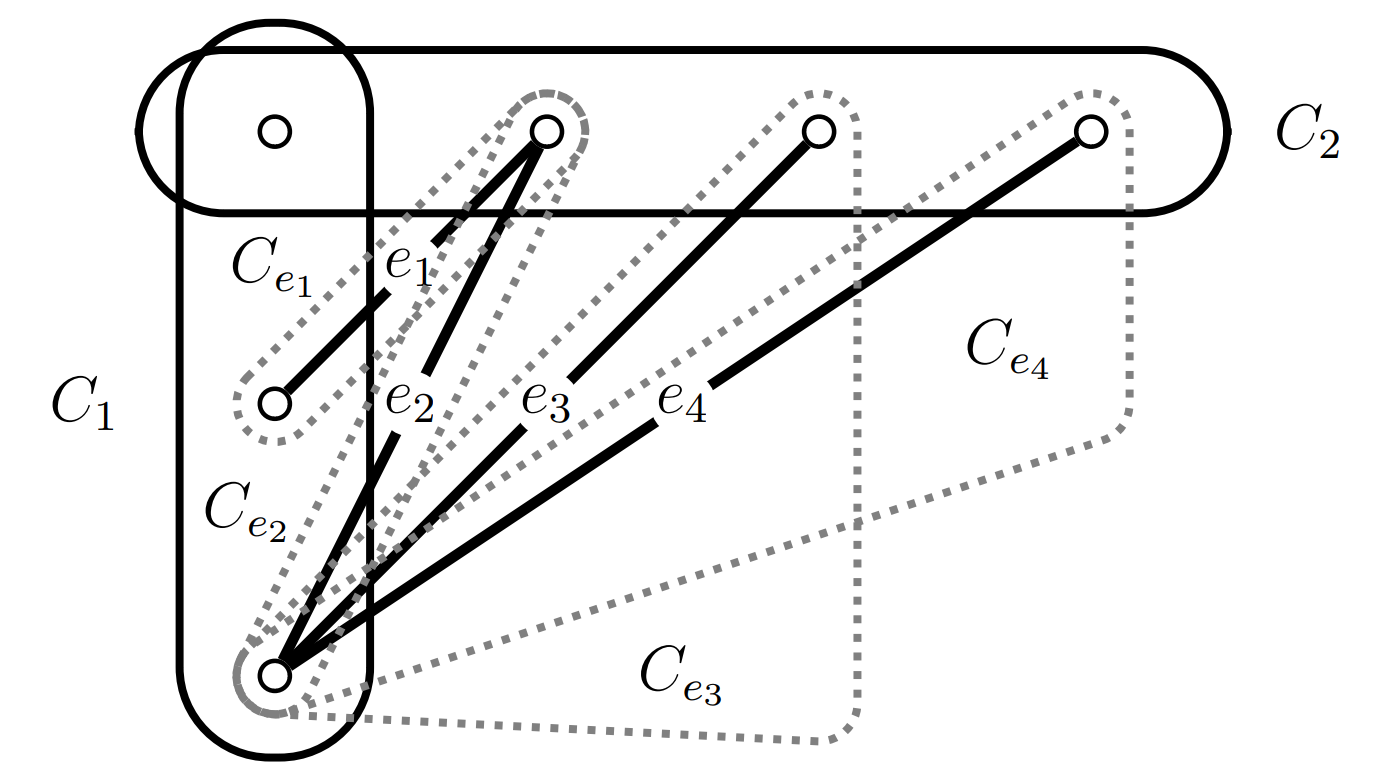}
	\end{center}
	\caption{\label{Fig:Glueing1} An example of $\FF(C_1,  C_2)=\{e_1,  e_2,  e_3,  e_4\}$ with $C_{e_i}$,   $i=1,  2,  3,  4$.}
\end{figure}

The next lemma allows for a more refined comparison between the probability of a clique cover $\cover$ and the probability of a clique cover in which we replace some sets of $\cover$ (not necessarily 2-sets) by their union.

\begin{lemma}\label{Lem:Glueing} 
	Let $\cover\in \covers(H)$ be a simple clique cover of $H$ and $C_1,   C_2\in \cover$ be such that 
	$|C_1\cap C_2|=1$. Let
	$$
	\FF(C_1,   C_2):=\{e\in E(H): |e\cap C_1|=|e\cap C_2|=1\text{ and }C_1\cap C_2\cap e = \emptyset \}.
	$$
	If $p\le m^{-\frac{k+1}{2k}}$ 
	and $|\FF(C_1,   C_2)|\ge (k-1)/2$ for some $k\ge1$,    then
	there exists a simple clique cover $\cover'\in\covers(H)$ such that
	\begin{itemize}
		\item[(i)] $\pi(H,   \cover)\lesssim \pi(H,   \cover')$;
		\item[(ii)] $\exists_{C'\in\cover'}$ $C_1\cup C_2\subseteq C'$;
		\item[(iii)] $|\cover'|\le |\cover|-1$.
	\end{itemize}
\end{lemma}
For an example of a configuration described in the lemma see Figure~\ref{Fig:Glueing1}.
\begin{proof} For each $e\in \FF := \FF(C_1,   C_2)$,   
	denote by $C_e$ the set that covers $e$ in $\cover$. Due to  {\Pone} the set $C_e$ is unique for any $e\in E(H)$ (since each edge is covered by exactly one set). Moreover,  the function $e\mapsto C_e$,   $e\in \FF$,   defined in this manner is an injection. Indeed,   if there were distinct  $e_1,   e_2\in \FF$ such that $C_{e_1}=C_{e_2}$ then $C_{e_1}$ would intersect $C_1$ or $C_2$ on at least two vertices (see Figure~\ref{Fig:Glueing1}). This contradicts {\Pone}. Moreover,   again by {\Pone},   each $C_e$ intersects $C_i$,   $i=1,  2$,   on exactly one vertex.  
	
	Let $$C_\FF=C_1\cup C_2\cup \bigcup_{e\in \FF} C_e\quad\text{ and }\quad\cover_\FF=\{C_1,  C_2\}\cup \{C_e,    e\in \FF\}.$$ 
	Then,   using all the observations given above,   we obtain
	\begin{align*}
		|C_\FF|=\left|
		C_1\cup C_2\cup \bigcup_{e\in \FF} C_e
		\right|
		&\le
		|C_1|+|C_2|-1+\sum_{e\in \FF}(|C_e|-2)\\
		&=
		|C_1|+|C_2|-1-2|\FF| +\sum_{e\in \FF}|C_e|.
	\end{align*}
	Let $\cover'=\cover[\cover_\FF \mapsto \{C_\FF\}]\in \covers(H)$.
	Since for $p\le m^{-\frac{k+1}{2k}}$ it holds that $mp^2\leq1$,     the formula \eqref{eq:CC2} gives us 
	\begin{align*}
		\frac
		{\pi(H,   \cover)}
		{\pi(H,   \cover')}
		\asymp 
		\frac
		{mp^{|C_1|}mp^{|C_2|}\prod_{e\in \FF}mp^{|C_e|}}
		{mp^{\left|C_\FF\right|}}
		\le m^{|\FF|+1}p^{2|\FF|+1}
		\le \left(m^{\frac{|\FF|+1}{2|\FF|+1}}m^{-\frac{k+1}{2k}}\right)^{2|\FF|+1}\le 1,   
	\end{align*}
	where the last inequality is a consequence of the assumption $|\FF|\ge (k-1)/2$ and the fact that  $\frac{k+1}{2k}=\frac{\frac{k-1}{2}+1}{2\frac{k-1}{2}+1}$ and $f(x)=\frac{x+1}{2x+1}=\frac{1}{2}+\frac{1}{2(2x+1)}$ is a decreasing function of $x>0$. 
	Since $\cover$ has property {\Ptwo},    all graphs $H[C]$,   $C\in\cover_\FF$, are connected. Moreover each set in $\cover_\FF$ intersect $C_1$ and $C_2$,   therefore $H[C_\FF]$ is also connected and $\cover'$ has {\Ptwo}. Obviously, $\cover'$ fulfills (ii) and (iii). If $\cover'$ does not have {\Pone},    then the lemma follows by Lemma~\ref{Lem:Simple}~(ii) applied to $\cover'$.
\end{proof}
The following straightforward corollary of the previous lemma will play a key role when considering  $mp\le 1$.  For an alternative proof, see \cite{SubgraphProb}.
\begin{corollary}\label{Cor:Glueing}
	Let $mp\le 1$ and $H$ be a connected graph. Then
	$$\pi(H)\asymp \pi(H,   \{V(H)\})\asymp mp^{|V(H)|}.$$
\end{corollary}
\begin{proof}
	Let $\cover\in \covers(H)$ be a simple clique cover such that $|\cover|\ge 2$. Therefore,    since $H$ is connected,    there are $C_1$ and $C_2$ in $\cover$ that intersect on exactly one vertex. By Lemma~\ref{Lem:Glueing} with $k=1$,    there exists a simple $\cover'\in\covers(H)$ such that $|\cover'|\le |\cover|-1$ and $\pi(H,   \cover)\lesssim\pi(H,   \cover')$. Hence,   by induction,   we may prove that there exists simple $\cover''\in\covers(H)$ such that $|\cover''|=1$ and $\pi(H,   \cover)\lesssim\pi(H,   \cover'')$. Since $H$ is connected,    it does not have isolated vertices,    and therefore  $\cover''=\{V(H)\}$. Summing up,   for all $\cover\in \covers(H)$ we have $\pi(H,  \cover)\lesssim \pi(H,   \{V(H)\})$. The lemma follows now by \eqref{Eq:SubgraphProb} and \eqref{eq:CC2}.
\end{proof}
Let us note that the assertion of Corollary~\ref{Cor:Glueing} can be easily extended to disconnected graphs. Namely,  the existence probability of such a graph is simply the product of the existence probabilities of its connected components.

In the next lemma,  we will apply Lemma~\ref{Lem:Glueing} to a special type of graphs,  which constitutes a generalization of some graphs $G_{t, \II}$, see \eqref{eq:GtI}.

\begin{lemma}\label{Lem:GlueingBis}
	Let $H$ be a graph such that $\kli_1\oplus_j\kli_2\subseteq H$,   $j\ge 2$. Let $\cover\in \covers(H)$ be its simple clique cover and $C\in\cover$ be such that $p\le m^{-\frac{|C|+1}{2|C|}}$.
	If $$|C\cap V(\kli_1\oplus_j\kli_2)|\ge \frac{|C|+1}{2},$$
	then there exists a simple clique cover $\cover'\in \covers(H)$ such that
	\begin{itemize}
		\item[(i)] $\exists_{C'\in\cover'}\,  C\cup\big(V(\kli_1\cap\kli_2)\big)\subseteq C'$,  
		\item[(ii)] $\pi(H,   \cover)\lesssim \pi(H,   \cover')$.
	\end{itemize}
\end{lemma}
\begin{figure}
	\begin{center}
		\includegraphics[width=12cm]{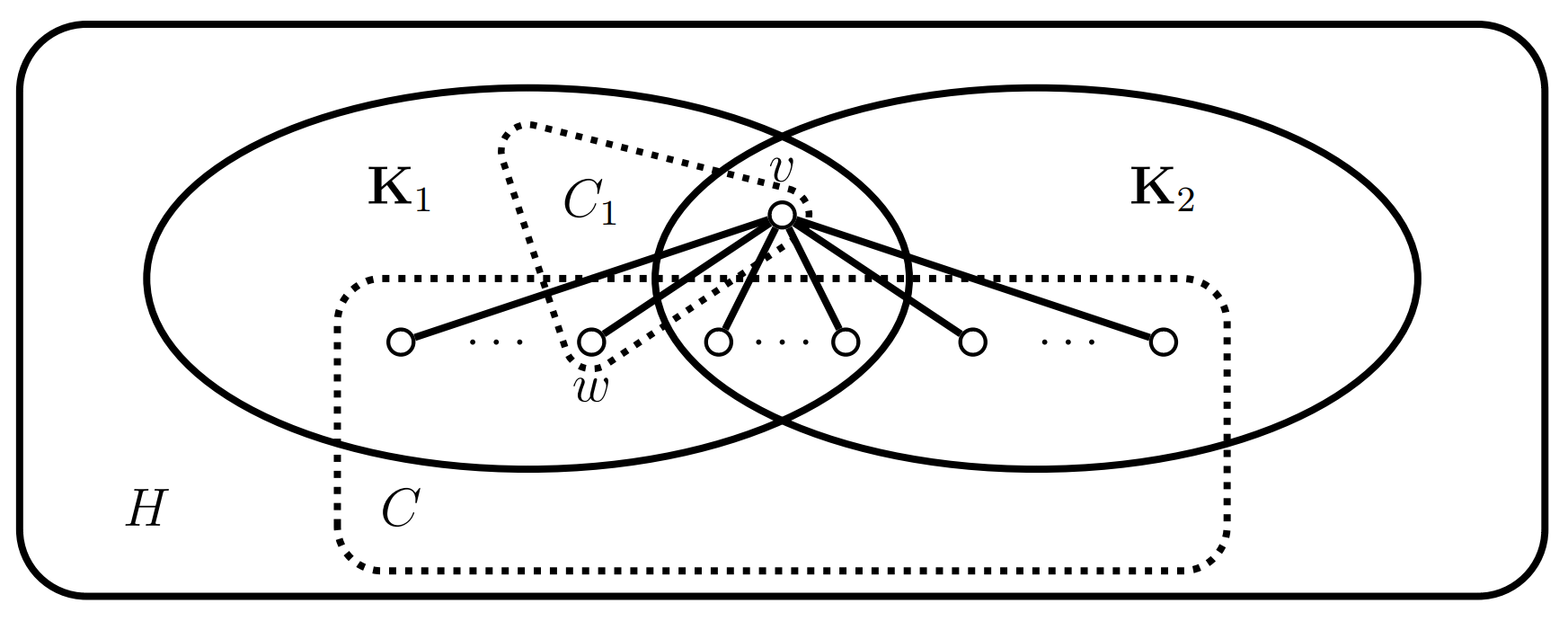}
	\end{center}
	\caption{\label{Fig:Glueing2} The graph $H$ with its subgraph $\kli_1\oplus_j\kli_2$ and the sets $C$ and $C_1$,   as defined in the proof of Lemma~\ref{Lem:GlueingBis}.}
\end{figure}

\begin{proof}
	If $V(\kli_1\cap \kli_2)\subseteq C$,    then the lemma follows. Assume that   there exists $v\in( V(\kli_1\cap\kli_2))\setminus C$,   and
	let us fix $w\in C\cap V(\kli_1\oplus_j\kli_2)$. Since $v$ is connected by an edge with all the vertices from $ V(\kli_1\oplus_j\kli_2)$,   we have $vw\in E(H)$. The clique cover $\cover$ is simple,   thus there exists exactly one $C_1\in \cover$ that contains the edge $vw$ and it satisfies $|C\cap C_1|=1$. 
	Define $\FF(C,   C_1)$ as in Lemma~\ref{Lem:Glueing}. Since $v$ is connected by an edge with all the vertices from $V(\kli_1\oplus_j\kli_2)$,   we have (see Figure~\ref{Fig:Glueing2})
	$$\FF(C,   C_1)\supseteq \{vz: z\neq w,   z\in C\cap V(\kli_1\oplus_j\kli_2)\}.$$ 
	By the assumption $|C\cap V(\kli_1\oplus_j\kli_2)|\ge \frac{|C|+1}{2}$,   we have
	\begin{align*}
		|\FF(C,   C_1)|&\ge |C\cap V(\kli_1\oplus_j\kli_2)|-1\ge \frac{|C|-1}{2}.
	\end{align*}
	Consequently,    by Lemma~\ref{Lem:Glueing} there is a simple $\cover'\in\covers(H)$ with $C'\in\cover'$,    such that $ C\cup \{v\}\subseteq C\cup C_1\subseteq C'$ and $\pi(H,   \cover)\lesssim \pi(H,   \cover')$. Next,  we apply the reasoning to each of $v\in( V(\kli_1\cap\kli_2))\setminus C$ until we get $\cover''\in \covers(H)$ with $C''\in \cover''$ such that $C\cup\big(V(\kli_1)\cap V(\kli_2)\big)\subseteq C''$ and $\pi(H,   \cover)\lesssim \pi(H,   \cover'')$.
\end{proof}

\begin{remark}\label{rem:dominatingccs}
	To understand the next two lemmas (Lemma~\ref{Lem:Division} and Lemma~\ref{Lem:Transition}),  it is important to note that even if the value of $\pi(H_1)$ is dominated by the probability of one particular clique cover,  for example $\pi(H_1) = (1+o(1)) \pi(H_1,  E(H_1))$,  it may happen that $\pi(H_1 \cup H_2)$ is dominated by a clique cover that induces on $H_1$ a clique cover of a completely different structure. For example,  if $\pi(H_1) = (1+o(1)) \pi(H_1,  E(H_1))$, it may happen that  $\pi(H_1\cup H_2)=(1+o(1))\pi(H_1\cup H_2, \{V(H_1\cup H_2)\})$. Nevertheless,  even in such cases,  we can sometimes bound $\pi(H_1\cup H_2)$ in terms of $\pi(H_1)$ and $\pi(H_2)$,  in a manner reminiscent of properties of $\Gnp$.
\end{remark}

The next two lemmas allow us to estimate the probability of a union of graphs in terms of the probabilities of the graphs that constitute this sum. Together with Lemma~\ref{Lem:GlueingBis},  they will form the key part of the proof of Lemma~\ref{Lem:Odrywanie}.  Lemma~\ref{Lem:Division} will also be useful for analysing the probabilities of the graphs 
$G_{t}$ in the cases when some intersecting cliques  $\kli_{ab}$ in $G_{t}$ share exactly one vertex.

\begin{lemma}\label{Lem:Division}
	Let $mp\ge 1$,  $mp^2\le 1$,   $H_1,  H_2$  be graphs,   and $\kli$ be a clique on $r$ vertices. 
	Then,
	\begin{itemize}
		\item[(i)] for a simple $\cover\in\covers(H_1\oplus_1 H_2)$
		$$
		\pi(H_1\oplus_1 H_2,   \cover)\lesssim  \pi(H_1,   \cover[V(H_1)])\pi(H_2,   \cover[V(H_2)]).
		$$
		\item[(ii)]  for a simple $\cover\in\covers(H_1\oplus_{1,  1} H_2)$  
		$$
		\pi(H_1\oplus_{1,  1} H_2,   \cover)\lesssim  \frac{\pi(H_1,   \cover[V(H_1)])\pi(H_2,   \cover[V(H_2)])}{mp^2}.
		$$
		\item[(iii)] if $H=H_1\oplus_{i} H_2$ or $H=H_1\oplus_{i,   j} H_2$,    for $i,   j\ge 1$,   $i+j\le r$,  $H_1=\kli$ and  $\cover\in\covers(H)$ is simple,   
		$$
		\pi(H,   \cover)\lesssim \frac{\pi(H_1,   \cover[V(H_1)])\pi(H_2,   \cover[V(H_2)])}{\pi(H_1[V(H_1\cap H_2)],   \cover[V(H_1\cap H_2)])}.
		$$ 	
	\end{itemize}
\end{lemma}
\begin{proof}
	Let $H=H_1\oplus_{i} H_2$ or $H=H_1\oplus_{i,   j} H_2$,    and $\cover\in \covers(H)$ be a simple clique cover. Furthermore,  let $V_s=V(H_s)$,   $s\in \{1,  2\}$. Note that for $C\in\cover$,  if $H_1[C]$ and $H_2[C]$ each contains an edge,  then  {\Ptwo} implies   $|C\cap V(H_1)\cap V(H_2)|\ge1$. Otherwise $H[C]$ would not be connected. Therefore $\cover$ may be divided into four disjoint subfamilies $\cover=\cover_1\dot{\cup}\,    \cover_2\dot{\cup}\,    \cover_{\cap 1}\dot{\cup}\,    \cover_{\cap \ge 2}$,    where 
	\begin{align*}
		\cover_1&=\{C\in \cover: C\text{ does not contain any edge of }H_2\},   \\
		\cover_2&=\{C\in \cover: C\text{ does not contain any edge of }H_1\},   \\
		\cover_{\cap 1}&=\{C\in \cover: C\text{ contains edges of both $H_1$ and $H_2$,     and }|C\cap V_1\cap V_2|=1\},   \\
		\cover_{\cap \ge 2}&=\{C\in \cover: C\text{ contains edges of both $H_1$ and $H_2$,    and }|C\cap V_1\cap V_2|\ge 2\}.
	\end{align*}
	$(i)$ First,  consider $H=H_1\oplus_1 H_2$. In this case $\cover_{\cap \ge 2}=\emptyset$,   because $|V_1\cap V_2|=1$. By \eqref{eq:CC2},   for  $mp\ge 1$ we get
	\begin{align*}
		\pi(H,   \cover)
		&\asymp
		\prod_{C\in \cover_1} mp^{|C|}
		\prod_{C\in \cover_2} mp^{|C|}
		\prod_{C\in \cover_{\cap 1}} mp^{|C\cap V_1|+|C\cap V_2|-1}\\
		&\le 
		\frac{\prod_{i=1,   2}
			\left(
			\prod_{C\in \cover_i} mp^{|C\cap V_i|}
			\prod_{C\in \cover_{\cap 1}} mp^{|C\cap V_i|}
			\right)
		}
		{\prod_{C\in \cover_{\cap 1}} mp}\\
		&\lesssim 
		\pi(H_1,   \cover[V_1])\pi(H_2,   \cover[V_2]). 
	\end{align*}
	$(ii)$ Now let $H=H_1\oplus_{1,  1} H_2$. In this case $|V_1\cap V_2|=2$ and there is at most one set $C\in\cover$ such that $|V_1\cap V_2\cap C|\ge 2$. Otherwise,   there would be two sets in $\cover$ that intersect on $V_1\cap V_2$,   i.e., on more than one vertex. That contradicts {\Pone}. Thus $|\cover_{\cap \ge 2}|\le 1$. Therefore,   by  \eqref{eq:CC2},   for  $mp\ge 1$ and $mp^2\le 1$ we get
	\begin{align*}
		\pi(H,   \cover)
		&\asymp
		\prod_{C\in \cover_1} mp^{|C|}
		\prod_{C\in \cover_2} mp^{|C|}\prod_{C\in \cover_{\cap 1}} mp^{|C\cap V_1|+|C\cap V_2|-1}
		\prod_{C\in \cover_{\cap \ge 2}} mp^{|C\cap V_1|+|C\cap V_2|-2}\\
		&\lesssim
		\frac{
			\prod_{i=1,   2}\left(\prod_{C\in \cover_i} mp^{|C\cap V_i|}\prod_{C\in \cover_{\cap 1}} mp^{|C\cap V_i|}\prod_{C\in \cover_{\cap \ge 2}} mp^{|C\cap V_i|}\right)
		}{\prod_{C\in \cover_{\cap 1}} mp \prod_{C\in \cover_{\cap \ge 2}} mp^{2}}\\
		&\lesssim 
		\frac{\pi(H_1,   \cover[V_1])\pi(H_2,   \cover[V_2])}{mp^2}. 
	\end{align*}
	$(iii)$ Now let us consider $H=\kli\oplus_{i}H_2$ or $H=\kli\oplus_{i,  j}H_2$,   i.e. $H_1=\kli$. Note that in this case $\kli[V_1\cap V_2]$ is a clique on $i$ or $i+j$ vertices,   respectively. By {\Pone} each edge of $\kli[V_1\cap V_2]$ is covered by exactly one set of $\cover$. Therefore,   any two sets in $\{C\cap V_1\cap V_2: C\in \cover_{\cap \ge 2}\}$ are distinct. 
	
	We define a subset of $\cover_1$ of those sets $C\in \cover$ that intersect $V_1\cap V_2$ on at least $2$ vertices but cover no edges of $H_2$:
	\begin{align*}
		\cover_{1, \cap\ge 2}&=\{C\in \cover_1: |C\cap V_1\cap V_2|\ge 2\}.
	\end{align*}
	Moreover,  since $H_1=\kli$ is a clique,  for any $C\in \cover_2$ we have $|C\cap V_1|\le 1$,  otherwise $C$ would contain an edge of $H_1$. Hence,  we have
	$$\cover[V_1\cap V_2]=\{C\cap V_1\cap V_2: C\in \cover_{\cap\ge2}\cup \cover_{1,  \cap\ge2}\}.$$
	Combining this with  \eqref{eq:CC2},  we arrive at
	\begin{align*}
		\pi(H,   \cover)
		&\lesssim
		\prod_{C\in \cover_{1}} mp^{|C\cap V_1|}
		\prod_{C\in \cover_2} mp^{|C\cap V_2|}\\
		&\hspace{.5cm}
		\prod_{C\in \cover_{\cap \ge 2}} mp^{|C\cap V_1|+|C\cap V_2|-|C\cap V_1\cap V_2|}
		\prod_{C\in \cover_{\cap 1}} mp^{|C\cap V_1|+|C\cap V_2|-1}\\[5pt]
		&=
		\frac
		{
			\prod_{i=1,   2}\left(\prod_{C\in \cover_i} mp^{|C\cap V_i|}\prod_{C\in \cover_{\cap 1}} mp^{|C\cap V_i|}\prod_{C\in \cover_{\cap \ge 2}} mp^{|C\cap V_i|}\right)
		}
		{\prod_{C\in \cover_{\cap \ge 2}} mp^{|C\cap V_1\cap V_2|}\prod_{C\in \cover_{\cap 1}} mp\ \prod_{C\in \cover_{ 1, \cap\ge 2}} mp^{|C\cap V_1\cap V_2|}	}
		\prod_{C\in \cover_{ 1, \cap\ge 2}} mp^{|C\cap V_1\cap V_2|}
		\\[5pt]
		&\asymp
		\frac{\pi(\kli,   \cover[\kli])\pi(H_2,   \cover[V_2])}{\pi(\kli[V_1\cap V_2],   \cover[V_1\cap V_2])}\cdot
		\frac{\prod_{C\in \cover_{ 1, \cap\ge 2}} mp^{|C\cap V_1\cap V_2|}}
		{\prod_{C\in \cover_{\cap 1}} mp}.
	\end{align*}
	We note that for  $C\in \cover_{1, \cap\ge 2}$ we have ${|C\cap V_1\cap V_2|}\ge 2$. Thus,  due to the assumptions $mp\ge1$ and $mp^2\le1$,  the last fraction is bounded from above by 1,  which completes the proof.
\end{proof}

\begin{lemma}\label{Lem:Transition}
	Assume $mp^2\le 1$. Let $H$ and $\bar H$ be graphs such that
	$$
	H=\kli \oplus_j \bar{H}\quad\text{or}\quad H=\kli \oplus_{j_1,   j_2} \bar{H},   \quad\text{ for some }j,   j_1,   j_2\ge 2.
	$$
	Moreover,  fix a simple clique cover  $\cover\in \covers(H)$. If at least one of the following conditions is fulfilled 
	\begin{itemize}
		\item[(i)] 
		$\forall_{C\in \cover}\ p\ge m^{-\frac{|C|+1}{2|C|}}$, 
		\item[(ii)] 
		$
		\exists_{C\in\cover}\  V(\kli\cap \bar H)\subseteq C;
		$
	\end{itemize}
	then    it holds that 
	$$
	n^{|V(H)|}\pi(H,   \cover)\lesssim \frac{n^r\qq\cdot n^{|V(\bar H)|}\pi(\bar H)}{\ble}.
	$$
\end{lemma}
\begin{proof}
	Without loss of generality,  we may set $j_1+j_2=j$. Then,   in the both definitions  of $H$,   $\kli[V(\bar{H}\cap \kli)]$ is a clique on $j$ vertices.

	Let us first assume that (i) is fulfilled. Then by Lemma~\ref{Lem:Splitting}(ii) we have 
	$$
	\pi(H,   \cover)\lesssim \pi(H,   E(H)). 
	$$	
	We consider two cases. For $p\le m^{-\frac{r+1}{2r}}$,   using first Lemma~\ref{Lem:Splitting}(i) and  Lemma~\ref{Lem:Division}(iii) and  then, in the last line, the fact that $\kli[V(\kli\cap\bar H)]$ is a clique on $j$ vertices,   \eqref{Eq:Def:Blei},  and \eqref{Eq:BleaBlei},  we obtain
	\begin{align*}
		n^{|V(H)|}\pi(H,   \cover)&\lesssim n^{|V(H)|}\pi(H,   E(H)) \\
		&\lesssim n^{|V(H)|}\pi(H,   \{V(\kli)\}\cup (E(\bar H)\setminus E(\kli\cap\bar H)))\\
		&\lesssim
		\frac{n^r\pi(\kli,   \{V(\kli)\})n^{|V(\bar H)|}\pi(\bar H,   \{V(\kli\cap\bar H)\}\cup (E(\bar H)\setminus E(\kli\cap\bar H))}
		{n^{|V(\kli\cap\bar H)|}\pi(\kli[V(\kli\cap\bar H)],   \{V(\kli\cap\bar H)\})}
		\\
		&\lesssim
		\frac{n^{r}\qq\cdot n^{|V(\bar H)|}\pi(\bar H)}{\ble_j}\lesssim \frac{n^{r}\qq\cdot n^{|V(\bar H)|}\pi(\bar H)}{\ble}.
	\end{align*}

	Similarly,   for $m^{-\frac{r+1}{2r}}\le p\le m^{-1/2}$,   using Lemma~\ref{Lem:Division}(iii),  the fact that $\kli[V(\kli\cap\bar H)]$ is a clique on $j$ vertices,   \eqref{Eq:Def:Blei},  and \eqref{Eq:BleaBlei}  we obtain
	\begin{align*}
		n^{|V(H)|}\pi(H,   \cover)&\lesssim n^{|V(H)|}\pi(H,   E(H)) 
		\lesssim
		\frac{n^r\pi(\kli,   E(\kli))n^{|V(\bar H)|}\pi(\bar H,   E(\bar H))}
		{n^{|V(\kli\cap\bar H)|}\pi(\kli[V(\kli\cap\bar H)],   E(\kli[V(\kli\cap\bar H)]))}
		\\
		&\lesssim
		\frac{n^{r}\qq\cdot n^{|V(\bar H)|}\pi(\bar H)}{\ble_j}\le \frac{n^{r}\qq\cdot n^{|V(\bar H)|}\pi(\bar H)}{\ble}.
	\end{align*}

	We shall now proceed to the case when assumption (ii) is fulfilled.  
	In this case $\cover[V(\kli\cap \bar{H})]=\{V(\kli\cap \bar{H})\}$.
	First, we assume that $p\le m^{-\frac{r+1}{2r}}$ or $np\ge 1$. 
	
	For $p\le m^{-\frac{r+1}{2r}}$,  by \eqref{Eq:Def:Blei} and \eqref{Eq:BleaBlei}
	$$
	n^{|V(\kli\cap\bar H)|}\pi(\kli[V(\kli\cap\bar H)],   \{V(\kli\cap\bar H)\})\asymp \ble_j\gtrsim \ble.$$ 
	On the other hand,   for $np\ge 1$,    by \eqref{eq:CC2},  \eqref{Eq:Def:Blei},  and \eqref{Eq:BleaBlei} 
	$$
	n^{|V(\kli\cap\bar H)|}\pi(\kli[V(\kli\cap\bar H)],   \{V(\kli\cap\bar H)\})\asymp n^jmp^j\ge n^2mp^2 \asymp \ble_2 \gtrsim \ble.
	$$
	Therefore,   by Lemma~\ref{Lem:Division}(iii)  
	\begin{align*}
		n^{|V(H)|}\pi(H,   \cover)
		&\lesssim
		\frac{n^r\pi(\kli,   \cover[V(\kli)])n^{|V(\bar H)|}\pi(\bar H,   \cover[V(\bar H)]}
		{n^{|V(\kli\cap\bar H)|}\pi(\kli[V(\kli\cap\bar H)],   \{V(\kli\cap\bar H)\})}
		\\
		&\lesssim
		\frac{n^r\qq\cdot n^{|V(\bar H)|}\pi(\bar H)}
		{\ble}.
	\end{align*}
	
	Now we assume (ii),   $np\le 1$,   and $m^{-\frac{r+1}{2r}}\le p\le m^{-1/2}$. The last assumption implies that $mp\ge 1$.
	Set 
	$$C_{\kli}=C\cap V(\kli)\text{ and }C_{\bar H}=C\cap V(\bar H).$$  Recall that $\cover$ is simple. Therefore $C$ is the only set in $\cover$ that contains edges of $\kli[V(\kli\cap\bar H)]$,    otherwise $\cover$ would not have {\Pone}. Let $C'\in \cover\setminus \{C\}$. If $C'$ covers only edges of one of $\kli$ or $\bar H$,  then either $C'\cap V(\kli)$ is included in $\cover[\kli]$ or $C'\cap V(\bar H)$ is included in $\cover[\bar H]$. If $C'$ covers edges of both $\kli$ and $\bar H$. Then by {\Ptwo} $|C'\cap V(\kli\cap \bar H)|=1$ and
	$$
	|C'|=|C'\cap V(\kli)|+|C'\cap V(\bar H)|-|C'\cap V(\kli\cap \bar H)|=|C'\cap V(\kli)|+|C'\cap V(\bar H)|-1.
	$$
	Let $c$ be the number of sets in $\cover\setminus\{C\}$ that cover edges of both $\kli$ and $\bar H$. 
	Then,  employing the above observations,  by \eqref{eq:CC2},   since $mp\ge 1$,  
	\begin{align}\nonumber
		\pi(H,   \cover)
		&\asymp
		mp^{|C|}\prod_{C'\in \cover\setminus\{C\}}mp^{|C'|}\\ \nonumber
		&\le 
		mp^{|C|}\frac{\prod_{C'\in \cover[V(\kli)]\setminus\{C_\kli\}}mp^{|C'|}\prod_{C''\in \cover[V({\bar H})]\setminus\{C_{\bar H}\}}mp^{|C''|}}{(mp)^{c}}\\ \nonumber
		\\
		&\le 
		mp^{|C|}\prod_{C'\in \cover[V(\kli)]\setminus\{C_\kli\}}mp^{|C'|}\prod_{C''\in \cover[V({\bar H})]\setminus\{C_{\bar H}\}}mp^{|C''|}\label{Eq:Lem:Division1}
		.
	\end{align}
	Let us consider the first product in \eqref{Eq:Lem:Division1}. By \eqref{eq:CC2} we get
	\begin{align*}
		\prod_{C'\in \cover[V(\kli)]\setminus\{C_\kli\}}mp^{|C'|}&=\frac{\prod_{C'\in \cover[V(\kli)]\setminus\{C_\kli\}}mp^{|C'|}\prod_{\{i,   j\}\in C_\kli,   i\neq j}mp^2}{(mp^2)^{\binom{|C_{\kli}|}{2}}}\\
		&\asymp \frac{\pi\big(\kli,   (\cover[V(\kli)]\setminus \{C_{\kli}\})\cup E(\kli[C_{\kli}])\big)}
		{\pp^{\binom{|C_{\kli}|}{2}}}\leq \frac{\pi(\kli)}{\pp^{\binom{|C_{\kli}|}{2}}}.
	\end{align*}
	Regarding the latter product in \eqref{Eq:Lem:Division1},   in an analogous manner we get
	\begin{align*}
		\prod_{C''\in \cover[V({\bar H})]\setminus\{C_{\bar H}\}}mp^{|C''|}=
		\frac{\prod_{C''\in \cover[V({\bar H})]}mp^{|C''|}}{mp^{|C_{\bar H}|}}
		\asymp \frac{\pi(\bar H,   \cover[V({\bar H})])}{mp^{|C_{\bar H}|}}\leq \frac{\pi(\bar H)}{mp^{|C_{\bar H}|}}. 
	\end{align*}
	Applying the last two estimates to \eqref{Eq:Lem:Division1} and using the equality 
	$$|V(\kli)\cap V(\bar H)|=|(C\cap V(\kli))\cap (C\cap V(\bar H))|=|C_{\bar H}|+|C_{\kli}|-|C|,  $$ we arrive at
	\begin{align*}
		n^{|V(H)|}\pi(H,   \cover)&\lesssim n^{|V(\bar H)|+r-|V(\kli)\cap V(\bar H)|}mp^{|C|}\frac{\pi(\kli)}{\pp^{\binom{|C_{\kli}|}{2}}}\frac{\pi(\bar H)}{mp^{|C_{\bar H}|}}\\
		&\asymp
		\frac{n^r\pi(\kli)\cdot n^{|V(\bar H)|}\pi(\bar H)}
		{n^{|C_{\kli}|}\pp^{\binom{|C_{\kli}|}{2}}}
		(np)^{|C|-|C_{\bar H}|}
		\\
		&\lesssim
		\frac{n^r\qq\cdot n^{|V(\bar H)|}\pi(\bar H)}{\ble},   
	\end{align*}
	where in the last bound we used the assumption $np\le 1$,  \eqref{Eq:Def:Blei} for $i=|C_{\kli}|$,  and \eqref{Eq:BleaBlei}.
\end{proof}

The next lemma completes the preparation for the proof of Lemma~\ref{Lem:Odrywanie}.

\begin{lemma}\label{Lem:Odrywanie:Pre} Let  $j-i\ge 2$ and let $i\ge 0$ for $t=1$ and $i\ge 2$ for  $t=2, 3$. 
	Moreover let $\II\subseteq \indeksy$,  $|\II|\ge 2$ be such that $G_{t, \II}$ is connected. 
	If for a simple $\cover\in \covers(G_{t, \II})$
	\begin{equation}\label{Eq:Odrywanie:Cdef}
		\exists_{C\in \cover}\ p \le m^{-\frac{|C|+1}{2|C|}},
	\end{equation}
	then there exists another simple clique cover $\cover'$,  a set $C'\in \cover$ and  $ab\in \II$ satisfying
	\begin{itemize}
		\item[(i)] $\pi(G_{t, \II},   \cover)\lesssim\pi(G_{t, \II},   \cover')$, 
		\item[(ii)]
		$
		V(G_{t,   \II\setminus \{ab\}}\cap \kli_{ab})\subseteq C', 
		$
		\item[(iii)] $G_{t,   \II\setminus \{ab\}}$ is connected.
	\end{itemize}
\end{lemma}
\begin{proof} 
	We consider the following decompositions
	$
	G_{t,\II} = H_1 \cup H_2,
	$
	chosen so as to allow us to apply Lemma~\ref{Lem:GlueingBis}. Namely, for each such decomposition, we will verify that the assumptions of Lemma~\ref{Lem:GlueingBis} are satisfied for $H' = H_1$ or $H' = H_2$.
	\begin{align}
		G_{1, \indeksy}&=(\kli_{11}\oplus_j\kli_{12})\cup(\kli_{21}\oplus_j\kli_{22}), &&\nonumber \\
		G_{t, \indeksy}&=(\kli_{11}\oplus_{i}\kli_{12})\cup(\kli_{21}\oplus_i\kli_{22}), &&\text{ for }t=2, 3, \nonumber \\
		G_{2, \indeksy}&=(\kli_{12}\oplus_{j-i}\kli_{22})\cup(\kli_{11}\oplus_{j-i}\kli_{21}), &&\nonumber \\
		G_{1, \{11, 12, 21\}}&=(\kli_{11}\oplus_j\kli_{12})\cup(\kli_{11}\oplus_{j-i}\kli_{21}), &&
		\nonumber\\
		G_{t, \{11, 12, 21\}}&=(\kli_{11}\oplus_{i}\kli_{12})\cup(\kli_{11}\oplus_{j-i}\kli_{21}), &&\text{ for }t=2, 3 \nonumber\\
		G_{t, \II}&=(\kli_y \oplus_x \kli_z)\cup(\kli_y \oplus_x \kli_z),  &&\text{ for } \II=\{y, z\},  x\in \{i, j, j-i\}.\label{Eq:Odrywanie:Decomp}
	\end{align} 
	We note that,  for all connected $G_{t, \II}$  with $|\II|=3$,  we have $G_{t, \II}\cong G_{t, \{11, 12, 21\}}$. Therefore, for the case $|I|=3$ only $G_{t, \{11, 12, 21\}}$ are considered. Moreover,  in all the cases presented above,  $H_1\cap H_2$ has at least two vertices.

	For $C$ that exists by \eqref{Eq:Odrywanie:Cdef} and for $G_{t, \II}=H_1\cup H_2$,  we will show that
	\begin{equation}\label{Eq:Odrywanie:Clarge}
		\exists_{l\in \{1, 2\}}\,  |C\cap V(H_l)|\ge \frac{|C|+1}{2}.
	\end{equation} 
	If $C$ is  included in $H_1$ or $H_2$,  the assertion is clear. Let us therefore assume  $C\cap V(H_1),  C\cap V(H_2)\neq \emptyset$. Since $\cover$ is simple,  then $C\cap V(H_1\cap H_2)\neq \emptyset$.  Thus,  by inclusion-exclusion principle we get
	$$|C|=|C\cap V(H_1)|+|C\cap V(H_2)|-|C\cap V(H_1\cap H_2)|\le |C\cap V(H_1)|+|C\cap V(H_2)|-1, $$
	which implies \eqref{Eq:Odrywanie:Clarge}.
	
	We start by considering the most challenging case $\II=\indeksy$ and $t=2$. By the symmetry of $G_{2, \indeksy}$,  without loss of generality we may assume that in the decomposition listed in the second line of \eqref{Eq:Odrywanie:Decomp},  $|C\cap H_1|=|C\cap V(\kli_{11}\oplus_i\kli_{12})|\ge (|C|+1)/2$. Hence,   by Lemma~\ref{Lem:GlueingBis} applied with $H=G_{2, \indeksy}$ and $\kli_{11}\oplus_i\kli_{12}\subseteq G_{2, \indeksy}$,  there exists a simple clique cover $\cover'$ and  $C'\in\cover'$ such that $\pi(G_{2, \indeksy},   \cover)\lesssim\pi(G_{2, \indeksy},   \cover')$,    $C\subseteq C'$ (thus $|C|\le |C'|$, and  hence $p\le m^{-\frac{|C|+1}{2|C|}}\le m^{-\frac{|C'|+1}{2|C'|}}$),    and $V(\kli_{11}\cap \kli_{12})\subseteq C'$. The same reasoning applied to the decomposition listed in the third line of \eqref{Eq:Odrywanie:Decomp} and clique cover $\cover'$ implies (assuming without loss of generality that $|C'\cap V(\kli_{12}\oplus_j\kli_{22})|\ge (|C'|+1)/2$) that there exists a simple clique cover $\cover''$ and  $C''\in\cover''$ such that $\pi(G_{2, \indeksy},   \cover')\lesssim\pi(G_{2, \indeksy},   \cover'')$,    $V(\kli_{11}\cap \kli_{12})\subseteq C'\subseteq C''$,    and $V(\kli_{12}\cap \kli_{22})\subseteq C''$. This means that for such  $\cover''$ and $C''$ it holds that
	\begin{equation}
		V(\kli_{12}\cap G_{2, \indeksy\setminus\{12\}})\subseteq C'', \quad \pi(G_{2, \indeksy},   \cover)\lesssim\pi(G_{2, \indeksy},   \cover'').
	\end{equation}
	Since $G_{2, \indeksy\setminus\{12\}}$ is clearly connected,  the lemma follows in the case $t=2$ and $\II=\indeksy$.
	
	The same reasoning may be applied to the remaining decompositions listed in \eqref{Eq:Odrywanie:Decomp}. Note that in these cases,   $\cover'$ has already required properties and there is no need to consider $\cover''$. This implies the lemma for $|\II|=2$,  $|\II|=3$ (considering $\II=\{11, 12, 21\}$ and $ab=12$ or $ab=21$),  as well as $\II=\indeksy$ with $t=1, 3$ (taking $ab=12$ or $ab=21$). 
\end{proof}

Now we are ready to prove Lemma~\ref{Lem:Odrywanie}.

\begin{proof}[Proof of Lemma~\ref{Lem:Odrywanie}]
	We will apply Lemma~\ref{Lem:Transition}. For that we note that,  for any $|\II|\ge 2$ and any $ab\in \II$ we have $G_{t, \II}=\kli\oplus_{x} \bar{H}$ for $\kli=\kli_{ab}$,  $\bar{H}=G_{t, \II\setminus\{ab\}}$,  and some $x$ from the set $\{i, j, j-i, (i, j-i)\}$. The index $ab$ may always be chosen so that $G_{t, \II\setminus\{ab\}}$ is connected.
	
	By Lemma~\ref{Lem:Odrywanie:Pre} we have that for a clique cover $\cover\in \covers(G_{t, \II})$ 
	either (i) of  Lemma~\ref{Lem:Transition} is true or there exists $ab\in \II$,  $\cover'$,  and $C'\in \cover'$ such that $G_{t, \II\setminus\{ab\}}$ is connected,  $\pi(G_{t, \II},   \cover)\lesssim\pi(G_{t, \II},   \cover')$,  and for $\cover'$ the assumption (ii) of  Lemma~\ref{Lem:Transition} is fulfilled. Thus,  Lemma~\ref{Lem:Odrywanie} follows by Lemma~\ref{Lem:Transition}.
	
\end{proof}

\subsection{Decompositions of clique covers}\label{Sec:CCs:Decopmposition}

\begin{figure}
	\begin{center}
		\includegraphics[width=14cm]{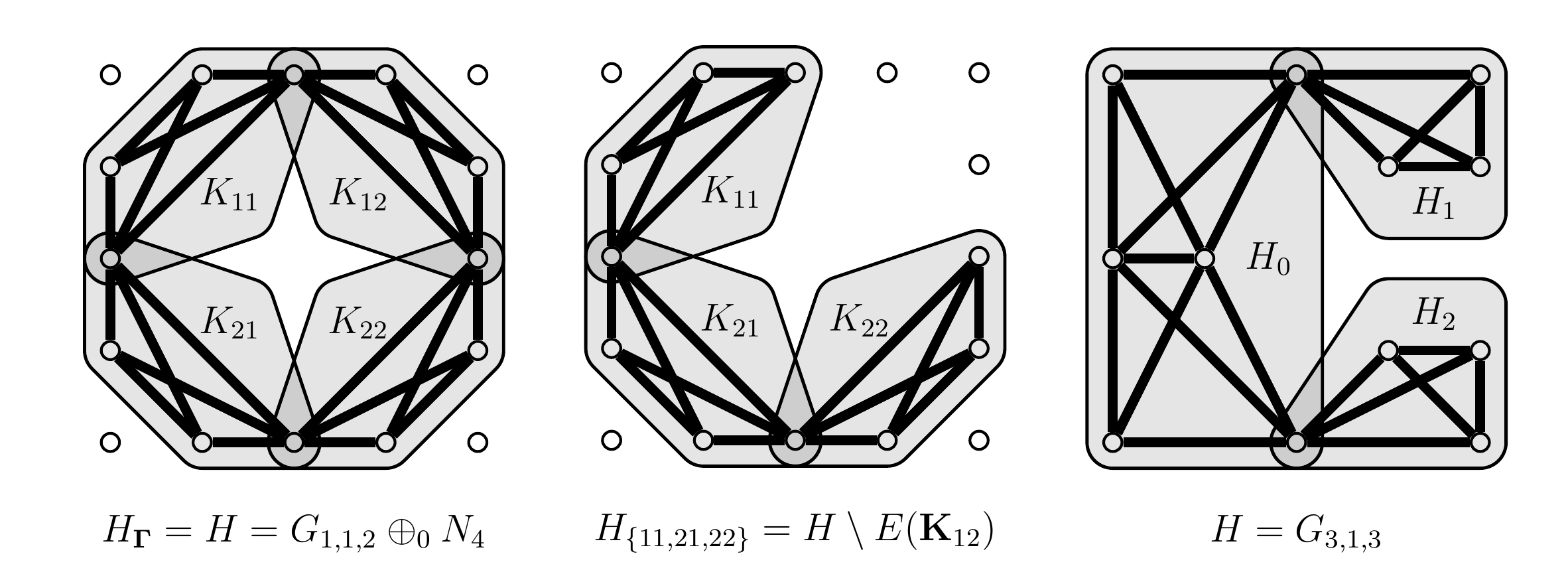}
	\end{center}
	\caption{\label{Fig:Partition}Let $r=4$. 
		$H=G_{2,  1,  2}\oplus_0 \emp_4$ has a natural {\it partition} with $\Lambda=\indeksy$ and $H_{ab}=\kli_{ab}$,   for $ab\in \indeksy$. Also $H_{\{11,  21,  22\}}$ in the picture in the middle has a partition $(\kli_{ab})_{ab\in \{11,  21,  22\}}$.  
		In the third picture we present $G_{3,  1, 3}$ and its partition into $H_0=\kli_{11}\oplus_2\kli_{12}$,   $H_1=\kli_{21}$,   and $H_2=\kli_{22}$. 
	}
\end{figure}

In this section we consider clique covers of the graphs that are unions of graphs that pairwise intersect on at most one vertex. 

\begin{definition}\label{Def:Partition}
	A family $\{H_i\}_{i\in \Lambda}$ of subgraphs of $H$ is called a {\it partition of H} if 
	\begin{itemize}
		\item[(i)] $\forall_{i,   j\in \Lambda,    i\neq j} |V(H_i)\cap V(H_j)|\le 1$;
		\item[(ii)] $E(H)=\bigcup_{i\in \Lambda}E(H_i)$.  
	\end{itemize} 
\end{definition}
For any $J\subseteq \Lambda$,    we define 
\begin{equation}\label{Eq:DecompDefHJ}
	H_{J}=\left(V(H),   \bigcup_{i\in J}E(H_i)\right)=H\setminus \bigcup_{i\in \Lambda\setminus J}E(H_i), \end{equation}
which is a graph created from $H$ by removing edges that do not belong to any $H_i$ with $i\in J$. In particular,   we have $H_{\emptyset}=(V(H),   \emptyset)$.\\

Later on,    we will consider some special partitions of graphs, as e.g.,
\begin{itemize}
	\item for any graph $G$ with the edge set $E(G)=\{e_1, \ldots, e_s\}$ we will take $\{H_i\}_{i\in \{1, \ldots, s\}}$ such that $H_i=G[e_i]$.
	\item for $G_{1, 0, 1}$ and $G_{t, 1, 2}$,  for $t=1, 2$,  we will consider the partition $\{\kli_{ab}\}_{ab\in\indeksy}$.
	\item for $G_{3, 1, j}$ we will analyse the partition $\{H_i\}_{i\in \Lambda}$,    where $\Lambda=\{0,   1,   2\}$,    $ H_0=G_{\{11,   21\}}=\kli_{11}\oplus_{j-1}\kli_{21}$,    $H_1=\kli_{12}$,    and  $H_2=\kli_{22}$.
\end{itemize}
For examples  of partitions we refer the reader  to Figure~\ref{Fig:Partition}.

\begin{definition}\label{def:decomp}
	Let $\{H_i\}_{i\in \Lambda}$  be a partition of $H$ and $\cover\in\covers(H)$. We say that $\cover$ is {\it decomposable} (in relation to the partition $\{H_i\}_{i\in \Lambda}$) if,    for some $i\in \Lambda$,    there exist  $\cover_i\subseteq \cover$ and $\bar \cover_i\subseteq \cover$ such that $\cover_i\in \covers(H_i)$,    $\bar \cover_i\in \covers(H_{\Lambda\setminus \{i\}})$,    and $\cover_i\cap\bar{\cover_i}=\emptyset$.   Otherwise $\cover$ is {\it indecomposable}. Figures~\ref{Fig:Decomposition} and \ref{Fig:Indecomposable2graphs}  illustrate several examples of graphs and their decomposable and indecomposable clique covers.
	
	For  $J\subseteq \Lambda$    we say that $\cover\in \covers(H)$ is {\it decomposable to $H_J$} (see Figure~\ref{Fig:Decomposition}),    if  there are pairwise disjoint subfamilies $\cover_{J}$ and $\cover_i$,    $i\in \Lambda\setminus J$,    of $\cover$ such that $\cover_{J}\cup\bigcup_{i\in \Lambda\setminus J} \cover_{i}\subseteq \cover$,    
	$\cover_i\in \covers(H_i)$,    for $i\in \Lambda\setminus J$,    and $\cover_{J}\in\covers(H_{J})$ is indecomposable with relation to the partition $\{H_i\}_{i\in J}$ of $H_{J}$.
	Furthermore,  
	by $\covers_{J}(H)$ we denote   the family of clique covers of $H$  that are   decomposable to $H_{J}$. \\
	\ \\
	In other words,   $\cover$ is decomposable,   if there are disjoint subfamilies $\cover_i$ and $\bar\cover_i$ of $\cover$,   such that sets of $\cover_i$ are subsets of $V(H_i)$,  and cover all edges of $H_i$ and sets of $\bar\cover_i$ cover all edges of $H\setminus E(H_i)$,   i.e. cover all the remaining edges.  Then,  $\cover$ is decomposable to $H_J$,  if one can separate in such a way every graph $H_i$ for $i\in \Lambda\setminus J$,  and only these graphs. 
	In particular $\cover$ is decomposable to $H_{\emptyset}$ if there exist disjoint $\cover_i\in \covers(H_i)$ for all   $i\in \Lambda$,    such that $\bigcup_{i\in \Lambda}\cover_i\subseteq \cover$. Note that,    in all considered cases of decompositions,    we allow some extra sets in $\cover$,    that  do not need to be included in any $\cover_i$  or $\cover_J$. 
\end{definition}

\begin{figure}
	\begin{center}
		\includegraphics[width=14cm]{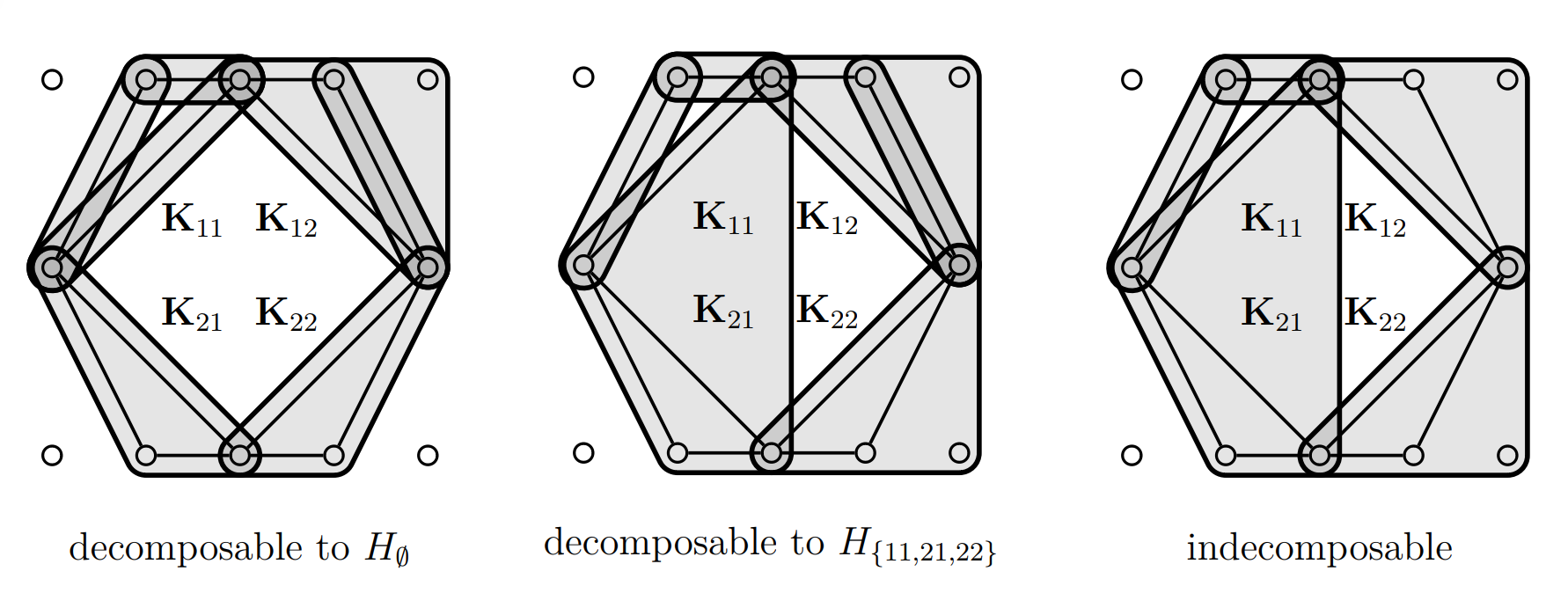}
	\end{center}
	\caption{\label{Fig:Decomposition}Let $r=3$. 
		The picture represents various clique covers of $H=G_{2,  1,  2}\oplus_0 \emp_4$. In all cases we consider partition $\{\kli_{ab}\}_{ab\in \indeksy}$.
		In the first picture the clique cover is decomposable to $H_{\emptyset}$. 
		The extra 3--set that contains an isolated vertex is not contained in any $\cover_{ab}$,   $ab\in \indeksy$. In the second picture the set that covers edges of $\kli_{22}$ is not contained in $V(\kli_{22})$,  therefore the clique cover is not decomposable to $H_{\{11, 21\}}$. For similar reasons the clique cover in the last picture is indecomposable.
	}
\end{figure}

\begin{lemma}\label{Lem:DecompositionComparing}
	Let $H^{(l)}$,    $l=1,   2$,    be two graphs with partitions $\{H_i^{(l)}\}_{i\in \Lambda}$,    $l=1,   2$,  and $J\subseteq \Lambda$. 
	If
	\begin{equation}\label{Eq:IsoDecompComp}
		H^{(1)}_{J}\cong H^{(2)}_{J}\quad\text{and}\quad \forall_{i\in \Lambda\setminus J} H^{(1)}_i\cong H^{(2)}_i,
	\end{equation}
	then 
	$$
	\pi(H^{(1)},    \covers_{J}(H^{(1)}))=\pi(H^{(2)},    \covers_{J}(H^{(2)})).
	$$
\end{lemma}
\begin{proof}
	Let 
	$$\varphi_J:V(H_J^{(1)})\to V(H_J^{(2)})\quad\text{ and } \quad\varphi_i:V(H_i^{(1)})\to V(H_i^{(2)}),   \text{ for } i\in \Lambda\setminus J,   $$ be bijections of vertex sets implied by \eqref{Eq:IsoDecompComp},    i.e., edge-preserving bijections of the respective vertex sets of $H_i^{(l)}$,    $i\in \Lambda\setminus J$,    and $H_J^{(l)}$,    $l=1,   2$. They define bijections
	$$
	\phi_x:\covers(H_x^{(1)})\to \covers(H_x^{(2)}),   \quad x\in \{J\}\cup (\Lambda\setminus J),   
	$$
	given by the formula
	$$
	\phi_x(\{C_1,   \ldots,   C_t\})=\{\varphi_x(C_1),   \ldots,   \varphi_x(C_t)\},   \quad x\in \{J\}\cup (\Lambda\setminus J),    
	$$
	for any $\cover=\{C_1,   \ldots,   C_t\}\in \covers(H_x^{(1)})$. Moreover,   \eqref{Eq:IsoDecompComp} gives us $\cP(H_{x}^{(1)})\Bumpeq \cP(H_{x}^{(2)})$,   where the relation $\Bumpeq$ is defined in \eqref{Bumpeq}. Therefore 
	\begin{equation}\label{Eq:BijectionsSameStructure}
		\forall_{x\in \{J\}\cup (\Lambda\setminus J)}
		\forall_{\cover \in \covers(H_x^{(1)})}\quad
		\cover \Bumpeq \phi_x(\cover)\quad\text{and}\quad
		\cP(H_{x}^{(1)})\setminus \cover\Bumpeq \cP(H_{x}^{(2)})\setminus\phi_x(\cover).
	\end{equation}
	
	Let $\covers_{J}(H_{J}^{(l)})$ be the family of indecomposable clique covers of $H_{J}^{(l)}$ with respect to $(H_i^{(l)})_{i\in J}$. 
	From the definition,    we can conclude that each clique cover of $H^{(l)}$ decomposable to $H_{J}^{(l)}$,    contains an indecomposable clique cover of $H_{J}^{(l)}$ and clique covers of $H_{i}^{(l)}$,    $i\in \Lambda\setminus J$. Thus
	\begin{align}\label{Eq:Decomp:SumOfCC}
		&\covers_{J}(H^{(l)})\\
		&=	
		\bigcup_{\cover_J\in \covers_{J}(H_J^{(l)})}
		\bigcup_{\substack{\cover_i\in \covers(H_i^{(l)})\\   i\in \Lambda\setminus J}}  
		\covers\left(
		H^{(l)},   
		\cover_J\cup\bigcup_{i\in \Lambda\setminus J}\cover_i,   
		\left({\cal P}(H^{(l)}_J)\setminus \cover_J\right) 
		\cup
		\bigcup_{i\in \Lambda\setminus J}\left({\cal P}(H^{(l)}_i)\setminus\cover_i\right)
		\right).
	\end{align}
	We will show that for a fixed $l\in\{1,   2\}$ the families in the union are pairwise disjoint. Without loss of generality,  let $l=1$. Assume that $\covers_{1}$ and $\covers_{2}$ are the families from the union generated by clique covers 
	$$
	\cover_{s,   J}\in \covers_{J}(H_J^{(1)}),    
	\quad
	\cover_{s,   i}\in \covers_{i}(H_i^{(1)}),    
	\quad
	i\in \Lambda\setminus J,   s=1, 2, 
	$$
	and satisfying
	$$\exists_{x\in \{J\}\cup (\Lambda\setminus J)}\ \cover_{1,   x}\neq \cover_{2,   x}.
	$$ 
	Thus,    
	without loss of generality,    we may assume that  there exists 
	$C\in \cover_{1,   x}\setminus \cover_{2,   x}$ for some $x\in \{J\}\cup (\Lambda\setminus J)$. We may write equivalently $C\in \cover_{1,   x}$ and $C\in {\cal P}(H^{(1)}_x)\setminus \cover_{2,   x}$.  Hence
	$$
	C\in \cover_{1,   J}\cup\bigcup_{i\in \Lambda\setminus J}\cover_{1,   i}\quad\text{and}\quad
	C\in \left({\cal P}(H^{(1)}_{J})\setminus \cover_{2,   J}\right) 
	\cup
	\bigcup_{i\in \Lambda\setminus J}\left({\cal P}(H_i^{(1)})\setminus\cover_{2,   i}\right).$$
	This implies that $C$ is an element of every clique cover in 
	$
	\covers_1
	$
	and is contained in no clique cover in 
	$
	\covers_2
	$.
	Hence,   $\covers_1$ and $\covers_2$ are disjoint.
	
	Consequently,    combining  \eqref{Eq:BijectionsSameStructure},     \eqref{Eq:SubgraphProbabilityCCfamily}  and Lemma~\ref{Lem:CCequality} we get
	\begin{align*}
		&\pi(H^{(1)},    \covers_{J}(H^{(1)}))=\\
		&=  \sum_{\cover_J\in \covers_{J}(H_J^{(1)})}
		\sum_{\cover_i\in \covers(H_i^{(1)}),   i\in \Lambda\setminus J}  
		\pi\Bigg(
		H^{(1)},    
		\cover_J\cup\bigcup_{i\in \Lambda\setminus J}\cover_i,\\
		&\hspace{8cm}    
		\left({\cal P}(H^{(1)}_J)\setminus \cover_J\right) 
		\cup
		\bigcup_{i\in \Lambda\setminus J}\left({\cal P}(H_i^{(1)})\setminus\cover_i\right)
		\Bigg)\\
		&=  \sum_{\cover_J\in \covers_{J}(H_J^{(1)})}
		\sum_{\cover_i\in \covers(H_i^{(1)}),   i\in \Lambda\setminus J}  
		\pi\Bigg(
		H^{(2)},   
		\phi_{J}(\cover_J)\cup\bigcup_{i\in \Lambda\setminus J}\phi(\cover_i),  \\
		&\hspace{7cm}   
		\left({\cal P}(H^{(2)}_J)\setminus \phi(\cover_J)\right) 
		\cup
		\bigcup_{i\in \Lambda\setminus J}\left({\cal P}(H_i^{(2)})\setminus\phi(\cover_i)\right)
		\Bigg)\\
		&=  \sum_{\cover_J\in \covers_{J}(H_J^{(2)})}
		\sum_{\cover_i\in \covers(H_i^{(2)}),   i\in \Lambda\setminus J}  
		\pi\Bigg(
		H^{(2)},    
		\cover_J\cup\bigcup_{i\in \Lambda\setminus J}\cover_i, \\
		&\hspace{8cm}      
		\left({\cal P}(H^{(2)}_J)\setminus \cover_J\right) 
		\cup
		\bigcup_{i\in \Lambda\setminus J}\left({\cal P}(H_i^{(2)})\setminus\cover_i\right)
		\Bigg)\\
		&=\pi(H^{(2)},    \covers_{J}(H^{(2)})),  
	\end{align*}
	as required.
\end{proof}

The next result will be essential for obtaining sharp estimates of subgraph probabilities in the case $mp^2\ge 1$. It also clearly illustrates how the lemma above can be beneficial.

\begin{lemma}\label{Lem:DuzeMp2} 
	Let $mp^2\ge 1$,    $mp^3\le 1$,    and $H_1$,  $H_2$ be non-empty graphs. Then we have
	$$
	\pi(H_1\cup H_2)=\frac{\pi(H_1)\pi(H_2)}{\pi(H_1\cap H_2)}+O(mp^3e^{-2mp^2}).
	$$
\end{lemma}
\begin{proof}
	Let $G^{(1)}=(H_1\cup H_2)\oplus_0 (H_1\cap H_2)$ and $G^{(2)}=H_1\oplus_0 H_2$. 
	By definition,   $G^{(1)}$ and $G^{(2)}$ have the same number of edges and the same number of vertices.
	Set $E(G^{(1)})=\{e^{(1)}_1,   \ldots,   e^{(1)}_l\}$ and $E(G^{(2)})=\{e^{(2)}_1,   \ldots,   e^{(2)}_l\}$. Now,    for $j=1,   2$,    consider partitions $\{\bar H^{(j)}_{i}\}_{i\in \{1,   \ldots,   l\}}$,    where $\bar H^{(j)}_{i}$ is the  subgraph of $G^{(j)}$  induced by the edge $e^{(j)}_i$. We clearly have  $\bar H^{(1)}_{i}\cong \bar H^{(2)}_{i}$,  for all $i\in \{1,   \ldots,   l\}$,  and therefore Lemma~\ref{Lem:DecompositionComparing} gives us
	\begin{align}\label{Eq:all2sets}
		\begin{array}{rl}
			(i)&\pi(G^{(1)},    \covers_{\emptyset}(G^{(1)}))=\pi(G^{(2)},    \covers_{\emptyset}(G^{(2)})), \\[5pt] 
			(ii)&\forall_{i\in \{1,   \ldots,   l\}}\pi(G^{(1)},    \covers_{\{i\}}(G^{(1)}))=\pi(G^{(2)},    \covers_{\{i\}}(G^{(2)})).
		\end{array}
	\end{align}
	Recall that $\covers_{\emptyset}(G^{(j)})$ contains all clique covers decomposable to $\bar H^{(j)}_\emptyset$ in relation to the partition $\{\bar H^{(j)}_i\}_{i\in \{1,   \ldots,   l\}}$. 
	This means that $\covers_{\emptyset}(G^{(j)})$ contains all clique covers $\cover$ such that $E(G^{(j)})\subseteq\cover$,    i.e. all edges are covered by 2--sets.	Moreover, $\covers_{\{i\}}(G^{(j)})$ contains all clique covers of $G^{(j)}$  decomposable to $\bar H^{(j)}_{\{i\}}$ in relation to the partition $\{\bar H^{(j)}_i\}_{i\in \{1,   \ldots,   l\}}$, i.e. such that all edges but $e_i^{(j)}$ are covered by 2--sets and  the 2--set covering $e_i^{(j)}$ is forbidden. As a consequence,    given $j=1,   2$,    all the families of clique covers $\covers_{\emptyset}(G^{(j)})$,    $\covers_{\{i\}}(G^{(j)})$,    $i=1,   \ldots,   l$,    are pairwise disjoint. Define
	$$
	\covers^*(G^{(j)}):=\covers_{\emptyset}(G^{(j)})\cup\bigcup_{i=1}^{l}\covers_{\{i\}}(G^{(j)}),   \hspace{20pt} j=1,   2, 
	$$ 
	which is the family of all clique covers of $G^{(j)}$ with at most one forbidden $2$--set. Since,  as mentioned before,  the families in the union above are pairwise disjoint,  by \eqref{Eq:all2sets} and \eqref{Eq:SubgraphProbabilityCCfamily} we have
	\begin{equation}\label{Eq:Lem:DuzeMp2:1}
		\pi(G^{(1)},  \covers^*(G^{(1)}))
		=\pi(G^{(2)},  \covers^*(G^{(2)})).
	\end{equation}	
	
	All clique covers that are not contained in $\covers^*(G^{(j)})$  have at least two forbidden $2$--sets and at least one $k$--set,    $k\ge 3$,    that covers the edges that are related to the forbidden $2$--sets. Therefore, by \eqref{eq:CC4} we deduce
	\begin{equation}\label{Eq:Lem:DuzeMp2:2}
		\forall_{\cover\in \covers(G^{(j)})\setminus \covers^*(G^{(j)})} \pi(G^{(j)},   \cover)\lesssim mp^3e^{-2mp^2}, \quad j=1,  2.
	\end{equation}
	Combining \eqref{Eq:Lem:DuzeMp2:1},  \eqref{Eq:Lem:DuzeMp2:2},  and \eqref{Eq:SubgraphProb} we obtain
	\begin{equation*}
		\begin{split}
			\pi(H_1\cup H_2)\pi(H_1\cap H_2)&=\pi(G^{(1)})\\
			&=\pi(G^{(1)},    \covers^*(G^{(1)}))+\pi(G^{(1)},    \covers(G^{(1)})\setminus \covers^*(G^{(1)}))\\
			&=\pi(G^{(1)},    \covers^*(G^{(1)}))+O(mp^3e^{-2mp^2})\\
			&=\pi(G^{(2)},    \covers^*(G^{(2)}))+O(mp^3e^{-2mp^2})\\
			&=\pi(G^{(2)},    \covers^*(G^{(2)}))+\pi(G^{(2)},    \covers(G^{(2)})\setminus \covers^*(G^{(2)}))+O(mp^3e^{-2mp^2})\\
			&=\pi(G^{(2)})+O(mp^3e^{-2mp^2})\\
			&=\pi(H_1)\pi(H_2)+O(mp^3e^{-2mp^2}).
		\end{split}
	\end{equation*}
	This,  together with the fact that for $mp^2\ge 1$ we have $\pi(H_1\cap H_2)\asymp 1$,  implies the assertion of the lemma.
\end{proof}

We conclude this section with a series of results that will allow us to estimate the probabilities of indecomposable clique covers.

\begin{figure}
	\begin{center}
		\includegraphics[width=15cm]{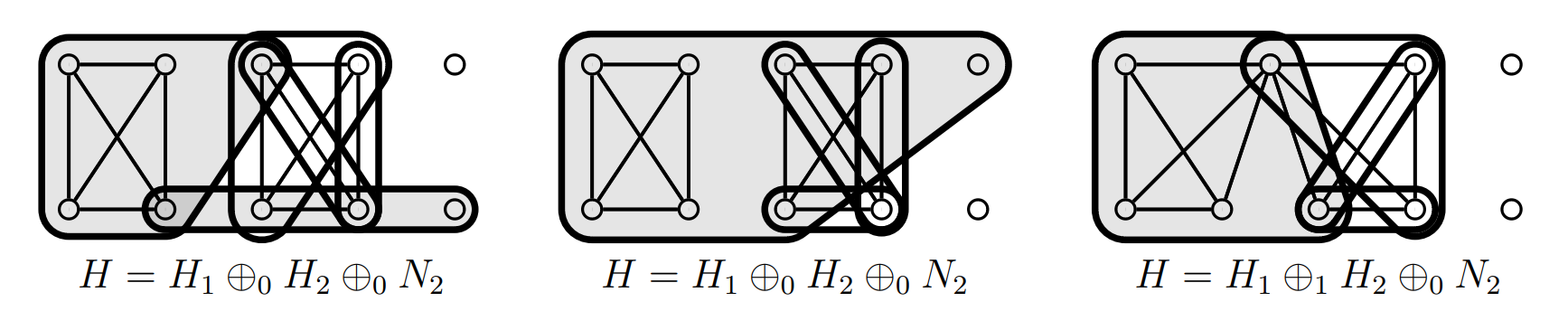}
	\end{center}
	\caption{\label{Fig:Indecomposable2graphs} Examples of indecomposable clique covers of $H$,  with relation to the partition $\{H_i\}_{i\in\{1, 2\}}$,   where $H_i$,   $i=1,  2$,   are isomorphic to $\klik_4$. In the pictures,   the sets of the clique cover that are not included neither in $V(H_1)$ nor in $V(H_2)$ are  grey. $C_{\cap}$ is empty in the first of the depicted graphs, and  consists of one set in the remaining graphs. }
\end{figure}

\begin{lemma}\label{Lem:InDecomposable2}
	Let $mp^2\le 1$,  $H_1$ and $H_2$ be two graphs and $\emp_j$,   $j\ge 0$,  be an empty graph on $j$ vertices.
	Furthermore,  let $H=H_1\oplus_0 H_2\oplus_0 \emp_j$ and $\cover$ be its indecomposable clique cover with relation to the partition $\{H_i\}_{i\in\{1,   2\}}$. Then
	$$
	\pi(H,   \cover)\lesssim m^{-1}\pi(H_1)\pi(H_2).$$
\end{lemma}
\begin{proof}
	Recall that $\cover\in \covers(H)$ is indecomposable with relation to the partition $\{H_i\}_{i\in \{1,  2\}}$,   if for each $i\in\{1,  2\}$ it has no subfamily that is a clique cover of $H_i$. Examples of indecomposable clique covers of $H$ are presented in Figure~\ref{Fig:Indecomposable2graphs}. 
	
	Let $\cover$ be an indecomposable  clique cover of $H$. For $i=1,   2$ we define
	\begin{align*}
		\cover_i&=\{C\in \cover: C\not\subseteq V(H_i)\text{ and $C$ covers only edges of }H_i\}, \\
		\cover_\cap&=\{C\in \cover: \text{$C$ covers edges of $H_1$ and $H_2$}\}.
	\end{align*}
	First, consider $C\in \cover_\cap$. Since $H_1$ and $H_2$ are vertex disjoint graphs,  $|C|\ge |C\cap V(H_1)|+|C\cap V(H_2)|$,  and hence
	\begin{equation}\label{Eq:Inde:1}
		\forall_{ C\in \cover_\cap}\ \frac{mp^{|C|}}{mp^{|C\cap V(H_1)|}mp^{|C\cap V(H_2)|}}\le m^{-1}.
	\end{equation}
	Now consider $C\in \cover_i$. Since $C$ has at least one vertex  outside $V(H_i)$,  we have
	$|C|\ge |C\cap V(H_i)|+1$. Therefore,  for $mp^2\le 1$, 
	\begin{equation}\label{Eq:Inde:2}
		\forall_{C\in \cover_i}\ \frac{mp^{|C|}}{mp^{|C\cap V(H_i)|}}\le p\le m^{-\frac12} .
	\end{equation}	Moreover,  for any $C\in \cover \setminus (\cover_{\cap}\cup\cover_1\cup\cover_2)$ we have that $V(H_i)\supset C\in \cover[V(H_i)]$ for some $i\in\{1,  2\}$. Therefore
	\begin{equation}\label{Eq:Inde:3}
		\forall_{C\in \cover \setminus (\cover_{\cap}\cup\cover_1\cup\cover_2)}\exists_{i\in \{1, 2\}}\ \   mp^{|C|}=mp^{|C\cap V(H_i)|}.
	\end{equation}
	Furthermore,  
	\begin{itemize}
		\item[(i)] for any $C\in \cover_\cap$ we have $C\cap V(H_i)\in \cover[V(H_i)]$,  $i=1,  2$;
		\item[(ii)] for $C\notin \cover_\cap$  there exists $j\in\{1,  2\}$  such that  $C\cap V(H_j)\in \cover[V(H_j)]$ and $C\cap V(H_{3-j})\notin \cover[V(H_{3-j})]$; in particular,  for $C\in\cover_i$,  $j=i$. 
	\end{itemize}	
	
	By \eqref{eq:CC2},  \eqref{Eq:Inde:3},  (i)  and (ii),  we get
	\begin{align*}
		&\frac{\pi(H,   \cover)}{\pi(H_1,   \cover[V(H_1)])\pi(H_2,   \cover[V(H_2)])}\\
		&\asymp 
		\frac{\prod_{C\in\cover}mp^{|C|}}{\prod_{C\in\cover[V(H_1)]}mp^{|C|}\prod_{C\in\cover[V(H_2)]}mp^{|C|}}\\
		&=\frac{\prod_{C\in\cover_\cap}mp^{|C|}}{\prod_{C\in\cover_\cap}mp^{|C\cap V(H_1)|}mp^{|C\cap V(H_2)|}}
		\cdot \frac{\prod_{C\in\cover_1}mp^{|C|}}{\prod_{C\in\cover_1}mp^{|C\cap V(H_1)|}}
		\cdot \frac{\prod_{C\in\cover_2}mp^{|C|}}{\prod_{C\in\cover_2}mp^{|C\cap V(H_2)|}}
		.
	\end{align*}
	Let us note that in the last line there are no terms related to $C\in \cover \setminus (\cover_{\cap}\cup\cover_1\cup\cover_2)$, as they  cancel out   by \eqref{Eq:Inde:3} and (ii).
	
	Since $\cover$ is indecomposable,  there exists $C\in \cover_\cap$ or there exist $C_1\in \cover_1$ and $C_2\in \cover_2$. Combining this observation with \eqref{Eq:Inde:1},  \eqref{Eq:Inde:2}  and the above equality we obtain 
	\begin{align*}
		\frac{\pi(H,   \cover)}{\pi(H_1,   \cover[V(H_1)])\pi(H_2,   \cover[V(H_2)])}
		&\lesssim m^{-1}.
	\end{align*}
	This together with \eqref{Eq:SubgraphProbabilityCCfamily} implies the assertion. 
\end{proof}
\begin{lemma}\label{Lem:InDecomposable2bis}
	Let  $mp^2\ge 1$ and $mp^3\le 1$,    $H$ be a graph with a partition $\{H_i\}_{i\in \Lambda}$,   and $\cover\in \covers(H)$ be indecomposable with relation to the partition   $\{H_i\}_{i\in \Lambda}$. Then 
	$$
	\pi(H,   \cover)\lesssim mp^3e^{-|\Lambda|mp^2}.
	$$
	Additionally,  for $H=H_1\oplus_0 H_2$ we have
	$$\pi(H,  \cover)\lesssim (mp^3)^2e^{-2mp^2}.$$
\end{lemma}

\begin{proof}
	Each indecomposable clique cover of $H$ has at least one forbidden 2--set $e_i$,   $i\in \Lambda$,   for each $H_i$,    $i\in\Lambda$. Otherwise the clique cover could be decomposed with $\cover_i=E(H_i)$,  for some $i\in \Lambda$. Moreover,   each of the edges $e_i$,   $i\in \Lambda$,   related  to those forbidden 2--sets,    must be covered by some $k$--set,   $k\ge 3$. Therefore,    by \eqref{eq:CC4},   
	\begin{align*}
		\pi(H,   \cover)\lesssim mp^3e^{-|\Lambda|mp^2}, 
	\end{align*}
	where the constant in the inequality may be chosen to depend only on $|V(H)|$. 
	
	In the case $H=H_1\oplus_0 H_2$,  in an indecomposable clique cover $\cover$,   there are at least two forbidden 2-sets,  related to some edges $e_i\in H_i$,  $i=1,  2$. As $H_1$ and $H_2$ are disjoint graphs,  $e_1$ and $e_2$ are covered by at least one set of size at least $4$ or at least two sets of size at least $3$ each. Therefore by \eqref{eq:CC4},  
	\begin{align*}
		\pi(H,   \cover)\lesssim ((mp^3)^2+mp^4)e^{-2mp^2}\lesssim (mp^3)^2 e^{-2mp^2}.
	\end{align*}

	The proof is complete.
\end{proof}

\begin{figure}\begin{center}
		\includegraphics[width=11cm]{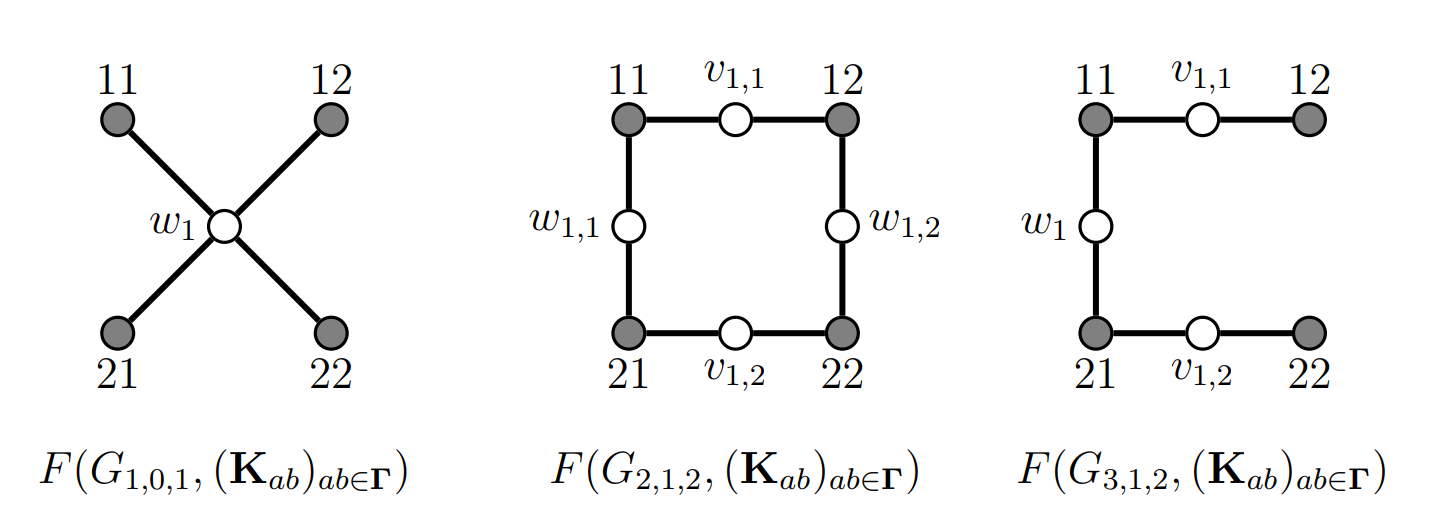}
	\end{center}
	\caption{\label{Fig:GraphF} Graphs $F(G_t, (\kli_{ab})_{ab\in\indeksy})$,  $t=1,  2,  3$. In this case $\Lambda=\indeksy=\{11, 12, 21, 22\}$ and  $H_{ab}=\kli_{ab}$,  $ab\in \Lambda$. In the picture white vertices are those from $V_{\cap 1}$, and gray ones are those from $\Lambda=\indeksy$.}
\end{figure}

\begin{definition}\label{def:F}
	For a partition  $\{H_i\}_{i\in \Lambda}$ of $H$,    let us define 
	$$V_{\cap 1}:=\{v:\exists_{i\neq j}V(H_i)\cap V(H_j)=\{v\}\}, $$
	i.e. the set of all the vertices that belong to at least two graphs from the partition. 
	Furthermore,   define the bipartite graph $F=F(H,   (H_i)_{i\in \Lambda})$,   with   the bipartition $(\Lambda,   V_{\cap 1})$,  as follows
	\begin{align*}
		V(F)&:=\Lambda\cup V_{\cap1},\\
		E(F)&:=\{\{i,   v\}: i\in \Lambda,    v\in V_{\cap1},    v\in V(H_i)\}.
	\end{align*} 
\end{definition}
In other words,   $\{i, v\}\in E(F)$  if and only if $v\in V(H_i\cap H_j)$ for some $j\in \Lambda$. Examples of graphs $F=F(H,   (H_i)_{i\in \Lambda})$ are presented in Figure~\ref{Fig:GraphF}.

\begin{lemma}\label{Lem:InDecomposable}
	Let $mp\ge 1$ and $mp^2\le 1$,    $H$ be a graph and $(H_i)_{i\in \Lambda}$ its partition such that $F=F(H, \{H_i\}_{i\in \Lambda})$ contains at most one cycle. If  $\cover\in \covers(H)$ is indecomposable with relation to the partition $\{H_i\}_{i\in \Lambda}$  of $H$, then
	$$
	\pi(H,   \cover)\lesssim \frac{\prod_{i\in \Lambda}\pi(H_i)}{(mp)^{\lceil |\Lambda|/2\rceil}p^{\delta}},   
	$$
	where $\delta=1$ if $F$ has a cycle, and $\delta=0$ otherwise.
\end{lemma}
\begin{proof}
	Let $\cover\in \covers(H)$ be indecomposable with relation to the partition $\{H_i\}_{i\in \Lambda}$,    i.e., for each $i\in \Lambda$,    there is an edge $e_i\in E(H_i)$ that is covered only by sets that are not subsets of $V(H_i)$. 
	If $\cover$ does not have {\Pone} then by Lemma~\ref{Lem:Simple}(ii),  it may be replaced by $\cover'$ with {\Pone} such that  each $e_i$,   $i\in \Lambda$,   is covered by exactly one $C_i$  (due to {\Pone}). Moreover, Lemma~\ref{Lem:Simple}(ii) implies that $C_i\nsubseteq V(H_i)$,  because all sets that covered $e_i$ in $\cover$ were not contained in $V(H_i)$.   Therefore $\cover'$ is also indecomposable. Thus, we may assume that $\cover$ has {\Pone}.
	
	Let $\cover^*=\{C_1,   \ldots,   C_s\}\subseteq \cover$ be the family of all sets $C\in \cover$ such that 
	$$
	\forall_{i\in \Lambda}C\nsubseteq V(H_i),  
	$$
	i.e., sets that are not contained in the vertex set of any of the graphs of the partition.
	
	For all $j\in \{1, \ldots, s\}$,   define
	$$\Lambda_j:=
	\{i\in \Lambda: \text{$C_j$ covers an edge of $H_i$}\}, $$
	that is  the index set of graphs from the partition  whose edges are covered by $C_j$.
	
	Since $\cover$ is indecomposable,    for each $i\in \Lambda$    there is an edge in $H_i$ covered by some $C\in \cover^*$. Therefore,  $\bigcup_{j=1}^{s}\Lambda_j=\Lambda$ and consequently 
	\begin{equation}\label{Eq:Indecomp0}
		\sum_{j=1}^{s}|\Lambda_j|\ge |\Lambda|.
	\end{equation}
	Now we define a clique cover
	\begin{equation}\label{Eq:LemInDecompCH}
		\cover_H
		=\bigcup_{i\in \Lambda}\cover[V(H_i)].  
	\end{equation}
	Note that 
	$$
	\forall_{C\notin \cover^*}\exists_{i\in \Lambda} C\subseteq V(H_i)\text{ and }mp^{|C|}=mp^{|C\cap V(H_i)|}.
	$$
	Hence,  by \eqref{eq:CC2},  since the terms related to $C\notin \cover^*$ cancel out (see the equation above),  we obtain 
	\begin{align}\nonumber
		\frac{\pi(H,  \cover)}{\pi(H,  \cover_H)}
		&\asymp
		\frac
		{\prod_{C\in \cover}mp^{|C|}}
		{\prod_{i\in \Lambda}\prod_{C'\in \cover[H_i]}mp^{|C'|}}\asymp \prod_{j=1}^{s}\frac{mp^{|C_j|}}{\prod_{i\in \Lambda_j}mp^{|C_j\cap V(H_i)|}}\\
		&=\prod_{j=1}^{s}m^{-(|\Lambda_j|-1)}p^{|C_j|-\sum_{i\in \Lambda_j}|C_j\cap H_i|}.
		\label{Eq:LamIndecomposableMain}
	\end{align}
	Now we analyse each term of the product. 
	
	If $|\Lambda_j|=1$ then for some $l\in\Lambda$ we have $\Lambda_j=\{l\}$. {Therefore} $|C_j|-\sum_{i\in \Lambda_j}|C_j\cap V(H_i)|=|C_j|-|C_j\cap V(H_l)|\ge 1$ (since $C_j\nsubseteq V(H_l)$). Thus,   for $mp^2\le 1$,  
	\begin{equation}\label{Eq:LemInDecomposibleLambda1}
		m^{-(|\Lambda_j|-1)}p^{|C_j|-\sum_{i\in \Lambda_j}|C_j\cap H_i|}\le p\le \frac{1}{mp}.
	\end{equation} 	
	
	Now consider $|\Lambda_j|\ge 2$. Recall that $V_{\cap 1}$ is the set of all the vertices of $H$ that are on intersections of graphs of the partition $\{H_i\}_{i\in \Lambda}$. 
	First, observe that $V(H_i)\setminus V_{\cap1}$,  $i\in \Lambda$,  are pairwise disjoint sets and,  for any $j\in\{1, \ldots, s\}$,  $\bigcup_{i\in \Lambda_j}(V(H_i)\setminus V_{\cap1})\subseteq( V(H)\setminus V_{\cap1})$. Therefore
	\begin{equation*}
		\sum_{i\in \Lambda_j} 
		\Big|(C_j\cap (V(H_i))\setminus V_{\cap1})\Big|
		=\Big|\Big(C_j\cap\bigcup_{i\in \Lambda_j} (V(H_i))\setminus V_{\cap1}\Big)\Big|\le |C_j\setminus V_{\cap1}|.
	\end{equation*}
	Thus
	\begin{align}\nonumber
		&|C_j|-\sum_{i\in \Lambda_j}|C_j\cap V(H_i)|\\ \nonumber
		&=  
		|C_j\setminus V_{\cap1}| + |C_j\cap V_{\cap1}|
		-\sum_{i\in \Lambda_j} 
		(|(C_j\cap V(H_i))\setminus V_{\cap1}| + |C_j\cap V(H_i)\cap V_{\cap1}|) \\ \nonumber
		&\ge |C_j\cap V_{\cap1}|-\sum_{i\in \Lambda_j} |C_j\cap V(H_i)\cap V_{\cap1}|\\ \nonumber
		&=|C_j\cap V_{\cap1}|- \sum_{i\in \Lambda_j}\sum_{v\in V_{\cap1}\cap C_j} |V(H_i) \cap\{v\}|\\
		&=|C_j\cap V_{\cap1}|- \sum_{v\in V_{\cap1}\cap C_j}\sum_{i\in \Lambda_j} |V(H_i) \cap\{v\}|
		.
		\label{Eq:LemInDecomposableSum1}	
	\end{align}
	Recall that the degree of $v\in V_{\cap1}$ in $F$ is the number of graphs $H_i$,   $i\in \Lambda$,   that contain~$v$. Moreover,   observe that $\sum_{i\in \Lambda_j} |V(H_i) \cap\{v\}|$ is the number of graphs $H_i$,   $i\in \Lambda_j$,   that contain $v$, which is the degree of $v$ in the subgraph of $F$ induced on $\Lambda_j\cup (V_{\cap1}\cap C_j)$. Thus,   for $v\in C_j\cap V_{\cap1}$,    
	$$
	\sum_{i\in \Lambda_j} |V(H_i) \cap\{v\}|=\deg_{F_j}(v),  \quad\text{where }F_j=F[\Lambda_j\cup (V_{\cap1}\cap C_j)].
	$$
	Moreover,   since $F_j$ is bipartite with the bipartition $(V_{\cap1}\cap C_j,  \Lambda_j)$,   we have
	$$
	\sum_{v\in V_{\cap1}\cap C_j}\sum_{i\in \Lambda_j} |V(H_i) \cap\{v\}|
	=
	\sum_{v\in V_{\cap1}\cap C_j}\deg_{F_j}(v)
	=|E(F_j)|.
	$$
	Therefore,   when we substitute this equation to \eqref{Eq:LemInDecomposableSum1} we obtain
	$$
	|C_j|-\sum_{i\in \Lambda_j}|C_j\cap V(H_i)|
	= |V_{\cap1}\cap C_j|-|E(F_j)|.
	$$
	Let us define $\delta_j=1$ if $F_j$ has a cycle and $\delta_j=0$ if $F_j$ does not have any cycle. Recall that $F$ has at most one cycle,   thus $F_j$ has at most one,   too. Therefore,   for all $j=1,  \ldots,  s$,   we have
	$$
	|E(F_j)|\le |V(F_j)|-1+\delta_j =|\Lambda_j|+|C_j\cap V_{\cap1}|-1+\delta_j.
	$$
	Combining the last two equalities
	$$
	|C_j|-\sum_{i\in \Lambda_j}|C_j\cap V(H_i)|\ge -|\Lambda_j|+1-\delta_j.
	$$
	Therefore 
	\begin{equation}\label{Eq:LemInDecomposibleLambda2}
		m^{-(|\Lambda_j|-1)}p^{|C_j|-\sum_{i\in \Lambda_j}|C_j\cap H_i|}\le m^{-(|\Lambda_j|-1)}p^{-|\Lambda_j|+1-\delta_j}=(mp)^{-(|\Lambda_j|-1)}p^{-
			\delta_j}.
	\end{equation} 
	Substituting this and \eqref{Eq:LemInDecomposibleLambda1} to \eqref{Eq:LamIndecomposableMain} we obtain
	\begin{align}\nonumber
		\frac{\pi(H,  \cover)}{\pi(H,  \cover_H)}
		&\lesssim 
		\left(\prod_{j,   |\Lambda_j|= 1}(mp)^{-1}\right) \cdot
		\left(\prod_{j,   |\Lambda_j|\ge 2}m^{-(|\Lambda_j|-1)}p^{-|\Lambda_j|+1-\delta_j}\right)\\
		&=
		\frac{1}{(mp)^{\sum_{j,   |\Lambda_j|=1}|\Lambda_j|+\sum_{j,   |\Lambda_j|\ge 2}(|\Lambda_j|-1)}p^{\sum_{j=1}^{s}\delta_j}}
		.
		\label{Eq:LemIndecomposableMain3}
	\end{align}
	Since $\cover$ has {\Pone},  the  inequality $|C_{j_1}\cap C_{j_2}|\le 1$ holds  for distinct $j_1,  j_2\in \{1, \ldots, s\}$. Therefore,    for any $j_1\neq j_2$,    $F_{j_1}$ intersect with $F_{j_2}$ on at most one vertex from $V_{\cap1}$. It implies that,    if $F$ contains one cycle,    it is included in at most one of its subgraphs $F_j$,    $j=1,   \ldots,   s$ (as a cycle in the bipartite graph $F$ contains at least two vertices from $V_{\cap1}$). Thus 
	$\delta_j=1$ for at most one $j$,    $j=1,   \ldots,   s$,  and only if $F$ has a cycle. It implies that $$\sum_{j=1}^{s}\delta_j\le \delta.$$ 
	
	Moreover,    using \eqref{Eq:Indecomp0},   
	\begin{equation}\label{Eq:Indecomp3}
		\sum_{j,    |\Lambda_j| =1}|\Lambda_j|+\sum_{j,    |\Lambda_j|\ge 2}(|\Lambda_j|-1)\ge 
		\bigg\lceil
		\frac{\sum_{j}|\Lambda_j|}{2}
		\bigg\rceil
		\ge \bigg\lceil
		\frac{|\Lambda|}{2}
		\bigg\rceil.
	\end{equation}
	The two equalities above together with \eqref{Eq:LemIndecomposableMain3} imply 
	$$
	\frac{\pi(H,  \cover)}{\pi(H,  \cover_H)}\lesssim \frac{1}{(mp)^{\lceil |\Lambda|/2\rceil}p^{\delta}}.
	$$
	Thus,   recalling the definition of $\cover_H$ \eqref{Eq:LemInDecompCH},   by \eqref{eq:CC2} and the fact that graphs $H_i$,   $i\in \Lambda$,   intersect on at most one vertex (i.e., $ \cover[V(H_i)]$,   $i\in \Lambda$ are pairwise disjoint)
	$$
	\pi(H,  \cover)
	\lesssim
	\frac{\pi\left(H,  \bigcup_{i\in\Lambda} \cover[V(H_i)]\right)}{(mp)^{\lceil |\Lambda|/2\rceil}p^\delta}
	\lesssim \frac{\prod_{i\in \Lambda} \pi(H_i,  \cover[V(H_i)])}{(mp)^{\lceil |\Lambda|/2\rceil}p^\delta}
	\lesssim \frac{\prod_{i\in \Lambda} \pi(H_i)}{(mp)^{\lceil |\Lambda|/2\rceil}p^\delta}.
	$$
	The proof is complete.
\end{proof}

\section{Proof of Theorem~\ref{Tw:Wariancja}}\label{Sec:Variance}

Recall that  $\Subgraphs$ is the family of all cliques of size $r$ in the complete graph with the vertex set $\V$, and  
$$
\X=\sum_{\kli\in \Subgraphs}\ind_{\kli}, 
$$
where $\ind_{\kli}$ is the indicator random variable of the event $\{\kli\subseteq \Gnmp \}$.

In this section we study the variance of $\X$,  which can be expressed in the following manner. Since $\ind_{\kli}$ and $\ind_{\kli'}$ are independent whenever $\kli, \kli'$ are disjoint, we have
\begin{align}
	\nonumber
	\Var [\X]&= \sum_{\kli\in \Subgraphs}\Var[ \ind_{\kli}] + \sum_{{\substack{\kli_1,   \kli_2\in \Subgraphs,   \\ \kli_1\neq \kli_2}}}\cov (\ind_{\kli_1},   \ind_{\kli_2})\\ \nonumber
	&= \sum_{\kli\in \Subgraphs}\Var [\ind_{\kli}] + \sum_{i=1}^{r-1}\sum_{{\substack{\kli_1,   \kli_2\in \Subgraphs\\ |V(\kli_1\cap \kli_2)|=i}}}\cov(\ind_{\kli_1},   \ind_{\kli_2})\\ \nonumber
	&=\binom{n}{r}\qq(1-\qq)
	+\sum_{i=1}^{r-1}\big|\{\kli_1,   \kli_2\in \Subgraphs: |V(\kli_1\cap \kli_2)|=i\}\big|\left(\pi(\kli_1\oplus_i\kli_2)-\qq^2\right),  \\
	&\label{Eq:Var:FormulaGeneral2}
	=\sum_{i=1}^{r}\frac{(n)_{2r-i}}{[(r-i)!]^2i!}\left(\pi(\kli_1\oplus_i\kli_2)-\qq^2\right),  
\end{align}
where $(n)_k=n\cdot (n-1)\cdot (n-k+1)$,  and $\kli_1,  \kli_2$ are some fixed cliques of size $r$. Note that Lemma~\ref{Lem:Division2} implies that all the terms in \eqref{Eq:Var:FormulaGeneral2} are non-negative.

In the proofs of this section,  we will use the following relations concerning the terms in \eqref{eq:Var},  that are consequences of \eqref{eq:pp} and Fact \ref{Lem:DuzeMp2dwa}
\begin{align}
	\label{Eq:Var:Szacowanie:Duzemp2}
	\frac{(n^r\qq)^2}{\ble}(1-\pp)&\asymp n^{2r-2}e^{-mp^2},  \quad \text{ for }mp^2\ge 1, \\[7pt]
	\label{Eq:Var:Szacowanie:Przeciecie1}
	\left(\frac{\qq}{\pp}\right)^2mp^3(1-\pp)^2
	&\asymp
	\begin{cases}
		\frac{\qq^2}{mp}, &\text{ for }mp^2\le 1;\\[5pt]
		mp^3e^{-2mp^2}, &\text{ for }mp^2\ge 1\text{ and }mp^3\le 1. 
	\end{cases}
\end{align}
In particular,  this allows us to rewrite \eqref{eq:Var} into the form
\begin{align}\label{eq:Var2}
	\Var[\X]\asymp
	\begin{cases}
		\frac{(n^r\qq)^2}{\ble\wedge (nmp)}, &\text{ for }mp^2\le 1;\\[5pt]
		n^{2r-2}e^{-mp^2}+n^{2r-1}mp^3e^{-2mp^2}, &\text{ for }mp^2\ge 1\text{ and }mp^3\le c_r. 
	\end{cases}
\end{align}

Let us recall the definition \eqref{Eq:Def:Blei} of $\ble_i$. Theorem~\ref{Tw:Wariancja} follows directly from 
\eqref{Eq:Var:FormulaGeneral2},  the next two  lemmas,  and the equivalence \eqref{Eq:BleaBlei}.
\begin{lemma}\label{lem:Var2--r}
	Let $m$ be large enough and  $2\le  i \le r$. {If } $mp\le1$,    or if  $mp^2\ge 1$  and  $mp^3\le 1$, we have
	\begin{align}\label{eq:KiK1}
		n^{2r-i}\pi(\kli_1\oplus_i\kli_2)\asymp \frac{(n^r\qq)^2}{\ble_i},
	\end{align}
	while if  $mp\ge 1$  and  $mp^2\le 1$, it holds that
	\begin{align}\label{eq:KiK2}
		\frac{(n^r\qq)^2}{\ble_i}\lesssim n^{2r-i}\pi(\kli_1\oplus_i\kli_2)\lesssim  \frac{(n^r\qq)^2}{\ble}.
	\end{align}
	Furthermore,  for $mp^3\le 1$ we have
	\begin{align}\label{eq:KiK3}
		\pi(\kli_1\oplus_i\kli_2)-\qq^2\asymp (1-\pp)\pi(\kli_1\oplus_i\kli_2).  
	\end{align}
\end{lemma}

Note that in \eqref{eq:KiK2} the lower bound may not be of the same order as the upper one. 
This follows from the fact that, for certain parameter ranges, we have 
$
\pi(\kli)\asymp \pi(\kli, E(\kli)),
$
whereas
$
\pi(\kli\oplus_i \kli)\asymp \pi(\kli\oplus_i \kli, \cover)
$
holds for a clique cover $\cover$ that does not consist solely of $2$-sets, possibly
$
\cover = \{ V(\kli\oplus_i \kli) \}.
$
This phenomenon has already been addressed in Remark~\ref{rem:dominatingccs}, and analysing it was one of the most challenging tasks in this article.

\begin{proof}
	Assume first $mp\le 1$.  Then,  by Corollary~\ref{Cor:Glueing},   the estimate \eqref{eq:CC2} and Lemma~\ref{Lem:PiOptimal},   for any $2\le i\le r$  we have    
	\begin{equation*}
		n^{2r-i}\pi(\kli_1\oplus_i\kli_2)
		\asymp n^{2r-i} \pi(\kli_1\oplus_i\kli_2,   \{V(\kli_1\oplus_i\kli_2)\})
		\asymp
		n^{2r-i}mp^{2r-i} 
		\asymp\frac{(n^r\qq)^2}{\ble_i}, 
	\end{equation*}
	which implies \eqref{eq:KiK1}. Using Corollary~\ref{Cor:Glueing} once again we obtain
	\begin{align*}
		\frac{\qq^2}{\pi(\kli_1\oplus_i\kli_2)}&\asymp \frac{(mp^r)^2}{mp^{2r-i}}=\frac{m^ip^i}{m^{i-1}}\le \frac1m, 
	\end{align*}
	which gives us \eqref{eq:KiK3} in this range of parameters.
	
	If $mp^2\ge1$ and $mp^3\le1$,   then \eqref{eq:KiK1} follows directly from Fact~\ref{Lem:DuzeMp2dwa},  as any subgraph probability is then comparable to a constant.  Regarding \eqref{eq:KiK3},  by Lemma~\ref{Lem:DuzeMp2} and Fact~\ref{Lem:DuzeMp2dwa} we get
	\begin{align*}
		\pi(\kli_1\oplus_i\kli_2)-\qq^2
		&=
		\qq^2\left(\frac{1}{\pi(\klik_i)}-1\right)+O(mp^3e^{-2mp^2})\\
		&=\qq^2\, \frac{1-\pi(\klik_i)}{\pi(\klik_i)}+O((1-\pp)^2)\\
		&\asymp 1-\pp\\
		&\asymp \pi(\kli_1\oplus_i\kli_2)(1-\pp).
	\end{align*}
	
	We now turn our attention to the case when $mp\ge1$ and $mp^2\leq1$ and will focus on it in the remaining part of the proof. 
	
	The upper bound follows directly from Lemma~\ref{Lem:Odrywanie} and the fact that $\kli_1\oplus_i\kli_2\cong G_{2, \II}$ for $\II=\{11, 12\}$.
	
	Now we address the lower bound. 
	For $m^{-1}\le p\le m^{-\frac{r+1}{2r}}$,      by \eqref{Eq:SubgraphProb},    \eqref{eq:CC2},    and Lemma~\ref{Lem:PiOptimal} we have 
	\begin{align*}
		n^{2r-i}\pi(\kli_1\oplus_i\kli_2)
		&\ge n^{2r-i}\pi(\kli_1\oplus_i\kli_2,   \{V(\kli_1\oplus_i\kli)\})\\
		&\asymp  n^{2r-i}mp^{2r-i}= \frac{(n^{r}mp^{r})^2}{n^imp^i}\asymp \frac{(n^r\qq)^2}{\ble_i}.
	\end{align*} 
	Similarly,    for $m^{-\frac{r+1}{2r}}\le p\le m^{-\frac12}$,    by \eqref{Eq:SubgraphProb},    \eqref{eq:CC2},    and Lemma~\ref{Lem:PiOptimal}  we obtain
	\begin{align*}
		n^{2r-i}\pi(\kli_1\oplus_i\kli_2)
		\ge n^{2r-i}\pi(\kli_1\oplus_i\kli_2,   E(\kli_1\oplus_i\kli_2))
		\asymp n^{2r-i}\pp^{2\binom{r}{2}-\binom{i}{2}}
		= \frac{(n^r\qq)^2}{\ble_i}, 
	\end{align*}
	which concludes the proof of \eqref{eq:KiK2}. 
	
	Eventually,   we examine the difference $\pi(\kli_1\oplus_i\kli_2)-\qq^2$ for $mp\ge1$ and $mp^2\le1$. Recall that $|V(\kli_1\oplus_i\kli_2)|=2r-i$. First,  consider $p\le m^{-\frac{(2r-i)+1}{2(2r-i)}+\delta}$,  for some $0<\delta<1/8r$. 
	We remark that  (since $2(2r-i) < 4r$)
	$$
	-\frac{(2r-i)+1}{2(2r-i)}+\delta<-\frac{1}{2}-\frac{1}{2(2r-i)}+\frac{1}{8r}<-\frac{1}{2}-\frac{1}{8r}.
	$$
	In particular, this implies
	$$
	mp^2\le m^{-1/(4r)}.
	$$
	Then,  for all $2\le i\le r$ 
	$$
	\ble_i\lesssim n^i\left(mp^i+(mp^2)^{\binom{i}{2}}\right)\le n^imp^2\le n^i m^{-1/{(4r)}}. 
	$$
	Therefore,  using the lower bound in \eqref{eq:KiK2} we get 
	\begin{align}\label{eq:aux5}
		\qq^2\lesssim   \frac{n^im^{-1/(4r)}}{\ble_i}\qq^2
		= m^{-1/(4r)} n^{i-2r}\frac{(n^r\qq)^2}{\ble_i}
		\lesssim
		m^{-1/(4r)}\pi(\kli_1\oplus_i\kli_2).
	\end{align}
	Thus,  since  $\pp=o(1)$ in this case, 
	$$
	\pi(\kli_1\oplus_i\kli_2)-\qq^2=(1+O(m^{-1/(4r)}))\pi(\kli_1\oplus_i\kli_2)\asymp (1-\pp) \pi(\kli_1\oplus_i\kli_2). 
	$$
	This implies \eqref{eq:KiK3} for $p\le m^{-\frac{(2r-i)+1}{2(2r-i)}+\delta}$ with any $0<\delta<1/8r$.
	
	Now we assume  $m^{-\frac{(2r-i)+1}{2(2r-i)}+\delta}\le p\le m^{-\frac12}$. Observe that  $p$ satisfies the assumptions of Lemma~\ref{Lem:Splitting}$(iii)$ with $A=m^{\delta}$.
	Therefore, for fixed $0<\delta<\delta_{2r}<\delta_{r}$,  Lemma~\ref{Lem:Splitting}$(iii)$ gives us
	\begin{align*}
		\qq&=\pi(\kli)=(1+O(m^{-\delta}+m^{-\delta_{r}}))\pp^{\binom{r}{2}}=(1+O(m^{-\delta}))\pp^{\binom{r}{2}}, \\
		\pi(\kli_1\oplus_i\kli_2)&=(1+O(m^{-\delta}+m^{-\delta_{2r}}))\pp^{\binom{r}{2}-\binom{i}{2}}=(1+O(m^{-\delta}))\pp^{\binom{r}{2}-\binom{i}{2}}.
	\end{align*}
	Consequently,  using the above equalities,  we obtain for $i\ge 2$
	\begin{align}
		\qq^2
		&=
		(1+O(m^{-\delta}))\pp^{2\binom{r}{2}}
		=
		\pp^{\binom{i}{2}}(1+O(m^{-\delta}))\pp^{2\binom{r}{2}-\binom{i}{2}}\\
		&=\pp^{\binom{i}{2}}(1+O(m^{-\delta}))\pi(\kli_1\oplus_i\kli_2)\\\label{eq:aux6}
		&\le 
		\pp(1+O(m^{-\delta}))\pi(\kli_1\oplus_i\kli_2).
	\end{align}
	Hence,
	$$
	(1-\pp-O(m^{-\delta}))\pi(\kli_1\oplus_i\kli_2)\le \pi(\kli_1\oplus_i\kli_2)-\qq^2 \le \pi(\kli_1\oplus_i\kli_2).
	$$
	This,  together with $1-\pp$ being bounded away from $0$ by a constant for $mp^2\le 1$,  implies \eqref{eq:KiK3} in this case.
	This completes the proof.
\end{proof}

\begin{lemma}\label{Lem:Var:OneIntersection}
	There exists a constant $c_r>0$ dependent on $r$ only such that for $mp^3\le c_r$ and $m$ sufficiently large  it holds that
	$$
	(\pi(\kli_1\oplus_1\kli_2)-\qq^2)
	\asymp 
	\left(\frac{\qq}{\pp}\right)^2mp^3(1-\pp)^2.
	$$
\end{lemma}

\begin{proof}
	Let us define two auxiliary graphs $H^{(1)}=\kli_1\oplus_1\kli_2\oplus_0 \emp_1$ and $H^{(2)}=\kli_1\oplus_0\kli_2$, and their partitions $\{\kli_i\}_{i\in\Lambda}$,    $\Lambda=\{1,   2\}$ (see Definition~\ref{Def:Partition}). Let  $\covers_J(H^{(l)})$,    $J\subseteq \Lambda$,    $l=1,   2$,   be as in Definition \ref{def:decomp}. 
	Given $l=1,   2$,    by definition the families $\covers_J(H^{(l)})$,    $J\subseteq \Lambda$,    are pairwise disjoint and $\covers(H^{(l)})=\bigcup_{J\subseteq \Lambda} \covers_J(H^{(l)})$. 
	Moreover, $H_{\{i\}}^{(1)}\cong K_r\oplus_0N_r\cong H_{\{i\}}^{(2)}$,   for $i=1,   2$,    and $H_{\emptyset}^{(1)}\cong \emp_{2r}\cong H_{\emptyset}^{(2)}$. Therefore,    by \eqref{Eq:SubgraphProb} and Lemma~\ref{Lem:DecompositionComparing}
	\begin{align*}
		\pi(\kli_1\oplus_1\kli_2)-\qq^2
		&=\pi(H^{(1)})-\pi(H^{(2)})\\
		&=\sum_{J\subseteq \Lambda} \left(\pi(H^{(1)},    \covers_J(H^{(1)}))-\pi(H^{(2)},    \covers_J(H^{(2)}))\right)\\
		&=\pi(H^{(1)},    \covers_\Lambda(H^{(1)}))-\pi(H^{(2)},    \covers_\Lambda(H^{(2)})).
	\end{align*}
	Note that $\covers_\Lambda(H^{(l)})$,    $l=1,   2$,    consists of  all clique covers of $H^{(l)}$,    $l=1,   2$,  that are  indecomposable with relation to partition $\{\kli_i\}_{i\in\Lambda}$,    $\Lambda=\{1,   2\}$. This will allow us to employ Lemmas \ref{Lem:InDecomposable2},  \ref{Lem:InDecomposable2bis},  and \ref{Lem:InDecomposable}. Furthermore,  in view of \eqref{Eq:Var:Szacowanie:Przeciecie1},  it now sufficies to show 
	\begin{align}\label{eq:aux7}
		\pi(H^{(1)},    \covers_\Lambda(H^{(1)}))-\pi(H^{(2)},    \covers_\Lambda(H^{(2)}))\asymp
		\begin{cases}
			\frac{\qq^2}{mp}, &\text{ for }mp^2\le 1;\\[5pt]
			mp^3e^{-2mp^2}, &\text{ for }mp^2\ge 1\text{ and }mp^3\le c_r. \hspace{20pt}
		\end{cases}
	\end{align}
	
	First,  we consider the case $mp\le 1$. Since $\{V(\kli_1\oplus_1\kli_2)\}\in  \covers_\Lambda(H^{(1)})$,  as it is an indecomposable clique cover of $\kli_1\oplus_1\kli_2$,    Corollary~\ref{Cor:Glueing} gives us
	$$
	\pi(H^{(1)},    \covers_\Lambda(H^{(1)}))\asymp \pi(H^{(1)},   \{V(\kli_1\oplus_1\kli_2)\}) \asymp mp^{2r-1}=\frac{(mp^r)^2}{mp} \asymp \frac{\qq^2}{mp}.
	$$
	On the other hand,    by Lemma~\ref{Lem:InDecomposable2} with $j=1$,  and using the inequality $p\le m^{-1} (\Leftrightarrow mp\le 1)$ we obtain
	\begin{align}\label{eq:aux8}
		\pi(H^{(2)},    \covers_\Lambda(H^{(2)}))\lesssim \frac{\qq^2}{m} \lesssim m^{-1} \frac{\qq^2}{mp}.
	\end{align}
	Therefore
	$$
	\pi(H^{(1)},    \covers_\Lambda(H^{(1)}))-\pi(H^{(2)},    \covers_\Lambda(H^{(2)})) \asymp (1+O(m^{-1}))  \frac{\qq^2}{mp}\asymp \frac{\qq^2}{mp},   
	$$
	which implies \eqref{eq:aux7} for $mp\le 1$.
	
	In the case $m^{-1}\le p\le m^{-1/2}$,  for $H^{(1)}$ we use Lemma~\ref{Lem:InDecomposable}. We recall that we consider partition $\{\kli_i\}_{i\in \Lambda}$,  $\Lambda=\{1, 2\}$,  and the corresponding graph $F=F(H^{(1)}, \{\kli_i\}_{i\in \Lambda})$ is a path with two edges. Therefore,  by Lemma~\ref{Lem:InDecomposable}
	\begin{equation}
		\pi(H^{(1)},    \covers_\Lambda(H^{(1)}))\lesssim\frac{\qq^2}{mp}.
	\end{equation} 
	To obtain the lower bound,   we consider two indecomposable clique covers. The first one is simply $\{V(\kli_1\oplus_1\kli_2)\}$. To construct the other one,    we set $v\in V(\kli_1)\cap V(\kli_2)$,    $v_i\in V(\kli_i)\setminus\{v\}$. Then let $\cover_\cap$ be obtained from $E(H^{(1)})$ by replacing $\{v,   v_1\}$ and $\{v,   v_2\}$ with $\{v,   v_1,   v_2\}$,  i.e.
	\begin{align}\label{eq:Ccap}
		\cover_\cap=\big(E(H^{(1)})\cup\{v,   v_1,   v_2\}\big)\setminus\big\{\{v,   v_1\},  \{v,   v_2\}\big\}.
	\end{align}
	The clique cover $\cover_\cap$ is clearly indecomposable as well. 
	Therefore,    by \eqref{eq:CC2} and Lemma~\ref{Lem:PiOptimal} we get for $p\le m^{-\frac{r+1}{2r}} $
	$$
	\pi(H^{(1)},    \covers_\Lambda(H^{(1)}))\ge \pi(H^{(1)},    \{V(H^{(1)})\})\asymp n^{2r-1}mp^{2r-1} \asymp \frac{\qq^2}{mp}, 
	$$  
	whereas  for $m^{-\frac{r+1}{2r}}\le p\le m^{-1/2}$
	$$
	\pi(H^{(1)},    \covers_\Lambda(H^{(1)}))\ge \pi(H^{(1)},    \cover_\cap)\asymp mp^3(mp^2)^{2\binom{r}{2}-2} \asymp \frac{\qq^2}{mp}. 
	$$  
	Eventually,  analogously as in \eqref{eq:aux8},   Lemma~\ref{Lem:InDecomposable2} yields $\pi(H^{(2)},    \covers_\Lambda(H^{(2)}))=O(\qq^2/m)=O(m^{-1/2}\qq^2/mp)$,  and hence
	$$
	\pi(H^{(1)},    \covers_\Lambda(H^{(1)}))-\pi(H^{(2)},    \covers_\Lambda(H^{(2)})) \asymp (1+O(m^{-1/2})) \frac{\qq^2}{mp},   
	$$
	as required.

	Now,    let us consider the case $mp^2\ge 1$ and $mp^3\le 1$. Combining  Lemma~\ref{Lem:InDecomposable2bis} with \eqref{Eq:SubgraphProbabilityCCfamily} we obtain
	\begin{align*}
		\pi(H^{(1)},    \covers_\Lambda(H^{(2)}))&\lesssim  mp^3e^{-2mp^2}, \\
		\pi(H^{(2)},    \covers_\Lambda(H^{(2)}))&\lesssim  (mp^3)^2e^{-2mp^2} .
	\end{align*}
	On the other hand,    for $\cover_\cap$ defined in \eqref{eq:Ccap},    by \eqref{eq:CC4} we have
	$$
	\pi(H^{(1)},    \covers_\Lambda(H^{(1)}))\ge \pi(H^{(1)},   \cover_\cap) \gtrsim \,  mp^3e^{-2mp^2}.   
	$$
	This gives us
	\begin{align}
		mp^3e^{-2mp^2} \gtrsim
		\pi(H^{(1)},    \covers_\Lambda(H^{(1)}))-\pi(H^{(2)},    \covers_\Lambda(H^{(2)}))
		\gtrsim  mp^3e^{-2mp^2} (1+O(mp^3)), 
	\end{align}
	which implies \eqref{eq:aux7} for sufficiently small $c_r$. Let us note that $c_r$ depends on $r$ only as the constants in all of the applied estimates  depend on graphs of at most $2r$ vertices.
\end{proof}

\section{Proof of Theorem \ref{Thm:Main} }\label{Sec:Proof}

The main strength of Theorem \ref{Thm:Main} lies in the bound \eqref{eq:main}. In order to fully concentrate on its proof,  we first deal with the remaining parts of the assertion of the theorem.

\begin{lemma}\label{lem:necessary}
	Assume $mp^3\le c_r$ with $c_r$ from Theorem \ref{Tw:Wariancja}. Then 
	$$\widetilde \X\stackrel{\mathcal D}{\longrightarrow}\mathcal N \ \Rightarrow\ (1-\pp)\ble\rightarrow\infty.$$
\end{lemma}
\begin{proof}
	We adapt the argument from \cite[proof of Theorem 3.5]{MNS}. 
	By Theorem \ref{Tw:Wariancja},  for $mp^3\le c_r$ and some constant $c=c(r)$ we have
	$$\widetilde \X=\frac{ \X-\E[\X]}{\sqrt{\Var[\X]}}\ge -\frac{ \E[\X]}{\sqrt{\Var[\X]}}\ge -c\frac{n^r\qq}{\sqrt{(1-\pp)(n^r\qq)^2/\ble}} =-c\sqrt{\frac{\ble}{(1-\pp)}}.$$
	Since  the normal distribution is supported on the whole real line,    the last expression has to tend to minus infinity,  which follows $\ble/(1-\pp)\rightarrow \infty$. This yields the required condition for $\pp\le 1/2$.
	
	Consider now $\pp\ge 1/2$ and let  $V_n^{(r)}=\binom{n}{ r}-\X$ be the number of copies of $K_r$ that do not appear in $\Gnmp$,  i.e., at least one of  the edges of a given copy is absent. Then we have  $\Var[V_n^{(r)}]=\Var[\X]$ and,  due to the symmetry of the standard normal distribution,  $\widetilde {V_n^{(r)}}=-\widetilde \X\stackrel{\mathcal D}{\longrightarrow} \mathcal N$. Next,  we proceed as previously.  Since,   by Fact \ref{Lem:DuzeMp2dwa},  $\qq\asymp1$ and $(1-\qq)\asymp (1-\pp)$,  we obtain for some constant $c'=c'(r)$ and large $n$
	$$\widetilde {V_n^{(r)}}=\frac{ V_n^{(r)}-\E[V_n^{(r)}]}{\sqrt{\Var[V_n^{(r)}]}}\ge -\frac{ \E[V_n^{(r)}]}{\sqrt{\Var[V_n^{(r)}]}}\ge -c\frac{\binom{n}{ r}(1-\qq)}{\sqrt{(1-\pp)(n^r\qq)^2/\ble}}\ge -{c'}\sqrt{(1-\pp)\ble}, $$
	which follows $(1-\pp)\ble\rightarrow\infty$. This concludes the proof.
\end{proof}
\begin{lemma}\label{lem:1/2->1/4}We have
	$$\bigg[\left(\frac{1}{(1-\pp)\ble}\right)+\frac1n+\frac1{nmp}\bigg]^{1/2}\wedge1\lesssim \left(\frac{1}{(1-\pp)\ble}\right)^{1/4}.$$
\end{lemma}
\begin{proof}Since $x^{1/2}\wedge1\leq x^{1/4}$,  $x\geq0$,  we get
	\begin{align*}
		\bigg[\left(\frac{1}{(1-\pp)\ble}\right)+\frac1n+\frac1{nmp}\bigg]^{1/2}\wedge1&\le \bigg[\left(\frac{1}{(1-\pp)\ble}\right)^{1/2}\wedge1\bigg]+\left(\frac{1}{n}\right)^{1/2}+\left(\frac1{nmp}\right)^{1/2}\\
		&\le \left(\frac{1}{(1-\pp)\ble}\right)^{1/4}+\left(\frac{1}{n^2\pp}\right)^{1/4}+\left(\frac1{m\big(n^2(mp^2)\big)}\right)^{1/4}\\
		&\lesssim\left(\frac{1}{(1-\pp)\ble}\right)^{1/4}, 
	\end{align*}
	where we also used $mp^2\gtrsim \pp$,  cf. \eqref{eq:pp},  and $\ble\leq n^2\pp$.
\end{proof}

In order to prove Theorem~\ref{Thm:Main},  it now suffices to establish the bound \eqref{eq:main}. Indeed,  since convergence of the distance between $\widetilde \X$ and $\mathcal N$ -- both in the Wasserstein and Kolmogorov metrics -- to zero implies the convergence $\widetilde\X \stackrel{\mathcal D}{\longrightarrow} \mathcal N$,  the estimate \eqref{eq:main},  together with Lemma~\ref{lem:necessary},  yields the asymptotic normality characterization. Furthermore,  since $d_{K/W}(X,  Y)\le 2$ for any standardized random variables $X$ and $Y$,  the bound \eqref{eq:main},  combined with Lemma~\ref{lem:1/2->1/4},  gives the estimate \eqref{eq:main2}.

Thus,  in view of Lemma~\ref{Lemma:Main}, Theorem~\ref{Thm:Main}  follows from Lemma~\ref{Lem:IntegralUpperGraphs} stated below. We devote the remainder of this section to its proof.

\begin{lemma}\label{Lem:IntegralUpperGraphs}
	Let $G_t=G_{t,   i,   j}$ with $t=1$ and $0\le i < j\le r$ or $t=2,    3$ and $1\le i<j\le r$. Then for $mp^3\le c_r$
	\begin{align}\label{eq:Lem:IntegralUpperGraphs}
		n^{|V(G_{t})|}\sum_{\II\subseteq\indeksy}(-1)^{|\II|}\qq^{|\indeksy|-|\II|}\pi(G_{t,   \II})
		\lesssim \left(\frac{1}{(1-\pp)\ble}+\frac{1}{n}+\frac1{nmp}\right)(\Var N)^2.
	\end{align}
\end{lemma}
We remark that from the method of the proof of Lemma~\ref{Lemma:Main} (see \cite{Edges}) it follows that the  sum above is nonnegative.
We split the proof into several cases. Each case is considered  in a separate subsection. In some cases a very delicate estimation will be required. However,  it will be often sufficient to  use the following reasoning.
For and any $\II\subseteq \indeksy$,   Lemma~\ref{Lem:Division2} lets us bound
$$
\qq^{|\indeksy|-|\II|}\pi(G_{t,   \II})\le \pi(G_t), 
$$
which follows
\begin{align}\label{Eq:4th:ReductionToGt}
	\sum_{\II\subseteq\indeksy}(-1)^{|\II|}\qq^{|\indeksy|-|\II|}\pi(G_{t,   \II})
	\le
	\sum_{\II\subseteq\indeksy}\qq^{|\indeksy|-|\II|}\pi(G_{t})
	\le 
	2^{|\Gamma|}\pi(G_t) \lesssim \pi(G_t).
\end{align}
Let us also note that (see Figure~\ref{Fig:GraphRepresentations}))
\begin{align}\label{eq:|V(G)|}
	\begin{array}{rl}
		|V(G_1)|&=4r-3j+i, \\
		|V(G_2)|&=4r-2j, \\
		|V(G_3)|&=4r-j-i.
	\end{array}
\end{align}

The outline of the proof of Lemma \ref{Lem:IntegralUpperGraphs} is as follows. 
First,  in Section~\ref{Sec:Proof:SmallMp},  we fully address the case $mp \le 1$. 
When $mp \ge 1$,  we pay special attention to two specific cases that require a more sophisticated approach:
\begin{enumerate}
	\item any two graphs $\kli_{ab}$,  $ab \in \indeksy$,  intersect on at most one vertex (Section~\ref{Sec:Proof:IntOnOne}), 
	\item the graphs $G_{3, 1, j}$ for $3 \le j \le r$ (Section~\ref{Sec:Proof:G3}).
\end{enumerate}
The remaining cases are considered separately: 
for $mp^2 \le 1$ in Section~\ref{Sec:Proof:SmallMp2},  
and for $mp^2 \ge 1$ in Section~\ref{Sec:Proof:LargeMp2}.

\subsection{Case $mp\le 1$}\label{Sec:Proof:SmallMp}

For $t=1$,    $0\le i<j\le r-1$,    by Corollary~\ref{Cor:Glueing}  and Theorem~\ref{Tw:Wariancja} we have
\begin{align*}
	n^{4r-3j+i}\pi(G_{1})
	\asymp
	n^{4r-3j+i}mp^{4r-3j+i}
	\asymp 
	\frac{(n^rmp^r)^4}{(n^jmp^j)^2n^{j-i}mp^{j-i}}
	\asymp
	\frac{(n^r\qq)^4}{\ble_j^2\ble_{j-i}} \lesssim \frac{(\Var N)^2}{\ble}.
\end{align*} 
Similarly,    for $t=2,    3$ and $1\le i < j \le r$ we have 
\begin{align*}
	n^{4r-2j}\pi(G_{2})\asymp n^{4r-2j}mp^{4r-2j}
	\asymp
	\frac{(n^r\qq)^4}{\ble_j\ble_{j-i}\ble_i} \lesssim \frac{(\Var N)^2}{\ble}, 
\end{align*}
and
\begin{align*}
	n^{4r-j-i}\pi(G_{3})\asymp n^{4r-j-i}mp^{4r-j-i}
	\asymp
	\frac{(n^r\qq)^4}{\ble_{j-i}\ble_i^2} \lesssim \frac{(\Var N)^2}{\ble}.
\end{align*} 
In view of  \eqref{Eq:4th:ReductionToGt},  this proves Lemma \ref{Lem:IntegralUpperGraphs} for $mp\le 1$.

\subsection{No all-single  intersections and $mp\ge 1$,   $mp^2\le 1$}\label{Sec:Proof:SmallMp2}

\subsubsection{ No single intersections}
Here we assume that if two cliques $\kli_{ab}$,   $ab\in \indeksy$,   in $G_{t,   i,   j}$ intersect,  then they share  at least two vertices,    i.e.
\begin{align*}
	\begin{array}{l}t=1,  \\i\ge 0,  \\i+2\le j\le r, 
	\end{array}
	\quad\quad\text{or}
	\quad\quad
	\begin{array}{l}
		t=2,   3, \\ i\ge 2, \\ i+2\le j\le r.
	\end{array}
\end{align*}
We will show that 
\begin{align}\label{Eq:LargeInt:Main}
	n^{|V(G_{t})|}\pi(G_t, \cover) 
	\lesssim \frac{(n^r\qq)^4}{\ble^3}.
\end{align}
In view of \eqref{Eq:SubgraphProb}, \eqref{Eq:4th:ReductionToGt}, and Theorem~\ref{Tw:Wariancja},  this proves Lemma~\ref{Lem:IntegralUpperGraphs} in this case.

We will use Lemma~\ref{Lem:Odrywanie} three times.
First, by Lemma~\ref{Lem:Odrywanie}, there exists a connected $G_{t, \II}$, $|\II|=3$, such that  
\begin{align}
	n^{|V(G_t)|}\pi(G_t, \cover) = n^{|V(G_{t, \indeksy})|}\pi(G_{t, \indeksy}, \cover)
	\lesssim \frac{n^r\qq n^{|V(G_{t, \II})|}\pi(G_{t, \II})}{\ble}.
\end{align}
Given $t, i, j$, all connected  $G_{t, \II}$, $|\II|=3$, are isomorphic. Therefore, after summing over all $\cover \in \covers(G_t)$, by \eqref{Eq:SubgraphProb} we obtain
\begin{align}\label{Eq:LargeInt:4}
	n^{|V(G_t)|}\pi(G_t)
	\lesssim \frac{n^r\qq n^{|V(G_{t, \{11,12,21\}})|}\pi(G_{t, \{11,12,21\}})}{\ble}.
\end{align}
Similarly, by Lemma~\ref{Lem:Odrywanie}, for any $\cover\in \covers(G_{t, \{11,12,21\}})$ there exists $\II\subseteq \{11,12,21\}$, $|\II|=2$, such that $G_{t,\II}$ is connected and 
\begin{align}\label{Eq:LargeInt:3}
	n^{|V(G_{t, \{11,12,21\}})|}\pi(G_{t, \{11,12,21\}}, \cover) 
	\lesssim \frac{n^r\qq n^{|V(G_{t, \II})|}\pi(G_{t, \II})}{\ble}.
\end{align}
Moreover, for any $\II\subseteq \indeksy$ such that $|\II|=2$ and $G_{t, \II}$ is connected,  by Lemma~\ref{Lem:Odrywanie} and \eqref{Eq:SubgraphProb}, we have
\begin{equation}
	n^{|V(G_{t, \II})|}\pi(G_{t, \II})=\sum_{\cover\in\covers(G_{t, \II})}n^{|V(G_{t, \II})|}\pi(G_{t, \II}, \cover)
	\lesssim \frac{(n^r\qq)^2}{\ble}.
\end{equation}
Combining this with  \eqref{Eq:LargeInt:3},  we obtain for all $\cover\in \covers(G_{t, \{11,12,21\}})$
\begin{align}
	n^{|V(G_{t, \{11,12,21\}})|}\pi(G_{t, \{11,12,21\}}, \cover) 
	\lesssim \frac{(n^r\qq)^3}{\ble^2},
\end{align}
and consequently, by \eqref{Eq:SubgraphProb}, we have
\begin{align}\label{Eq:LargeInt:3bis}
	n^{|V(G_{t, \{11,12,21\}})|}\pi(G_{t, \{11,12,21\}}) 
	\lesssim \frac{(n^r\qq)^3}{\ble^2}.
\end{align} 
Finally, \eqref{Eq:LargeInt:Main} follows by  \eqref{Eq:LargeInt:4} and \eqref{Eq:LargeInt:3bis}.

\subsubsection{{Both: single and non single intersections}}
\label{sec:mixed}
Here we consider the cases when

\begin{align*}
	\begin{array}{l}t=1,  \\i\ge 1,  \\ j=i+1\le r-1, 
	\end{array}
	\quad\quad\text{or}
	\quad\quad
	\begin{array}{l}
		t=2,   3, \\ i\ge 2, \\ j=i+1\le r.
	\end{array}
	\quad\quad\text{or}
	\quad\quad
	\begin{array}{l}
		t=2, \\ i=1, \\ i+2\le j\le r.
	\end{array}
\end{align*}
Since $G_{2,   j-1,   j}\cong G_{2,   1,   j}$, we may assume that $j=i+1$ in all considered cases. 

As previously,    we use \eqref{Eq:4th:ReductionToGt},  so it suffices to estimate the probabilities $\pi(G_t)$. For $mp^2\le 1$ and $mp\ge 1$,    by the identity \eqref{Eq:SubgraphProb},  first two inequalities of Lemma~\ref{Lem:Division},       the variance representation \eqref{Eq:Var:FormulaGeneral2} and Lemma \ref{lem:Var2--r},  we get
\begin{align*}
	n^{4r-2(i+1)-1}\pi(G_{1,   i,   i+1})
	&\lesssim 
	\frac{
		n^{2r-(i+1)}\pi(\kli_{11}\oplus_{i+1}\kli_{12})
		n^{2r-(i+1)}\pi(\kli_{21}\oplus_{i+1}\kli_{22})}
	{n}&&\\ 	
	&\lesssim
	\frac{(\Var N)^2}{n},    &i&\ge 1, \\
	n^{4r-2-2i}\pi(G_{2,   i,   i+1})
	&\lesssim 
	\frac{
		n^{2r-i}\pi(\kli_{11}\oplus_{i}\kli_{12})
		n^{2r-i}\pi(\kli_{21}\oplus_{i}\kli_{22})}
	{n^2mp^2}\lesssim
	\frac{(\Var N)^2}{\ble},    &i&\ge 2, \\ 	
	n^{4r-1-2i}\pi(G_{3,   i,   i+1})
	&\lesssim 
	\frac{
		n^{2r-i}\pi(\kli_{11}\oplus_{i}\kli_{12})
		n^{2r-i}\pi(\kli_{21}\oplus_{i}\kli_{22})}
	{n}\lesssim
	\frac{(\Var N)^2}{n},    &i&\ge 2, 
\end{align*}
which gives us the required bound in this case.

\subsection{No all-single intersections and $mp^2\ge1$}\label{Sec:Proof:LargeMp2}
In this subsection we deal with $mp^2\ge 1$,   $mp^3\le c_r$ and all the parameters $t,  i,  j $ excluding those analysed in Sections~\ref{Sec:Proof:IntOnOne} and \ref{Sec:Proof:G3}.
We start with a general result.
\begin{lemma}\label{lem:(1-p)^2}
	For $t=1,  j= i+1$ or $t=2,  3$ we have
	$$\sum_{\II\subseteq\indeksy}(-1)^{|\II|}\qq^{|\indeksy|-|\II|}\pi(G_{t,   \II})\lesssim (1-\pp)^2, $$
	while if $t=1$ and $j\ge i+2$ it holds that
	$$\sum_{\II\subseteq\indeksy}(-1)^{|\II|}\qq^{|\indeksy|-|\II|}\pi(G_{1,   \II})\lesssim 1-\pp.$$
\end{lemma}
\begin{proof}
	We will determine the values of $\pi(G_{t, \II})$ for all $t$ and $\II$, up to a small error, by means of clique probabilities. For $|\II|\le 1$ it is straight forward. 
	For $|\II|\ge 2$ we will use Lemma~\ref{Lem:DuzeMp2}. For $\II=\indeksy$ we obtain (cf. Figure \ref{Fig:GraphRepresentations})
	\begin{align*}
		\pi(G_{t, \indeksy})
		&=\frac{\pi(G_{t, \{11, 12\}})\cup \pi(G_{t, \{21, 22\}})}{\pi(G_{t, \{11, 12\}})\cap \pi(G_{t, \{21, 22\}})}+O(mp^3e^{-2mp^2})\\
		&=\frac{\qq^4}{\pi(G_{t, \{11, 12\}})\cap \pi(G_{t, \{21, 22\}})\pi(\kli_{11}\cap\kli_{12})\pi(\kli_{21}\cap\kli_{22})}+O(mp^3e^{-2mp^2})\\
		&=O(mp^3e^{-2mp^2})+
		\begin{cases}
			\frac{\qq^4}{\pi^2(K_j)\pi(K_{j-i})}&\text{ for }t=1;\\
			\frac{\qq^4}{\pi^2(K_i)\pi^2(K_{j-i})}&\text{ for }t=2;\\
			\frac{\qq^4}{\pi^2(K_i)\pi(K_{j-i})}&\text{ for }t=3.\\
		\end{cases}
	\end{align*}
	We remark that in the case $t=1$,  $j=r$,  since by Fact~\ref{Fact:G1r} we have $\kli_{a1}\cap\kli_{a2}\cong\klik_r$,   $a=1, 2$,  some of the equalities are trivially true. Moreover,  if in addition $i=0$,  the term  $O(mp^3e^{-2mp^2})$ may be omitted. Analogous  observations apply to the case $t=1$ and $j=r$ in the following estimates.

	Similarly,  we obtain for $|\II|=3$
	\begin{align*}
		\pi(G_{t, \II})
		&=O(mp^3e^{-2mp^2})+
		\begin{cases}
			\frac{\qq^3}{\pi(K_j)\pi(K_{j-i})}&\text{ for }t=1;\\
			\frac{\qq^3}{\pi(K_i)\pi(K_{j-i})}&\text{ for }t=2;\\
			\frac{\qq^3}{\pi(K_i)\pi(K_{j-i})}&\text{ for }t=3\text{ and }\{11, 21\}\subseteq \II;\\
			\frac{\qq^3}{\pi(K_i)}&\text{ for }t=3\text{ and }\{12, 22\}\subseteq \II,
		\end{cases}
	\end{align*}
	and for $|\II|=2$ and $x,  y \in\indeksy$
	\begin{align*}
		\pi(G_{t, \{x, y\}})
		&=\frac{\qq^2}{\pi(\kli_x\cap\kli_y)}+O(mp^3e^{-2mp^2}).
	\end{align*}
	In the sequel, we will consistently use the following relation,  which follows from Fact \ref{Lem:DuzeMp2dwa}. For any non-empty $H$ and   $mp^2\ge 1$ we have
	$${\frac{1}{\pi(H)}-1}\asymp 1-\pp.$$

	Using all the above estimates, together with $1-\pp\asymp e^{-mp^2}$,  we obtain for $t=1$,  $mp^2\ge1$,  and $mp^3\le1$
	\begin{align*}
		&\sum_{\II\subseteq\indeksy}(-1)^{|\II|}\qq^{|\indeksy|-|\II|}\pi(G_{1,   \II})=\\
		&=\qq^4\left(\frac{1}{\pi^2(\klik_j)\pi(\klik_{j-i})}- \frac{4}{\pi(\klik_j)\pi(\klik_{j-i})}+\left(\frac{2}{\pi(\klik_j)}+\frac{4}{\pi(\klik_{j-i})}\right)-4+1\right)
		+ O(mp^3e^{-2mp^2})\\
		&=\qq^4\left[\frac{1}{\pi(\klik_{j-i})}\left(\frac{1}{\pi(\klik_j)}-1\right)^2
		+\left(\frac{1}{\pi(\klik_{j-i})}-1\right)\left(3-\frac{2}{\pi(\klik_{j})}\right)
		\right]+ O(mp^3e^{-2mp^2})\\
		&\lesssim (1-\pp)^2+(1-\pp)+(1-\pp)^2\lesssim 1-\pp.
	\end{align*}
	In particular,  this implies the latter bound from the assertion. Furthermore,  for $j=i+1$ we have $\pi(K_{j-i})=\pi(N_1)=1$,  and then
	\begin{align*}
		\sum_{\II\subseteq\indeksy}(-1)^{|\II|}\qq^{|\indeksy|-|\II|}\pi(G_{1,   \II})=\frac{\qq^4}{\pi(\klik_{j-i})}\left(\frac{1}{\pi(\klik_j)}-1\right)^2
		+ O(mp^3e^{-2mp^2})\lesssim (1-\pp)^2, 
	\end{align*}
	which gives us the first bound from the assertion in the case $t=1$. For $t=2$ we proceed similarly, which leads to
	\begin{align*}
		&\sum_{\II\subseteq\indeksy}(-1)^{|\II|}\qq^{|\indeksy|-|\II|}\pi(G_{2,   \II})
		\\
		&=\qq^4\left[\frac{1}{\pi^2(\klik_i)\pi^2(\klik_{j-i})}-\frac{4}{\pi(\klik_i)\pi(\klik_{j-i})}+\left(2+\frac{2}{\pi(\klik_{i})}+\frac{2}{\pi(\klik_{j-i})}\right)-4+1\right]\\
		&\quad+ O(mp^3e^{-2mp^2})\\
		&=\qq^4\left[\left(\frac{1}{\pi(\klik_i)\pi(\klik_{j-i})}-1\right)^2
		-2\left(\frac{1}{\pi(\klik_{j-i})}-1\right)\left(\frac{1}{\pi(\klik_{i})}-1\right)
		\right]
		+ O(mp^3e^{-2mp^2})\\
		&\lesssim (1-\pp)^2.
	\end{align*}
	Eventually,  for $t=3$ we get
	\begin{align*}
		&\sum_{\II\subseteq\indeksy}(-1)^{|\II|}\qq^{|\indeksy|-|\II|}\pi(G_{3,   \II})\\
		&=\qq^4
		\left[
		\frac{1}{\pi^2(\klik_i)\pi(\klik_{j-i})}
		-
		\left(\frac{2}{\pi(\klik_i)\pi(\klik_{j-i})}+\frac{2}{\pi(\klik_i)}\right)+\left(\frac{2}{\pi(\klik_i)}+3+\frac{1}{\pi(\klik_{j-i})}\right)-4+1\right]\\&\quad+ O(mp^3e^{-2mp^2})\\
		&=\qq^4\frac{1}{\pi(\klik_{j-i})}\left(\frac{1}{\pi(\klik_i)}-1\right)^2 + O(mp^3e^{-2mp^2})\lesssim (1-\pp)^2, 
	\end{align*}
	as required.
\end{proof}
The parameters considered in this section may be summarised in the following manner:
\begin{align*}&(\alpha)\quad
	\begin{array}{l}t=1,  \\i\ge 1,  \\ j=i+1\le r-1, 
	\end{array}
	\quad\quad\text{or}
	\quad\quad
	\begin{array}{l}
		t=2,   \\ i\ge 1, \\ 3\vee(i+1)\le j\le r, 
	\end{array}
	\quad\quad\text{or}
	\quad\quad
	\begin{array}{l}
		t=3, \\ i\ge2, \\ i+1\le j\le r, 
	\end{array}\\
	&(\beta)\quad
	\begin{array}{l}t=1,  \\i\ge 0,  \\ i+2\le j\le r-1.
	\end{array}
\end{align*}

Consider the case $(\alpha)$. By \eqref{eq:|V(G)|} we obtain
\begin{align*}
	|V(G_1)|&=4r-2j-(j-i)\le 4r-2\cdot2-1\le 4r-5, \\
	|V(G_2)|&=4r-2j\le 4r-2\cdot3\le 4r-6, \\
	|V(G_3)|&=4r-j-i\le 4r-3-2\le 4r-5.
\end{align*}
Thus,  Lemma \ref{lem:(1-p)^2} and \eqref{eq:Var2} give us
\begin{align*}
	n^{|V(G_{t})|}\sum_{\II\subseteq\indeksy}(-1)^{|\II|}\qq^{|\indeksy|-|\II|}\pi(G_{t,   \II})
	\lesssim n^{4r-5}(1-\pp)^2=\frac{(n^{2r-2}(1-\pp))^2}{n}\lesssim \frac{(\Var [\X])^2}{n}.
\end{align*}
In the case $(\beta)$,  we have
$$	|V(G_1)|=4r-3j+i\leq 4r-6, $$
and hence,  by the latter part of Lemma \ref{lem:(1-p)^2},  
\begin{align*}
	n^{|V(G_1)|}\sum_{\II\subseteq\indeksy}(-1)^{|\II|}\qq^{|\indeksy|-|\II|}\pi(G_{1,   \II})
	\lesssim n^{4r-6}(1-\pp)=\frac{(n^{2r-2}(1-\pp))^2}{n^2(1-\pp)}\lesssim \frac{(\Var [\X])^2}{(1-\pp)\ble}.
\end{align*}

\subsection{All-single intersections}\label{Sec:Proof:IntOnOne}

Here we consider the cases $t=1$,    $i=0$,    $j=1$ and $t=2,    3$,    $i=1$,    $j=2$,  i.e., we assume that any two intersecting cliques $\kli_{ab}$,  $ab\in\indeksy$,  share exactly one vertex.  We start with a general discussion,    applicable to all ranges of $p$.

For each $\II\subseteq \indeksy$ and $t=1,    2,    3$,    let $H_{t,   \II}$ be a graph on $4r$ vertices consisting of 
$G_{t,   \II}$,    $|\indeksy|-|\II|$ vertex disjoint cliques of size $r$,    and possibly some isolated vertices. One can think of $G_{t,   \II}$ as of $G_t$ whose cliques $\kli_{ab}$,    $ab\in \indeksy\setminus\II$ have been separated. For convenience,  we will denote the additional cliques by $\kli_{ab}$,    $ab\in \indeksy\setminus\II$,  i.e., in the same manner as before separating them from $G_{t,   \II}$. For example  $H_{1,   \emptyset}=\kli_{11}\oplus_0\kli_{12}\oplus_0\kli_{21}\oplus_0\kli_{22}$ and $H_{1,   \{11,   12\}}=G_{1,   \{11,   12\}}\oplus_0\kli_{21}\oplus_0\kli_{22}\oplus_0\emp_j$.  Since vertex disjoint subgraphs appear independently in $\Gnmp$ and isolated vertices do not affect the subgraph probability,  we have
\begin{equation*}
	\pi(H_{t,   \II})=\qq^{|\indeksy|-|\II|}\pi(G_{t,   \II}). 
\end{equation*}
Thus,    the left-hand side of \eqref{eq:Lem:IntegralUpperGraphs} may be rewritten as  follows
\begin{equation}\label{Eq:SumHtGt}
	\sum_{\II\subseteq\indeksy}(-1)^{|\II|}\qq^{|\indeksy|-|\II|}\pi(G_{t,   \II})
	=
	\sum_{\II\subseteq\indeksy}(-1)^{|\II|}\pi(H_{t,   \II}).
\end{equation}

For each  $t=1,    2,    3$,  and $\II\subseteq \indeksy$,    we consider the natural partition $\{\kli_{ab}\}_{ab\in\Gamma}$ of $H_{t,   \II}$ (see Definition~\ref{Def:Partition}).  Let us recall (see Definition \ref{def:decomp}) that   $\covers_{\JJ}(H_{t,   \II})$,  $\JJ\subseteq\indeksy$,   is the family of clique covers of $H_{t,   \II}$  that are decomposable (see Definition \ref{Def:Partition}) to 
$$H_{t,   \II,   \JJ}{:=\Bigg(V(H_{t,   \II}),    \bigcup_{ab\in \JJ}E(\kli_{ab})\Bigg)}, $$
see \eqref{Eq:DecompDefHJ}. 
Namely,  $\cover\in \covers_{\JJ}(H_{t,   \II})$ if $\cover =\cover_{\JJ}\dot\cup \dot\bigcup_{ab\in \indeksy\setminus \JJ}\cover_{ab}$,  where $\cover_{ab}$ is a clique cover of $\kli_{ab}$ for any $ab\in\indeksy\setminus\JJ$, and $\cover_J$  is indecomposable with relation to the partition $\{\kli_{ab}\}_{ab\in J}$ of $H_{i,  \II,  \JJ}$. In particular,  $\covers_{\indeksy}(H_{t,   \II})$ is the family of all indecomposable clique covers of $H_{t,   \II}$ in relation to the partition $\{\kli_{ab}\}_{ab\in\indeksy}$. 

According to the definition,  the families $\covers_\JJ(H_{t,   \II})$,    $\JJ\subseteq \indeksy$,    are pairwise disjoint and $\covers(H_{t,   \II})=\dot\bigcup_{\JJ\subseteq \indeksy} \covers_\JJ(H_{t,   \II})$. 
Hence,  by \eqref{Eq:SumHtGt},  we obtain
\begin{equation}\label{Eq42:CalkaDoSumaHIJ}
	\begin{split}
		\sum_{\II\subseteq\indeksy}(-1)^{|\II|}\qq^{|\indeksy|-|\II|}\pi(G_{t,   \II})
		&=\sum_{\II\subseteq \indeksy}(-1)^{|\II|}\sum_{\JJ\subseteq\indeksy}\pi(H_{i,   \II},     \covers_{\JJ}(H_{i,   \II}))\\
		&=\sum_{\emptyset\subseteq\JJ\subseteq\indeksy}\sum_{\II\subseteq \indeksy}(-1)^{|\II|}\pi(H_{i,   \II},     \covers_{\JJ}(H_{i,   \II})).
	\end{split}
\end{equation}
The assertion of Lemma \ref{Lem:IntegralUpperGraphs} under current assumption is now a consequence of the following lemma. Namely,  we show that all the inner sums above vanish if $\JJ\neq\indeksy$,  and provide a bound for $\JJ=\indeksy$.

\begin{lemma}\label{Lem:HIJprawdopodobienstwa}For $ \emptyset\subseteq\JJ\subsetneq\indeksy$  and $t=1,    2,    3$ we have
	$$\sum_{\II\subseteq \indeksy}(-1)^{|\II|}\pi(H_{t,   \II},     \covers_{\JJ}(H_{t,   \II}))=0.$$ 
	Moreover,    for $mp\ge1$ and $mp^3 \le c_r$  it holds that 
	\begin{align*}
		n^{|V(G_t)|}\left|\sum_{\II\subseteq \indeksy}(-1)^{|\II|}\pi(H_{t,   \II},     \covers_{\indeksy}(H_{t,   \II}))\right|
		&\lesssim
		\frac{(\Var N)^2}{n}. 	
	\end{align*}	
\end{lemma}
In the proof of the lemma,  we will take advantage of Lemma~\ref{Lem:DecompositionComparing}. For that,  we first derive an auxiliary result that explains when the graphs under consideration are isomorphic.
\begin{fact}\label{Fact:Isomorphism}
	Let $\II',   \II'',   \JJ\subseteq \indeksy$ be such that $\II'\cap\JJ=\II''\cap\JJ$.    Then 
	$$
	H_{t,   \II',   \JJ}\cong H_{t,   \II'',   \JJ}.
	$$	
\end{fact}
\begin{proof}
	Recall that $H_{t,   \II}$ consists of $G_{t,   \II}$ (the subgraph of $G_t$ induced on cliques with indices from $\II$),     $|\indeksy|-|\II|$ isolated cliques,    and possibly some isolated vertices. Therefore $H_{t,   \II,   \JJ}$ consists of $G_{t,   \II\cap \JJ}$,    $|\JJ\setminus \II|$ isolated cliques,    and possibly some isolated vertices. If $\II'\cap\JJ=\II''\cap\JJ$,    then by definition $G_{t,   \II'\cap \JJ}\cong G_{t,   \II''\cap \JJ}$,    $|\JJ\setminus \II'|=|\JJ|-|\JJ\cap \II'|=|\JJ|-|\JJ\cap \II''|=|\JJ\setminus \II''|$,    and $|V(H_{t,   \II',   \JJ})|=|V(H_{t,   \II'',   \JJ})|=4r$. This implies the required isomorphism. 	
\end{proof}
\begin{proof}[Proof of Lemma~\ref{Lem:HIJprawdopodobienstwa}]
	Let us rewrite the sum from the assertion by taking into consideration all  the possible  intersections of $\II$ and $\JJ$. Given $\JJ\subsetneq \indeksy$ we have	
	\begin{align}
		\sum_{\II\subseteq \indeksy}(-1)^{|\II|}\pi(H_{t,   \II},     \covers_{\JJ}(H_{t,   \II}))&
		=\sum_{\JJ^*\subseteq \JJ}\sum_{\substack{\II\subseteq\indeksy\\  \II\cap \JJ=\JJ^*}}(-1)^{|\II|}\pi(H_{t,   \II},     \covers_{\JJ}(H_{t,   \II}))\\
		\label{Eq42:SumaHIJdwa}
		&=\sum_{\JJ^*\subseteq \JJ}(-1)^{|\JJ^*|}\sum_{\substack{\II\subseteq\indeksy\\  \II\cap \JJ=\JJ^*}}(-1)^{|\II\setminus \JJ|}\pi(H_{t,   \II},     \covers_{\JJ}(H_{t,   \II})).
	\end{align}
	Set  $\JJ^*\subseteq \JJ$,    and let us take 
	$\II',  \II''\subset \Gamma$ such that
	$\II'\cap \JJ=\II''\cap \JJ=\JJ^*$. Consider a natural partition $\{\kli_{ab}\}_{ab\in \indeksy}$ of both $H_{t,   \II'}$ and $H_{t,   \II''}$.
	By Fact~\ref{Fact:Isomorphism} and Lemma~\ref{Lem:DecompositionComparing} we get that 
	$$
	\forall_{\II',   \II'',    \II'\cap\II''}\pi(H_{t,   \II'},     \covers_{\JJ}(H_{t,   \II'}))=\pi(H_{t,   \II''},     \covers_{\JJ}(H_{t,   \II''})).
	$$
	Consequently,  
	\begin{align*}
		\sum_{\substack{\II\subseteq\indeksy\\  \II\cap \JJ=\JJ^*}}(-1)^{|\II\setminus \JJ|}\pi(H_{t,   \II},     \covers_{\JJ}(H_{t,   \II}))
		&=\pi(H_{t,   \JJ^*},     \covers_{\JJ}(H_{t,   \JJ^*}))
		\sum_{\substack{\II\subseteq\indeksy\\  \II\cap \JJ=\JJ^*}}(-1)^{|\II\setminus \JJ|}
		\\
		&=\pi(H_{t,   \JJ^*},     \covers_{\JJ}(H_{t,   \JJ^*}))\sum_{l=0}^{|\indeksy\setminus \JJ|}\binom{|\indeksy\setminus \JJ|}{l}(-1)^{l}=0.
	\end{align*}
	Consider now $\JJ=\indeksy$ and let us recall that $\covers_{\indeksy}(H_{t,   \II})$ consists of all indecomposable clique covers of $H_{t,   \II}$ in relation to the partition $\{\kli_{ab}\}_{ab\in\indeksy}$. 
	
	In the case $mp^2\le 1$,    we will use Lemma~\ref{Lem:InDecomposable}. The graph $F(G_t,   (\kli_{ab})_{ab\in\indeksy})$ (see Definition \ref{def:F}) is a star with $4$ leaves for $t=1$,    a cycle on 8 vertices for $t=2$,   and  a  path on $7$ vertices for $t=3$. Furthermore,  for any $\II\subsetneq\indeksy$,    $F(H_{t,   \II},   (\kli_{ab})_{ab\in\indeksy})$ is a proper subgraph of $F(G_t,   (\kli_{ab})_{ab\in\indeksy})$,  and  therefore  does not contain any cycle,    except the case $t=2$ and $\II=\indeksy$. Thus,  by Lemma~\ref{Lem:InDecomposable} and Theorem~\ref{Tw:Wariancja},  and observing that $|V(G_1)|=|V(G_3)|=4r-3$,  $|V(G_2)|=4r-4$,  we get
	\begin{align*}
		&n^{|V(G_1)|}\sum_{\II\subseteq \indeksy}\pi(H_{1,   \II},     \covers_{\indeksy}(H_{1,   \II}))+	n^{|V(G_3)|}\sum_{\II\subseteq \indeksy}\pi(H_{3,   \II},     \covers_{\indeksy}(H_{3,   \II}))\\
		&\lesssim
		\frac{(n^r\qq)^4}{n^3(mp)^2}
		=\frac{(n^r\qq)^4}{nm(n^2mp^2)}
		\asymp
		\frac1{nm}\left(\frac{(n^r\qq)^2}{\ble_2}\right)^2
		\lesssim \frac{(\Var N)^2}{n}, 
	\end{align*}
	and
	\begin{align*}
		&n^{|V(G_2)|}\sum_{\II\subseteq \indeksy}\pi(H_{2,   \II},     \covers_{\indeksy}(H_{2,   \II}))\\
		&\lesssim
		\frac{(n^r\qq)^4}{n^4(mp)^2p}
		=\frac{(n^r\qq)^4}{n(n^2mp^2)(nmp)}	
		\asymp\frac1n
		\frac{(n^r\qq)^2}{\ble_2}\frac{(n^r\qq)^2}{nmp}
		\lesssim \frac{(\Var N)^2}{n}.
	\end{align*}
	
	Consider $mp^2\ge 1$. Since $E(\kli_{ab})$ is a clique cover of $\kli_{ab}$,  $ab\in\indeksy$,  every indecomposable clique cover has at least $4$ forbidden 2--sets (each for every $\kli_{ab}$,    $ab\in \indeksy$). Moreover,  edges that are not covered by $2$--sets (the forbidden ones) have to be covered by some larger sets. Therefore,    in an indecomposable clique cover of any $H_{t,   \II}$,    in order to cover four edges,    there must be at least two sets of size at least $3$ or at least one large set of size at least $5$    for $t=1,    3$, and of size at least  $4$ for $t=2$,  as in the last case a cycle containing an edge from each of $\kli_{ab}$,  $ab\in\indeksy$,  is possible. Therefore,    by \eqref{eq:CC4} and \eqref{eq:Var2},    for every indecomposable clique cover $\cover$ of  $H_{t,   \II}$ we have 
	\begin{align*}
		n^{|V(G_{t,   \indeksy})|}\pi(H_{t,   \II},   \cover)\lesssim n^{4r-3} (mp^4+(mp^3)^2)e^{-4mp^2}\lesssim \frac{(n^{2r-1}mp^3e^{-2mp^2})^2}{n} \lesssim \frac{(\Var N)^2}{n}.
	\end{align*}    
	Here, we used the inequality $mp^4\le  (mp^3)^2$,  which follows from the assumption $mp^2\geq1$. In view of \eqref{Eq:SubgraphProbabilityCCfamily},  this concludes the proof.
\end{proof}

\subsection{Case of $G_{3,   1,   j}$}\label{Sec:Proof:G3}
First,  let us explain why this case is not included in Section~\ref{sec:mixed}. Namely,  the method employed there fails in this case,  since $\pi\left(\kli_{21}\oplus_1\kli_{22}\right)$ is in general not comparable with $\pi\left(\kli_{21}\oplus_1\kli_{22}\right)-\qq^2$. On the other hand,  the methods from this subsection are similar to those from the previous one,  however,  different partitions are used,  which changes many details in the approach.

Let $G_3=G_{3,   1,   j}$ for $j\ge 3$. One more time we consider the sum from \eqref{Eq:SumHtGt}. The number of its terms is $2^{|\Gamma|}=16$,  however,  we will show that only $4$ of them matter. Namely,  we split the sum  into a double one as follows
\begin{align}\label{6.5:eq:1}
	\sum_{\II\subseteq\indeksy}(-1)^{|\II|}\pi(H_{3,   \II})
	&=\sum_{\II_1\subseteq\{11,   21\}}(-1)^{|\II_1|}\sum_{\II_2\subseteq \{21,   22\}}(-1)^{|\II_2|}\pi(H_{3,   \II_1\cup\II_2}).
\end{align}
We observe that   for $F_0=K_r\oplus_0K_r\oplus_0K_r\oplus_0K_r$ and $F_1=(K_r\oplus_1 K_r)\oplus_0K_r\oplus_0K_r$
\begin{align*}
	\II_1&=\{11\},  	&H_{3,  \{11\}\cup\{12,  22\}}&\cong H_{3,  \{11\}\cup\{12\}}\cong {F_1}\text{ and }
	H_{3,  \{11\}\cup\{22\}}\cong H_{3,  \{11\}\cup\emptyset}\cong {F_0};\\
	\II_1&=\{21\},  	&H_{3,  \{21\}\cup\{12,  22\}}&\cong H_{3,  \{21\}\cup\{22\}}\cong {F_1}\text{ and }
	H_{3,  \{21\}\cup\{12\}}\cong H_{3,  \{21\}\cup\emptyset}\cong {F_0};\\
	\II_1&=\emptyset,  	&H_{3,  \emptyset\cup\{12,  22\}}&\cong H_{3,  \emptyset\cup\{12\}}\cong H_{3,  \emptyset\cup\{22\}}\cong H_{3,  \emptyset\cup\emptyset}\cong {F_0}.\\
\end{align*}
Therefore, in \eqref{6.5:eq:1}   only the term with $\II_1=\{11,  21\}$ does not cancel out and
\begin{align}\label{Eq:G3:Glowne}
	\sum_{\II\subseteq\indeksy}(-1)^{|\II|}\pi(H_{3,   \II})
	&=\sum_{\II_2\subseteq \{12,   22\}}(-1)^{|\II_2|}\pi(H_{3,   \{11,   21\}\cup\II_2}).
\end{align}
Our aim now is  to find the upper bound for \eqref{Eq:G3:Glowne}. We will achieve this using the same reasoning as in Section~\ref{Sec:Proof:IntOnOne}.
For each  $H_{3,   \II}$,    $\{11,   21\}\subseteq\II\subseteq \indeksy$,    consider its partition $\{H_{3, \II, i}\}_{i\in \Lambda}$,    where $\Lambda=\{0,   1,   2\}$,    $H_{3, \II, 0}=G_{\{11,   21\}}=\kli_{11}\oplus_{j-1}\kli_{21}$,    $H_{3, \II, 1}=\kli_{12}$,    and  $H_{3, \II, 2}=\kli_{22}$.  As defined in \eqref{Eq:DecompDefHJ},  for $\LL\subseteq\Lambda$ we have
$$H_{3, \II, \LL}=H_{3,   \II}\setminus \bigcup_{i\in \Lambda\setminus \LL}E(H_{3, \II, i})=\Bigg(V(H_{3,   \II}),  \bigcup_{i\in  \LL}E(H_{3, \II, i})\Bigg).$$

\begin{figure}
	\begin{center}
		\includegraphics[width=13cm]{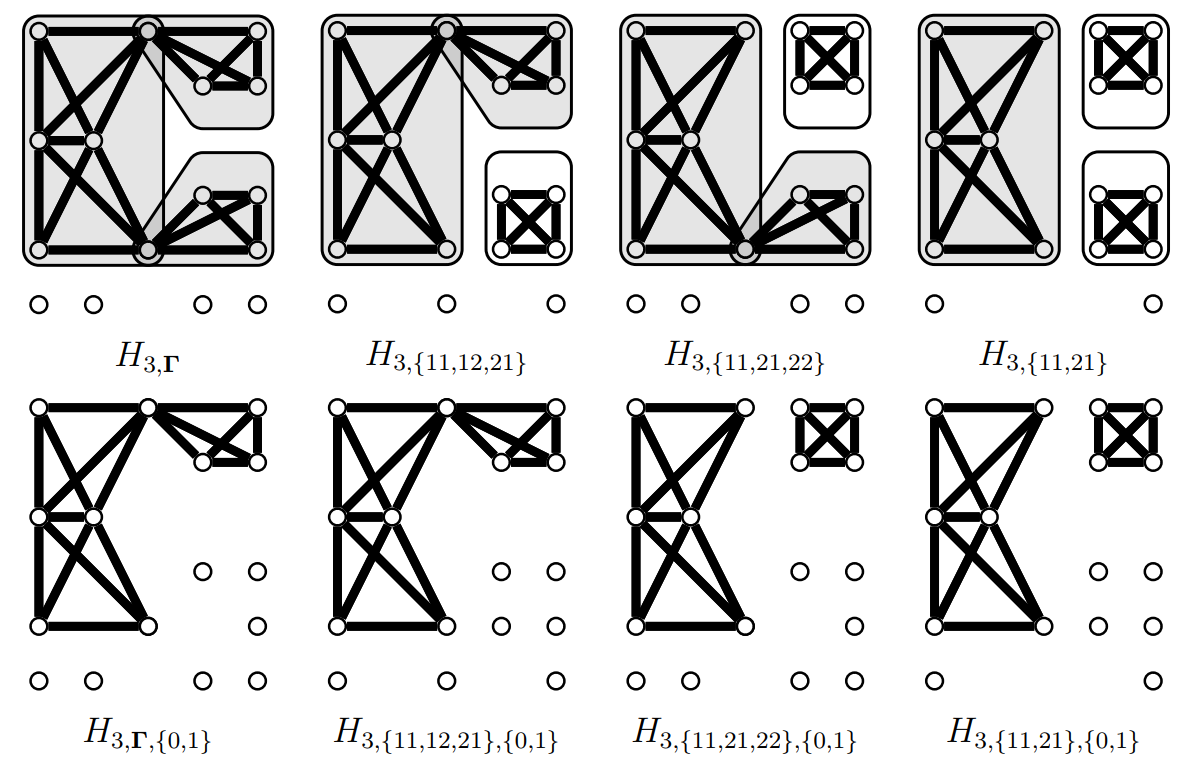}
	\end{center}
	\caption{\label{Fig:G3j1} In the first line are presented all graphs $H_{3, \II}$ for $\II$ such that $\{11,  21\}\subseteq \II$ and their partitions $\{H_{3, \II, i}\}_{i\in\Lambda}$.  Graphs constituting $G_{\II}$ are indicated in grey. Below,  there are depicted respective graphs $H_{3, \II, \LL}$ for $\LL=\{0,  1\}$.}
\end{figure}

For $\LL\subseteq\Lambda$,    let $\covers_{\LL}(H_{3,   \II})$ denote all clique covers decomposable to $H_{3,   \II,   \LL}$ in relation to the partition. Since  $\covers_{\LL}(H_{3,   \II})$,    for $\LL\subseteq\Lambda$,    are by definition pairwise disjoint and their union gives $\covers(H_{3,   \II})$,    for every $\II$ we have by \eqref{Eq:SubgraphProb}
\begin{equation}\label{Eq:G3:SumLL}
	\pi(H_{3, \II})=\sum_{\LL\subseteq \Lambda} \pi(H_{3,  \II},   \covers_{\LL}(H_{3,  \II})), 
\end{equation}
which,  applied to \eqref{Eq:G3:Glowne},  leads to 
\begin{align}\label{Eq:G3:SumLL2}
	\sum_{\{11,   21\}\subseteq\II\subseteq \indeksy}(-1)^{|\II|}\pi(H_{3,   \II})
	&=\sum_{\LL\subseteq \{0,   1,   2\}}\sum_{\{11,   21\}\subseteq\II\subseteq \indeksy}(-1)^{|\II|}\pi(H_{3,   \II},   \covers_{\LL}(H_{3,   \II})).
\end{align}
Next,  we are going to use Lemma~\ref{Lem:DecompositionComparing}. For any $J=\LL$ and $H_J^{(1)}=H_{3,  \II,  \LL}$,  $H_J^{(2)}=H_{3,  \II',  \LL}$,  $\II,  \II'\subset\indeksy$,  the right-hand side condition in \eqref{Eq:IsoDecompComp} is naturally satisfied. Hence,  $H_{3,  \II,  \LL}\cong H_J^{(2)}=H_{3,  \II',  \LL}$ suffices for $\pi(H_{3,   \II},   \covers_{\LL}(H_{3,   \II}))=\pi(H_{3,   \II'},   \covers_{\LL}(H_{3,   \II'}))$.
Let us therefore  study isomorphisms between  graphs  $H_{3,   \II,  \LL}$,  $\{11,   21\}\subseteq\II\subseteq \indeksy$ for any fixed $\LL\neq\Lambda$.
In the same manner as in Fact~\ref{Fact:Isomorphism},    or simply by checking one by one,    we may show the following isomorphisms (for  an example see Figure~\ref{Fig:G3j1}) 
\begin{align}
	H_{3,   \{11,   21\},    \LL}\cong H_{3,   \{11,   21,   12\},    \LL}
	&\cong H_{3,   \{11,   21,   22\},   \LL}\cong H_{3,   \indeksy,   \LL}&&\text{ for }\LL=\emptyset,   \{0\},  \{1,  2\},  \{1\},  \{2\} ;\\
	H_{3,   \{11,   21\},    \LL}\cong H_{3,   \{11,   21,   22\},    \LL}&\text{ and }
	H_{3,   \{11,   21,   12\},    \LL}\cong H_{3,   \indeksy,    \LL}&&\text{ for }\LL=  \{0,   1\};\\ 
	H_{3,   \{11,   21\},    \LL}\cong H_{3,   \{11,   21,   12\},    \LL}&\text{ and }
	H_{3,   \{11,   21,   21\},    \LL}\cong H_{3,   \indeksy,    \LL}&&\text{ for }\LL=  \{0,   2\}. \label{Eq:G3:Isomorphisms}
\end{align}
Consequently,  for $\LL=\emptyset,   \{0\},  \{1,  2\},  \{1\},  \{2\} $ the inner sum in \eqref{Eq:G3:SumLL2} consists of terms that differ only by their sign,  and it therefore vanishes. On the other hand,  for $\LL=\{0,  1\},  \{0,  2\}$ the terms of this sum might be paired up such that the paired terms also differ only by their signs,  so the sum vanishes again. Summing up,   for all $\LL\neq \Lambda$ the  the inner sum in \eqref{Eq:G3:SumLL2}  equals zero. Thus,  we obtain 
\begin{align}\label{Eq:G3:SumIndecomposible}
	\sum_{\{11,   21\}\subseteq\II\subseteq \indeksy}(-1)^{|\II|}\pi(H_{3,   \II})
	&=\sum_{\{11,   21\}\subseteq\II\subseteq \indeksy}(-1)^{|\II|}\pi(H_{3,   \II},   \covers_{\Lambda}(H_{3,   \II})).
\end{align}

Let $mp\ge 1$ and $mp^2\le 1$ and recall that
$\covers_{\Lambda}(H_{3,   \II})$ consists of a finite number of indecomposable clique covers of $H_{3,   \II}$. 
Therefore,   using  Lemma~\ref{Lem:InDecomposable} (with partition $\{H_{3, \II, i}\}_{i\in \Lambda}$,  for which $F(H_{3,   \II},  \{H_{3, \II, i}\}_{i\in \Lambda})$,  introduced in Definition \ref{def:F},  is a forest),    Lemma~\ref{lem:Var2--r} (for $\pi(H_{3, \II, 0})$,   as $H_{3, \II, 0}\cong \klik_r\oplus_{j-1} \klik_r$) for $j\ge 3$,   and the estimate \eqref{eq:Var2},  we obtain for each $ \{11,   21\}\subseteq\II\subseteq \indeksy$
\begin{align*}
	n^{4r-j-1}\pi(H_{3,   \II},   \covers_{\Lambda}(H_{3,   \II}))&\lesssim \frac{n^{4r-j-1}\pi(H_{3, \II, 0})\qq^2}{(mp)^2}=  \frac{n^{2r-{(j-1)}}\pi(\klik_r\oplus_{j-1} \klik_r)(n^r\qq)^2}{({n}mp)^2}\\&\lesssim \frac1{nmp}\frac{(n^r\qq)^2}{\ble}\frac{(n^r\qq)^2}{nmp}\lesssim \frac{(\Var N)^2}{nmp}\le \frac{(\Var N)^2}{n}.
\end{align*}
This gives us the desired upper bound of \eqref{Eq:G3:Glowne} and concludes the proof in the case $mp^2\le 1$.

Now let $mp^2\ge 1$. 
Since  $\covers_{\Lambda}(H_{3,   \II})$ consists of a finite number of indecomposable clique covers of $H_{3,   \II}$,     Lemma~\ref{Lem:InDecomposable2bis} (with partition $\{\bar H_i\}_{i\in \Lambda}$) and \eqref{eq:Var2} give us  for each $ \{11,   21\}\subseteq\II\subseteq \indeksy$
\begin{align*}
	n^{4r-4}\pi(H_{3,   \II},   \covers_{\Lambda}(H_{3,   \II}))&\lesssim
	n^{4r-4}mp^3e^{-3mp^2}\\
	&=\frac{(n^{2r-2}e^{-mp^2})(n^{2r-1}mp^3e^{-2mp^2})}{n} \lesssim \frac{(\Var N)^2}{n}.
\end{align*}
In view of \eqref{Eq:G3:SumIndecomposible},  this concludes the proof of Lemma \ref{Lem:IntegralUpperGraphs} for $i=1$ and $j\ge 3$.

\section*{Acknowledgements}	

Grzegorz Serafin's work was supported by the National Science Centre (NCN), Poland,  under  research grant no. 2021/43/D/ST1/03244. The funder had no role in study design, data collection and analysis, decision to publish, or preparation of the manuscript.


\begin{thebibliography}{10}

\bibitem{BKR}
A.~D. Barbour, M.~Karo\'{n}ski, and A.~Ruci\'{n}ski.
\newblock A central limit theorem for decomposable random variables with
applications to random graphs.
\newblock {\em J. Combin. Theory Ser. B}, 47(2):125--145, 1989.

\bibitem{BCJ}
B.~B. Bhattacharya, A.~Chatterjee, and S.~Janson.
\newblock Fluctuations of subgraph counts in graphon based random graphs.
\newblock {\em Combin. Probab. Comput.}, 32(3):428--464, 2023.

\bibitem{Bloznelis2018}
M.~Bloznelis and J.~Jaworski.
\newblock The asymptotic normality of the global clustering coefficient in
sparse random intersection graphs.
\newblock In A.~Bonato, P.~Pra{\l}at, and A.~Raigorodskii, editors, {\em
	Algorithms and Models for the Web Graph. WAW 2018}, volume 10836 of {\em
	Lecture Notes in Computer Science}, pages 16--29. Springer International
Publishing, 2018.

\bibitem{Bloznelis2023}
M.~Bloznelis, J.~Karjalainen, and L.~Leskelä.
\newblock Normal and stable approximation to subgraph counts in superpositions
of bernoulli random graphs.
\newblock {\em Journal of Applied Probability}, 61(2):401--419, Aug. 2023.

\bibitem{BrennanBresterNagaraj2020Equivalence}
M.~Brennan, G.~Bresler, and D.~Nagaraj.
\newblock Phase transitions for detecting latent geometry in random graphs.
\newblock {\em Probability Theory and Related Fields}, 178(3–4):1215--1289,
Sept. 2020.

\bibitem{Dong2023}
L.~Dong and Z.~Hu.
\newblock The number of triangles in random intersection graphs.
\newblock {\em Communications in Mathematics and Statistics}, 11(4):695--725,
Jan. 2023.

\bibitem{FLSZ}
X.~Fang, S.-H. Liu, Q.-M. Shao, and Y.-K. Zhao.
\newblock Normal approximation for exponential random graphs.
\newblock {\em Probab. Theory Related Fields}, 2025.

\bibitem{FillScheinSinger2000Equivalence}
J.~A. Fill, E.~R. Scheinerman, and K.~B. Singer-Cohen.
\newblock Random intersection graphs when $m=\omega(n)$: An equivalence theorem
relating the evolution of the {G}(n, m, p) and {G}(n, p) models.
\newblock {\em Random Structures \& Algorithms}, 16:156--176, 2000.

\bibitem{KsiazkaJLR}
S.~Janson, T.~\L~uczak, and A.~Rucinski.
\newblock {\em Random graphs}.
\newblock Wiley-Interscience Series in Discrete Mathematics and Optimization.
Wiley-Interscience, New York, 2000.

\bibitem{1999KarSchSin}
M.~Karo\'{n}ski, E.~R. Scheinerman, and K.~B. Singer-Cohen.
\newblock On random intersection graphs: The subgraph problem.
\newblock {\em Combinatorics, Probability and Computing}, 8(1–2):131–159,
1999.

\bibitem{KR}
G.~Kaur and A.~R\"{o}llin.
\newblock Higher-order fluctuations in dense random graph models.
\newblock {\em Electron. J. Probab.}, 26:Paper No. 139, 36, 2021.

\bibitem{KimLeeNa2018Equivalence}
J.~H. Kim, S.~J. Lee, and J.~Na.
\newblock On the total variation distance between the binomial random graph and
the random intersection graph.
\newblock {\em Random Structures \& Algorithms}, 52(4):662--679, Jan. 2018.

\bibitem{LP1}
Q.~Liu and N.~Privault.
\newblock Normal to {P}oisson phase transition for subgraph counting in the
random-connection model.
\newblock {\em Electron. J. Probab.}, 30, 2025.

\bibitem{LP2}
Q.~Liu and N.~Privault.
\newblock Gaussian fluctuations of generalized {$U$}-statistics and subgraph
counting in the binomial random-connection model.
\newblock {\em Stochastic Process. Appl.}, 192:Paper No. 104825, 2026.

\bibitem{Maugis}
P.~A. Maugis.
\newblock Central limit theorems for local network statistics.
\newblock {\em Biometrika}, 111(3):743--754, 2024.

\bibitem{MNS}
W.~Michalczuk, M.~Nieradko, and G.~Serafin.
\newblock Normal approximation for subgraph count in random hypergraphs.
\newblock {\em arxiv:}, 2408.06112, 2024.

\bibitem{PSbej}
N.~Privault and G.~Serafin.
\newblock Normal approximation for sums of weighted
{$U$}-statistics---application to {K}olmogorov bounds in random subgraph
counting.
\newblock {\em Bernoulli}, 26(1):587--615, 2020.

\bibitem{PSstoch}
N.~Privault and G.~Serafin.
\newblock Normal approximation for generalized {$U$}-statistics and weighted
random graphs.
\newblock {\em Stochastics}, 94(3):432--458, 2022.

\bibitem{R}
A.~R\"{o}llin.
\newblock Kolmogorov bounds for the normal approximation of the number of
triangles in the {E}rd\H{o}s-{R}\'{e}nyi random graph.
\newblock {\em Probab. Engrg. Inform. Sci.}, 36(3):747--773, 2022.

\bibitem{Ruc}
A.~Ruci\'{n}ski.
\newblock When are small subgraphs of a random graph normally distributed?
\newblock {\em Probab. Theory Related Fields}, 78(1):1--10, 1988.

\bibitem{Rybarczyk2011Equivalence}
K.~Rybarczyk.
\newblock Equivalence of the random intersection graph and {G}(n,p).
\newblock {\em Random Structures \& Algorithms}, 38:205--234, 2011.

\bibitem{Edges}
K.~Rybarczyk and G.~Serafin.
\newblock Normal approximation for subgraph counts\\ in random intersection
graphs: the case of edges.
\newblock {\em arxiv:}, 2503.23110, 2025.

\bibitem{SubgraphProb}
K.~Rybarczyk and G.~Serafin.
\newblock Subgraph probabilities in random intersection graphs.
\newblock {\em arxiv:}, 2026.

\bibitem{2010RybSta}
K.~Rybarczyk and D.~Stark.
\newblock Poisson approximation of the number of cliques in a random
intersection graph.
\newblock {\em Journal of Applied Probability}, 47(3):826--840, 2010.

\bibitem{2017RybSta}
K.~Rybarczyk and D.~Stark.
\newblock Poisson approximation of counts of induced subgraphs in random
intersection graphs.
\newblock {\em Discrete Mathematics}, 340(9):2183--2193, 2017.

\end{thebibliography}
\end{document}